\pdfoutput=1
\include{2c7m.tex} %

\documentclass[reqno,a4paper,12pt]{amsart}
\usepackage{amsmath}
\usepackage{amsthm}
\usepackage{amssymb}
\usepackage[utf8]{inputenc}
\usepackage{a4wide}
\usepackage{enumitem}
\usepackage{pstricks}
\usepackage{ifthen}
\usepackage{comment}
\usepackage{slashed}
\usepackage{xspace}
\usepackage[all]{xy}
\usepackage[pdftitle={The classification of 2-connected 7-manifolds},hidelinks]{hyperref}

\RequirePackage{booktabs}

\usepackage{multicol}
\usepackage{microtype}

\usepackage{graphics}
\providecommand{\texorpdfstring}[2]{#1}

\title{The classification of $2$-connected $7$-manifolds}
\usepackage{color}

\DeclareMathOperator{\aut}{Aut}
\DeclareMathOperator{\Hom}{Hom}
\newcommand{\shearb}{\textup{Shr}_p}
\newcommand{\shearq}{\textup{Shr}_{p/2}}

\DeclareMathOperator{\ind}{ind}

\DeclareMathOperator{\im}{Im}

\DeclareMathOperator{\coker}{coker}

\DeclareMathOperator{\lcm}{lcm}
\DeclareMathOperator{\ord}{ord}

\newcommand{\ie}{\emph{i.e.} }
\newcommand{\eg}{\emph{e.g.} }
\newcommand{\cf}{\emph{cf.} }
\newcommand{\sign}{\sigma}

\newcommand{\divs}{S}
\newcommand{\divsp}{S_{d_\pi}}
\newcommand{\dv}{c}
\newcommand{\tdv}{{\tilde \dv}\mkern1mu}
\newcommand{\divsdv}{S_\dv}

\newcommand{\EK}{\mu}
\newcommand{\gr}{g}
\newcommand{\lf}{g}
\newcommand{\ares}{j}

\newcommand{\catb}{\mathfrak{B}}
\newcommand{\catq}{\mathfrak{R}}
\newcommand{\catek}{\mathfrak{D}}

\newcommand{\ahat}{\widehat{A}}

\newcommand{\msub}{o}
\newcommand{\mdiv}{d_\msub}

\newcommand{\mpl}{m}
\newcommand{\pio}{u}

\newcommand{\an}[1]{\langle #1 \rangle}
\newcommand{\anq}[2]{\left\langle \! \left\langle \frac{#1}{#2} \right\rangle \! \right\rangle}
\newcommand{\biganq}[2]{\Big\langle \!\! \Big\langle \frac{#1}{#2} \Big\rangle \!\! \Big\rangle}
\newcommand{\anb}[2]{\left\langle \frac{#1}{#2} \right\rangle}

\newcommand{\mmod}{\!\!\mod}

\newcommand{\ADiff}{\textup{ADiff}}
\newcommand{\cal}[1]{\mathcal{#1}}
\newcommand{\wh}{\widehat}
\newcommand{\Arf}{A}
\newcommand{\proj}{\textup{Proj}}
\newcommand{\sect}{\textup{Sec}}
\newcommand{\splt}{\textup{Split}}
\newcommand{\Asplt}{\textup{ASplit}}
\newcommand{\acong}{\approxeq}
\newcommand{\dM}{\wh d_\pi}
\newcommand{\ol}{\overline}

\newcommand{\quart}{\frac{1}{4}}

\newcommand{\bbz}{\mathbb{Z}}
\newcommand{\Z}{\mathbb{Z}}
\newcommand{\Q}{\mathbb{Q}}

\newcommand{\C}{\mathbb{C}}

\newcommand{\into}{\hookrightarrow}
\newcommand{\xra}{\xrightarrow}
\newcommand{\Int}{\textup{Int}}
\newcommand{\Diff}{\textup{Diff}}
\newcommand{\NumB}[1]{\textup{Num}\left(#1\right)}
\newcommand{\Num}{\textup{Num}}

\newcommand{\Id}{\textup{Id}}
\newcommand{\del}{\partial}

\newcommand{\gtstr}{$G_{2}$\nobreakdash-\hspace{0pt}structure}

\newcommand{\gtmfd}{$G_{2}$\nobreakdash-\hspace{0pt}manifold}

\newcommand{\wt}{\widetilde}

\newcommand{\td}{\tilde d}
\newcommand{\tdelta}{\widetilde \Delta}
\newcommand{\tp}{\widetilde P}
\newcommand{\tdf}{\textstyle \frac{\td_\pi}{4}} 
\newcommand{\tdvf}{\textstyle \frac{\tdv}{4}} 
\newcommand{\hlambda}{\wh \lambda}
\newcommand{\rad}{K}
\newcommand{\res}{R}

\newcommand{\gen}[1]{\langle#1\rangle}

\newcommand{\dirac}{\slashed{D}}
\newcommand{\pd}{\mkern2mu\widecheck{\mkern-2mu p}}
\newcommand{\pc}{\wh p}
\newcommand{\spc}{$\textrm{spin}^{c}$\xspace}
\newcommand{\spcg}{\textup{Spin}^\textrm{c}}
\newcommand{\spg}{\textup{Spin}}
\newcommand{\rif}{\bar \lambda}
\newcommand{\twotwo}[4]{\left( \begin{array}{cc} #1 & #2 \\ #3 & #4 \end{array} \right)}

\makeatletter
\DeclareRobustCommand\widecheck[1]{{\mathpalette\@widecheck{#1}}}
\def\@widecheck#1#2{%
    \setbox\z@\hbox{\m@th$#1#2$}%
    \setbox\tw@\hbox{\m@th$#1%
       \widehat{%
          \vrule\@width\z@\@height\ht\z@
          \vrule\@height\z@\@width\wd\z@}$}%
    \dp\tw@-\ht\z@
    \@tempdima\ht\z@ \advance\@tempdima2\ht\tw@ \divide\@tempdima\thr@@
    \setbox\tw@\hbox{%
       \raise\@tempdima\hbox{\scalebox{0.9}[-0.9]{\lower\@tempdima\box
\tw@}}}%
    {\ooalign{\box\tw@ \cr \box\z@}}}
\makeatother

\newtheorem{thm}{Theorem}[section]
\newtheorem{prop}[thm]{Proposition}
\newtheorem{lem}[thm]{Lemma}

\newtheorem{cor}[thm]{Corollary}

\theoremstyle{definition}
\newtheorem{defn}[thm]{Definition}

\theoremstyle{remark}
\newtheorem{rmk}[thm]{Remark}
\newtheorem*{rmk*}{Remark}
\newtheorem{ex}[thm]{Example}

\setlist{leftmargin=*}
\begin{document}

\begin{abstract}
We present a
comprehensive classification of
closed smooth $2$-connected mani\-folds of dimension 7. 
This builds on the almost-smooth classification from
the first author's thesis. The main new ingredient is a generalisation of
the Eells--Kuiper invariant that is defined for any closed spin
\mbox{$7$-manifold} $M$, regardless of whether the spin characteristic class
$p_M \in H^4(M)$ is torsion. 

We also determine the inertia group of $2$-connected~$M$---equivalently the
number of oriented smooth structures on the underlying topological
manifold---in terms of $p_M$ and the torsion linking form.
\end{abstract}

\maketitle

\section{Introduction} \label{sec:introduction}

Throughout this paper $M$ will be a closed smooth spin $7$-manifold and 
all homeo\-morphisms and diffeomorphisms are assumed to preserve spin structures,
unless stated otherwise.

\subsection{Background} \label{ss:bg}

Wall classified $(s{-}1)$-connected $(2s{+}1)$-manifolds up to connected sum with homotopy spheres
except when $s = 1, 2, 3$ or  $7$
\cite[Theorem 7]{wall67}.
In this paper, we leave connected $3$-manifolds aside,
recall that Barden classified $1$-connected $5$-manifolds
\cite{barden65}
and focus on $2$-connected $7$-manifolds
(leaving dimension $15$ to Remark \ref{rmk:15} below).

The topologically simplest $7$-manifolds are homotopy $7$-spheres,
whose 
spin (equivalently oriented) 
diffeomorphism classes form the group
$ \Theta_7$. Kervaire and Milnor \cite{kervaire63} computed that
$\Theta_7 \cong \Z/28$.
Eells and Kuiper \cite{eells62}
defined an invariant $\mu(M)$ of certain spin $7$-manifolds $M$
with rationally trivial first Pontrjagin class, which distinguishes all homotopy $7$-spheres.

At this point, it made sense
to study $7$-manifolds up to almost diffeomorphism;
\ie\!\!, up to the action of $\Theta_7$ via connected sum.
Wilkens did this in his PhD \cite{wilkens71},
using the triple of invariants (see Section \ref{ss:basic_invariants})
$$ (H^4(M), b_M, p_M), $$
where $H^4(M)$ is the integral cohomology group,
 $b_M \colon TH^4(M) \times TH^4(M) \to \Q/\Z$ is the torsion linking form %
and $p_M \in 2H^4(M)$ is the spin characteristic class of $M$.
We call this triple the {\em base} of $M$.  
Modulo a %
finite ambiguity if $|TH^4(M)|$ is even,
Wilkens proved that the base %
classifies $2$-connected $M$ up to almost diffeomorphism.
When $M$ is the total space of an $S^3$-bundle over $S^4$,
this ambiguity was resolved by the first author and~Escher~\cite{crowley03}.

The first author completed the almost diffeomorphism classification 
of $2$-connected $M$ in his PhD \cite{crowley02}
by defining a quadratic refinement $q_M$ of the torsion linking form $b_M$ when $H^4(M)$ is torsion,
and a family of such refinements in general.
When $H^4(M) = TH^4(M)$ is torsion, 
\cite{crowley02} also proves that
the triple $(TH^4(M), q_M, \mu(M))$ gives a complete diffeomorphism invariant.
This left the smooth classification open when $H^4(M)$ is infinite: 
the difficulty being that the classical Eells-Kuiper invariant is not defined
when $p_M  \in H^4(M)$ has infinite order.

In this paper we present a comprehensive smooth classification of
closed $2$-connected \mbox{$7$-manifolds} %
by defining a generalisation of the Eells-Kuiper invariant for all spin $7$-manifolds.
The uniqueness part of this classification has also been proven by
Kreck \cite[Theorem~1]{kreck18}.
Our main classification results are stated in Sections \ref{ss:classification}
and \ref{ss:elaboration} and
Section \ref{ss:intro_eek} describes the definition of
the Generalised Eells-Kuiper invariant.
We give applications of the classi\-fication results to 
the inertia groups and mapping class groups of 
$7$-manifolds in Section \ref{ss:i_and_r} and
we continue the discussion of the background in 
Sections \ref{ss:surgery} and \ref{ss:7_and_15}.

An important motivation for this paper is the study of Riemannian manifolds
with holonomy the exceptional Lie group $G_2$: such manifolds always have $p_M$
of infinite order (see Joyce \cite[Proposition 10.2.7]{joyce00}).
In \cite{exotic} we use the generalized Eells-Kuiper invariant to distinguish
pairs of closed $G_2$-manifolds which are homeomorphic but not diffeomorphic.
\subsection{The classification} \label{ss:classification}
To any closed smooth spin 7-manifold $M$ we shall associate the following
algebraic invariants: 
\begin{itemize}
\item The integral cohomology group $H^4(M)$, which is a finitely generated abelian
group;
\item The torsion linking form $b_M : TH^4(M) \times TH^4(M) \to \Q/\Z$, which
is a torsion form on the torsion subgroup $TH^4(M) \subseteq H^4(M)$,
by this we mean that $b_M$ is symmetric, bilinear and nonsingular
(see \eqref{eq:b-by-cob} and Lemma \ref{lem:bboM});
\item The spin characteristic class $p_M$,
which is an even element of $H^4(M)$
(see Lemma \ref{lem:p_M}(i)).
It is a homeo\-morphism invariant by Remark \ref{rmk:topological_p_M},
and it is related to the first Pontrjagin class by ${2p_M = p_1(M)}$;
\item The \emph{quadratic linking family} $q^\circ_M$
(see Definition \ref{def:qlf}), 
which is a family of quadratic refinements of the base $(H^4(M), b_M, p_M)$;
\item The \emph{generalised Eells-Kuiper invariant} $\EK_M$ 
(see Definition \ref{def:GEK}),
which is a mod 28 Gauss refinement of the triple $(H^4(M), q_M^\circ, p_M)$.
\end{itemize}
Let us describe $q^\circ_M$ and then indicate the type of $\EK_M$
(leaving a more detailed introduction of $\EK_M$ to \S\ref{ss:intro_eek}). 

\enlargethispage{\baselineskip}
A quadratic refinement of $b_M$ is a function $q \colon TH^4(M) \to \Q/\Z$ 
which satisfies the equation $q(x{+}y) = q(x) + q(y) + b_M(x, y)$ 
and we denote the set of such $q$ by
$\cal{Q}(b_M)$. 
We note that for $t \in TH^4(M)$, the function $q_t(x) := q(x) + b(x, t)$
also belongs to $\cal{Q}(b_M)$.
The homogeneity defect of $q \in \cal{Q}(b_M)$ is the unique
element $\beta \in 2TH^4(M)$ such that $q(x) - q(-x) = b_M(x, \beta)$.
Let
\[ \divs_2 := \{ h \in H^4(M) : p_M - 2h \textrm{ is torsion} \} . \]
That $q^\circ_M$ is a \emph{family of quadratic refinements} of $(H^4(M), b_M, p_M)$
means that it is a function
\[ q_M^\circ \colon \divs_2 \to \cal{Q}(b_M), \quad h \mapsto q_M^h, \]
such that 
$q_M^{h+t} = (q_M^h)_{-t}$ for all $t \in TH^4(M)$
and $q_M^h$ has homogeneity defect $\beta_h := p_M - 2h$.
The family of quadratic refinements $q^\circ_M$ is defined in
Definition \ref{def:qlf}.

Let $d_\pi$ be the greatest integer dividing $p_M$ modulo torsion
(or $d_\pi := 0$ if $p_M$ is torsion),
$\td_\pi : = \lcm \left(4, d_\pi \right)$ and
$\dM : = \gcd \bigl( \frac{\td_\pi}{4}, 28 \bigr)$.
If $d_\pi > 0$ we set
\[ \divsp := \{ k \in H^4(M) : p_M - d_\pi k \textrm{ is torsion} \}, \]
and if $d_\pi = 0$ set $\divsp := TH^4(M)$.
We set $\beta_k := p_M - d_\pi k$ for each $k \in S_{d_\pi}$
and note that for $e_\pi := d_\pi/2$ we have $e_\pi k \in S_2$.
By saying that the generalised Eells-Kuiper invariant of $M$ is a
\emph{mod 28 Gauss refinement} of $(H^4(M), q_M^\circ, p_M)$ 
we mean (see Definition \ref{def:modngr}) that it is a function
\[ \EK_M \colon \divsp \to \Q/\dM\Z,  \]
such that  $\EK_M(k) = \Arf(q^{e_\pi k}_M) \mmod \Z$
(where $\Arf$ is the Arf invariant of a quadratic refinement, 
computed in terms of a Gauss sum in \eqref{eq:arf_def}), and such
that the following transformation rule
\begin{equation} \label{eq:TR}
\EK_M(k+t) - \EK_M(k) \; = \; e_\pi q^{e_\pi k}_M(t) - 
{ \textstyle {e_\pi + 1 \choose 2}} \, b_M(t,t) \mod \dM
\end{equation}
holds for all $k \in \divsp$ and $t \in TH^4(M)$ (note that both terms on the
RHS have coefficient divisible by $\frac{\td_\pi}{4}$, so are in particular
well-defined in $\Q/\dM\Z$).
The Generalised Eells-Kuiper invariant of $M$ is defined
in Definition \ref{def:GEK}.

Two of the main consequences of \eqref{eq:TR} are that a Gauss
refinement is defined by its value at a single element in $\divsp$, and that
the difference between two Gauss refinements of
$(H^4(M), q_M^\circ, p_M)$ is constant. The constraint in terms of the Arf
invariant then forces this constant to take values in $\Z/\dM\Z$. 

\begin{rmk} \label{rmk:original_ek}
If $p_M$ is a torsion element then $d_\pi = 0$ and $\dM = 28$, while
$\divsp = TH^4(M)$ contains the distinguished element $0$.
The value $\frac{1}{28}\EK_M(0) \in \Q/\Z$ recovers the original
Eells-Kuiper invariant.
See Remark \ref{rmk:maxinfo} for related statements
even when $p_M$ is not torsion.
\end{rmk}

If $G$ is a finitely generated abelian group, $p \in 2G$ and $b$ is a torsion
form on $T \subseteq G$, the torsion subgroup of $G$, 
then we call $(G,b,p)$ a \emph{base}. If $q^\circ$ is a family of
quadratic refinements of $(G,b,p)$ then we call $(G,q^\circ,p)$ a \emph{refinement};
we suppress $b$ since it can be recovered from $q^h$ for any $h$
and hence from $q^\circ$. If $\EK$ is
a mod 28 Gauss refinement of $(G,q^\circ\!,p)$, then we call the quadruple
$(G,q^\circ\!,\EK,p)$ a \emph{mod 28 distillation}.
If $F : G' \to G$ is a group isomorphism then we can
define another mod 28 distillation $(G', F^\#q, F^\#\EK, F^\#p)$ by pulling back:
$F^\#(p) := F^{-1}(p)$, $(F^\#q)^h(x) := q^{F(h)}(F(x))$,
and $F^\#\EK := \EK \circ F$.

The mod $28$ distillation 
$(H^4(M), q_M^\circ, \EK_M, p_M)$
of $M$ 
is an invariant of 
diffeomorphisms: if $f \colon M \to M'$ is a 
diffeomorphism
then $f^* \colon H^4(M') \to H^4(M)$ is an isomorphism and
$(q^\circ_{M'}, \EK_{M'}, p_{M'}) = ((f^*)^\#q^\circ_M, (f^*)^\#\EK_M, (f^*)^\#p_M)$.
In fact, only $\EK_M$ depends on the smooth structure and the refinement
$(H^4(M), q_M^\circ, p_M)$ is also invariant under spin homeomorphisms.

An almost diffeomorphism $f \colon M_0 \acong M_1$ is a homeomorphism which is 
smooth except perhaps at a finite number of points.
It follows from results of the first author's thesis, see Lemma \ref{lem:2_types_of_qlf}, 
that \emph{2-connected} 7-manifolds
are classified up to almost diffeomorphism and homeomorphism by their refinements.

\begin{thm}[Almost diffeomorphism and homeomorphism classification]
\label{thm:a_class}
Every refinement $(G, q^\circ\!, p)$ is isomorphic to
$(H^4(M), q_M^\circ, p_M)$ for some
$2$-connected $7$-manifold $M$.
Moreover, if $M_0$ and $M_1$ are $2$-connected, then 
an isomorphism
$F \colon H^4(M_1) \to H^4(M_0)$ is realised as $f^*$ for some almost
diffeomorphism $f \colon M_0 \acong M_1$
if and only if $(q_{M_1}^\circ, p_{M_1}) = F^\#(q_{M_0}^\circ, p_{M_0})$.

The same statement holds with ``almost diffeomorphism''
replaced by ``homeo\-morphism''.
\end{thm}

%
%
%
%
%
%
%


The central result of this paper is that the generalised Eells-Kuiper invariant
is precisely what needs to be added to Theorem \ref{thm:a_class} to obtain a
smooth classification of 2-connected 7-manifolds.
Consequently, 2-connected 7-manifolds
are classified up to diffeomorphism by their mod $28$ distillations.

\begin{thm}[Smooth classification] \label{thm:class}
Every mod 28 distillation $(G, q^\circ\!, \EK, p)$ is isomorphic to
$(H^4(M), q_M^\circ, \EK_M, p_M)$ for some 
$2$-connected $7$-manifold $M$.
Moreover, if $M_0$ and $M_1$ are $2$-connected,
then an isomorphism $F \colon H^4(M_1) \to H^4(M_0)$
is realised as $f^*$ for some diffeomorphism ${f \colon M_0 \cong M_1}$
if and only if
$(q^\circ_{M_1}, \EK_{M_1}, p_{M_1}) = F^\#(q^\circ_{M_0}, \EK_{M_0}, p_{M_0})$.
\end{thm}

\subsection{Elaboration of the classification} \label{ss:elaboration}

Theorem \ref{thm:class} is a ``polarised'' classification result in the sense
that it identifies whether a given isomorphism of the cohomology is realised
by some diffeomorphism.
If we are simply interested in whether $M_0$ and $M_1$ are diffeomorphic
(without specifying how the diffeomorphism acts on cohomology)
then we can consider a coarser invariant 
than the generalised Eells-Kuiper invariant.
Let $\aut(b_M)$ be 
the group of automorphisms of the linking form $b_M$ and recall that 
for each $k \in S_{d_\pi}$ we have
${\beta_k = p_M - d_\pi k \in TH^4(M)}$,
which is the homogeneity defect of the quadratic refinement $q_M^{e_\pi k}$.
We define the {\em smooth splitting set} of $M$ to be the set
\[ \bar{\cal{Q}}(M) := \{ \bigl( [\beta_k], \mu_M(k) \bigr) : k \in S_{d_\pi} \}
\subset \bigl(2 TH^4(M)/\!\aut(b) \bigr) \times \Q/\dM \Z. \]
An isomorphism 
$F \colon (H^4(M_1), b_{M_1}, p_{M_1}) \cong (H^4(M_0), b_{M_0}, p_{M_0})$
induces the map
\[ F^\# 
\colon \bigl(2 TH^4(M_0)/\!\aut(b_0) \bigr) \to \bigl(2 TH^4(M_1)/\!\aut(b_1) \bigr),
\quad [\beta] \mapsto [F^{-1}(\beta)]. \]
The following theorem generalises
\cite[Theorem 1.5]{crowley03} from the case when $M$ is the total space
of a smooth $S^3$-bundle over $S^4$ to all $2$-connected $M$.

\begin{thm}[Unpolarised smooth classification] \label{thm:classification_minimal}
Let $M_0$ and $M_1$ be $2$-connected, and let
$F \colon \! (H^4(M_1), b_{M_1}, p_{M_1})\! \to\! (H^4(M_0), b_{M_0}, p_{M_0})$
be an isomorphism.
Then the following are equivalent:
\begin{enumerate}
\item $M_0$ is diffeomorphic to $M_1$;
\item
$(F^\# \times \Id) \bigl( \bar{\cal{Q}}(M_0) \bigr) = \bar{\cal{Q}}(M_1)$;
\item
$(F^\# \times \Id) \bigl( \bar{\cal{Q}}(M_0) \bigr) \cap \;\! \bar{\cal{Q}}(M_1) \neq \emptyset$.
\end{enumerate}
\end{thm}
\noindent The corresponding result for almost diffeomorphisms is given in
Corollary~\ref{cor:minima_almost_classification}.

We now formulate the polarised classification of Theorem \ref{thm:class}
in categorical language, giving more information about the monoidal structure
of $2$-connected $7$-manifolds under connected sum.
Let $\catek$ denote the category of mod 28 distillations
$(G, q^\circ\!, \EK, p)$ with morphisms isomorphisms:
\[ {\rm Ob}(\catek) = \{ (G, q^\circ\!, \EK, p) \} \]
Let $\cal{M}_{7,2}^{\spg}$ denote the category of 2-connected spin 7-manifolds with 
morphisms diffeomorphisms:
\[ {\rm Ob}(\cal{M}_{7, 2}^{\spg}) = \{ M : \pi_1(M) = 0 = \pi_2(M) \} \]
Given a diffeomorphism $f \colon M_0 \cong M_1$, write $f^* \colon H^4(M_1) \cong H^4(M_0)$ for the induced action on cohomology.
Hence we obtain the contravariant functor
\[ \cal{D} \colon \cal{M}_{7, 2}^{\spg} \to \catek,  \quad 
\left\{ \begin{array}{ccc} M & \mapsto & (H^4(M), q_M^\circ, \EK_M, p_M), \\
f \colon M_0 \cong M_1 & \mapsto & f^*. \end{array} \right. \]
The operations of connected sum and reversing orientation in $\cal{M}_{7,2}^{\spg}$
are mirrored by corresponding operations in $\catek$.
For $i = 0,1 $ the orthogonal sum of two distillations
$(G_i, q^\circ_i, \EK_i, p_i)$ is defined as follows. 
Noting that 
$d_{\pi_i} = c_i d_{\pi_0 \oplus \pi_1}$ for some integer $c_i$,
in which case $c_0S_{d_{\pi_0}} \times c_1S_{d_{\pi_1}} \subseteq S_{d_{\pi_0 \oplus \pi_1}}$,
we define the orthogonal sum $q_0 \oplus q_1$ at $c_0 k_0 + c_1 k_1$ by
\[ (q_0 \oplus q_1)^{c_0 k_0 + c_1 k_1} : = q_0^{k_0} \oplus q_1^{k_1}, \]
and
\[ (\EK_0 \oplus \EK_1)(c_0 k_0 + c_1 k_1) : = \EK_0(k_0) + \EK_1(k_1) \mod 
\gcd \left( 28, \textstyle \frac{\tilde d_{\pi_0 \oplus \pi_1}}{4} \right). \]
Since $q_0^\circ \oplus q_1^\circ$ and
$\EK_0 \oplus \EK_1$ are determined by their values on a single $k \in S_{d_{\pi_0 \oplus \pi_1}}$,
this suffices to define the sum of distillations
and the transformations laws for refinements and distillations ensure that the orthogonal 
sum is well-defined.
We define the negative of a distillation by
\[  -(G, q^\circ\!, \EK, p) := (G, -q^\circ\!, -\EK, p).  \]

\begin{thm}[Categorical version of smooth classification] 
\label{thm:classification_categorical}
The functor $\cal{D} \colon \cal{M}_{7, 2}^{\spg} \to \catek$ is surjective
and faithful.  Moreover
\begin{enumerate}
\item $\cal{D}(M_0 \sharp M_1) = \cal{D}(M_0) \oplus \cal{D}(M_1)$ and
\item $\cal{D}(-M) = -\cal{D}(M)$.
\end{enumerate}
\end{thm}

We next present an oriented homotopy 
classification for $2$-connected $M$.
Such a classification was given in \cite[Theorem 6.11]{crowley02}
and we re-formulate that classification in the setting of this paper.
An important feature of the homotopy classification is that 
$p_M \in H^4(M)$ is not a homotopy invariant but
$\rho_{24}(p_M) \in H^4(M; \Z/24)$, the mod~$24$-reduction of $p_M$,
is a homotopy invariant, \cite[Theorem 1]{milgram87}.
As a consequence, there is a precise sense 
in which the homotopy classification is
the ``mod~$24$ reduction'' of the homeomorphism classification.

For a linking form $(b, T)$ 
define $J\mathcal{Q}_{}(b)$ to be the quotient of $\mathcal{Q}(b)$
where we identify two refinements $q_0$ and $q_1$ if 
$q_1 = (q_0)_{12t}$ for some $t \in T$
and write $\rho_{12} \colon \mathcal{Q}(b) \to J\mathcal{Q}_{}(b)$ for the quotient map.
A {\em $J$-quadratic refinement} of a base $(G, b, p)$ is a triple 
$(G, Jq^\circ, \rho_{24}(p))$
where $Jq^{\circ} \colon S_2 \to J\mathcal{Q}_{}(b)$ 
is a map such that $Jq^{h+t} = (Jq^h)_{-t}$ 
and $\rho_{24}(\beta_h) = \rho_{24}(p - 2h) \in T \otimes \Z/24$.
The pull-back of $J$-refinements is defined analogously to
the pull-back of refinements
and the $J$-refinement of $M$ is defined to be the triple
$(H^4(M), \rho_{12} \circ q^\circ_M, \rho_{24}(p_M))$.

\begin{thm}[Homotopy classification] \label{thm:homotopy}
Every $J$-refinement $(G, Jq^\circ, \rho_{24}(p))$ 
is isomorphic to
$(H^4(M), \rho_{12} \circ q_M^\circ, \rho_{24}(p_M))$ 
for some
smooth $2$-connected $7$-manifold $M$.
Moreover, if $M_0$ and $M_1$ are $2$-connected, then an isomorphism
$F \colon H^4(M_1) \to H^4(M_0)$ is realised as $f^*$ for some orientation preserving
homotopy equivalence $f \colon M_0 \simeq M_1$ %
if and only if $(\rho_{12} \circ q_{M_1}^\circ, \rho_{24}(p_{M_1})) = 
F^\#(\rho_{12} \circ q_{M_0}^\circ, \rho_{24}(p_{M_0}))$.
\end{thm}

\subsection{The generalised Eells-Kuiper invariant} \label{ss:intro_eek}
As explained in Section \ref{ss:bg}, the main novelty of Theorem
\ref{thm:class} lies in the smooth classification when $H^4(M)$ is infinite.
The key ingredient is the generalisation of the Eells-Kuiper invariant.

Let $X$ be a closed spin $8$-manifold.  By the index theorem \cite[Theorem 5.3]{atiyah68}
$\wh A(X)$, the $\wh A$-genus of~$X$, is equal to the index of the Dirac
operator on $X$, and so is an integer.  
The classical Eells-Kuiper invariant is derived from the relation
\begin{equation}
\label{eq:ahat}
p_X^2 - \sign(X) = 224\ahat(X),
\end{equation}
where $X$ has signature $\sign(X)$ and spin characteristic class $p_X$:
the latter is 
defined in Section~\ref{ss:basic_invariants}.
If $M$ is a closed 7-manifold such that $p_M$ is a \emph{torsion} class
(so rationally trivial) and $W$ is a spin coboundary of $M$, then 
$p_W^2$
has a well-defined integral over $W$ (it might in general take values in $\Q$
and not just $\Z$), 
and \eqref{eq:ahat} implies that
\begin{equation}
\label{eq:classical}
\EK(M):= \frac{p_W^2 - \sign(W)}{8} \in \Q/28\Z
\end{equation}
is independent of the choice of coboundary $W$. This defines the classical
Eells-Kuiper invariant, modulo normalisation by a factor of 28.

To define an analogue when $p_M$ is not a torsion class we have to let it take
values not modulo 28 but modulo the integer $\dM = \gcd(\frac{\tilde d_\pi}{4}, 28)$, depending on the
divisibility of $p_M$ modulo torsion as above. Moreover, the
generalisation is not simply a constant in $\Q/\dM\Z$ but a function.

To define the generalised Eells-Kuiper invariant $\EK_M$, suppose that $W$ is a
spin coboundary of $M$ and that there exists
$n \in H^4(W)$ such that the image of $p_W - d_\pi n$ under the restriction
map $j \colon H^4(W) \to H^4(M)$ is a torsion class;
equivalently $j(n) \in \divsp$.
If $W$ is 3-connected then such $n$ exist and any spin $M$ has 3-connected
coboundaries: see the start of Section~\ref{ss:coboundaries}.
Since $j(p_W - d_\pi n)$ is torsion we can define (\cf \eqref{eq:g_W-explicit})
\begin{equation} \label{eq:grW}
 \gr_W \bigl( j(n) \bigr) := \frac{(p_W {-} d_\pi n)^2 - \sign(W)}{8} \in \Q/\tdf \Z 
\end{equation}
and then extend $\gr_W$ to a function $\divsp \to \Q/\tdf \Z$ by the transformation
rule \eqref{eq:TR}. Then $\gr_W$ is independent of the choices of $n$.
The following lemma (\cf \eqref{eq:compare}) implies that the residue
$\EK_M := \gr_W \mmod \dM$ is independent of the choice of $W$, and functorial.

\begin{lem}
\label{lem:ek_defined}
Let $f : \partial W_0 \to \partial W_1$ be a 
diffeomorphism and
$X := (-W_0) \cup_f W_1$. Then
\[ \gr_{W_1} - (f^*)^\# \gr_{W_0} = 28 \ahat(X) \mmod \tdf. \]
\end{lem}

\begin{defn} \label{def:GEK}
The {\em generalised Eells-Kuiper invariant} of $M$
is defined to be the function $\mu_M \colon \divsp \to \Q/\dM \Z$.
\end{defn}

The idea of the definition is that the simplest way to change
\eqref{eq:classical} to something that is well-defined when the restriction of
$p_W$ to the boundary is rationally non-trivial is to compensate by subtracting
from $p_W$ a class that is divisible by $d_\pi$ and has the same rational image
in $H^4(M)$. The essentially different ways of doing that are parametrised by
$\divsp$ and that is why 
the generalised Eells-Kuiper invariant
is a function defined on $\divsp$.

The definition of the $s_1$ invariant by Kreck and Stolz \cite{kreck91}
provides as a byproduct a way to compute the classical Eells-Kuiper invariant
in terms of coboundaries that are not spin, but merely \spc. Proposition
\ref{prop:ek_via_spc} gives a similar way to compute the generalised Eells-Kuiper
invariant via \spc coboundaries.
We use this method in \cite{exotic} to 
compute the generalised Eells-Kuiper invariants of 
certain closed 7-manifolds with holonomy $G_2$ 
that are homeomorphic but not diffeomorphic.
\subsection{Inertia and reactivity} \label{ss:i_and_r}
Let $\Theta_7 = \{ \Sigma :  \Sigma \simeq S^7 \}$ be the group of spin
diffeomorphism classes of homotopy $7$-spheres $\Sigma$.  
This is equivalent to the standard definition of $\Theta_7$
in~\cite{kervaire63},
since homotopy spheres are simply connected.
By \cite{kervaire63}, $\Theta_7$
is an abelian group under connected sum and $\Theta_7 \cong \Z/28$.
The {\em inertia group} of $M$ is defined to be the following subgroup of $\Theta_7$:
\[ I(M) : = \{ \Sigma : M \sharp \Sigma \cong M \}\]

\begin{rmk} \label{rmk:inertia}
Let $M_+$ denote the oriented manifold underlying $M$.  If $M$ is simply connected then 
$I(M) = I(M_+)$, 
where $I(M_+)$ is the usual inertia group of $M_+$, which is 
defined using orientation preserving diffeomorphisms $f_+ \colon M_+ \sharp \Sigma_+ \cong M_+$. 
\end{rmk}

It turns out that even with Theorem \ref{thm:class} in
hand, the determination of $I(M)$ can be a delicate problem.
The reason is that $\EK_M$ is not a constant but rather a function and so it is
possible for almost diffeomorphisms of $M$ to act non-trivially on $\EK_M$.
Equivalently, the automorphism group of a refinement $(G,q^\circ\!,p)$ can act
non-trivially on the set of mod 28 Gauss refinements.

The inertia group is closely related to what we (therefore) call the
\emph{reactivity} of $M$. %
Let $\ADiff_{}(M)$ denote the group of spin almost diffeomorphisms
of $M$.  Given 
$f \in \ADiff_{}(M)$,
the mapping torus $T_f$ of $f$ has
as well-defined spin characteristic class $p_{T_f} \in H^4(T_f)$ and we define the integer
$p^2(f) : = \an{p_{T_f}^2, [T_f]}$.  This defines a homomorphism 
\[ p^2 \colon \ADiff_{}(M) \to \Z, \quad f \mapsto p^2(f),\]
and the reactivity of $M$ is the non-negative integer $R(M)$ defined by
\begin{equation}
\label{eq:defR}
p^2(\ADiff_{}(M)) = R(M) \Z .
\end{equation}
Clearly $R(M)$ is an almost diffeomorphism invariant of $M$.  Since $T_f$ has zero signature
and $p_{T_f}$ is characteristic for the intersection form of $T_f$ we have $R(M) \in 8 \Z$.
It is well understood that $f \in \ADiff_{}(M)$ is pseudo-isotopic to a diffeomorphism if and only if
$p^2(f)$ is divisible by 224 (see Lemma \ref{lem:mapping_torus}) and consequently
\begin{equation} \label{eq:i_and_r_intro}
I^{}(M) = \frac{R(M)}{8} \Theta_7.
\end{equation}

To determine $R(M)$ when $M$ is 2-connected, we first determine the values of
$p^2(f)$ which
are realised when $H^*(f) = {\rm Id}$ (see Proposition \ref{prop:p2null}).
This reduces the determination of $R(M)$ to understanding the action of
the automorphism group $\aut_{q^\circ}(H^4(M))$ of $(H^4(M), q_M^\circ, p_M)$ on
mod 28 Gauss refinements. That can in turn be reduced to understanding
the automorphism group $\aut_b(H^4(M))$ of $(H^4(M), b_M, p_M)$, which is much
easier to deal with in practice. In fact, $R(M)$ is 
almost completely determined just 
using the following `intermediate' notion of the divisibility of
$p_M$, whose significance was pointed out by
Wilkens \cite[Conjecture p.\,548]{wilkens74}:
\begin{equation} \label{eq:mdiv}
\mdiv := \left \{ \begin{array}{cl} 0 & \text{if $p_M$ is torsion,} \\
{\rm Max} \bigl\{ s : s, \mpl \in \Z, \; s\mpl^2 \text{~divides~} \mpl p_M
\bigr\} & \text{otherwise.} \end{array} \right.
\end{equation}
Corollary \ref{cor:p2image} and \eqref{eq:i_and_r_intro} give the next theorem, where
for a fraction $\frac{a}{b}$ written in lowest terms $\NumB{\frac{a}{b}} = a$. 

\begin{thm} \label{thm:i_and_r}
Let $M$ be $2$-connected and let $\mdiv = \mdiv(M)$.
There is an integer $r \in \{0, 1, 2\}$ depending only on the base $(H^4(M), b_M, p_M)$, such
that 
\[ R(M) = \lcm(8,2^{r} \mdiv) .\]
In particular, by \eqref{eq:i_and_r_intro},
\[ I(M) = \NumB{\frac{2^{r} \mdiv}{8}} \Theta_7 .\]
If $TH^4(M)$ has odd order then $r = 1$.
\end{thm}

If $H^4(M)$ does have some 2-torsion then in general one needs to look
at the torsion linking form in detail to determine $r$. 
Wilkens' conjecture
\cite[Conjecture p.\,548]{wilkens74} for the inertia group is equivalent to 
supposing that $r = 1$ always, which is not true.  The invariant $r = r(G, b, p)$ is
defined in Definition \ref{def:r} and while we do not have a closed formula for $r$,
it is feasible to compute $r$ for any given example. 
For examples where $r = 0, 1$ or $2$, see Example \ref{ex:r012}.

We next discuss some consequences of Theorem \ref{thm:i_and_r} and its proof.
If $N$ is a closed smooth manifold, let $n_+(N)$ denote 
the number of oriented diffeomorphism classes 
of smooth structures on the topological manifold underlying $N$.
From Theorems \ref{thm:a_class} and \ref{thm:i_and_r} and Remark \ref{rmk:inertia} 
we deduce

\begin{cor} \label{cor:n_+}
If $M$ is $2$-connected then $n_+(M) = \gcd\left( \Num(2^{r-3}\mdiv), 28 \right)$.
\end{cor}
We call a homotopy equivalence $f \colon N_0 \to N_1$ of smooth manifolds 
{\em tangential} if there is a bundle isomorphism
$f^* TN_1 \cong TN_0$, where $TN_i$ is the tangent bundle of $N_i$, $i = 0, 1$.
In Lemma \ref{lem:the1} we show that 
a homotopy equivalence $f \colon M_0 \simeq M_1$ of 
$2$-connected $7$-manifolds with
$f^*p_{M_1} = p_{M_0}$ is tangential.  
Together with Theorem \ref{thm:i_and_r} this entails
\begin{cor} \label{cor:I(M)_is_the}
Let $M_0$ and $M_1$ be $2$-connected and let $f \colon M_0 \simeq M_1$ be a tangential homotopy equivalence.
Then $I(M_0) = I(M_1)$.
\end{cor}
One may wonder if Corollary \ref{cor:I(M)_is_the} is true because tangentially
homotopy equivalent \mbox{2-connected 7-manifolds}
are almost diffeomorphic (equivalently homeomorphic by
Theorem \ref{thm:a_class}).  However this was shown not to be the
case in \cite[p.\,114]{crowley02}, contradicting statements of Madsen, Taylor
and Williams \cite[Theorem C and Theorem 5.10]{madsen80}: see Proposition
\ref{prop:the} and Remark \ref{rmk:the}.

The computation of the reactivity of $M$ also has applications in $G_2$-topology.
We define the {\em smooth reactivity of $M$}, $R^\Diff(M)$, using the equation 
$p^2(\Diff(M)) = R^\Diff(M)\Z$, and
in \mbox{\cite[Section 6]{nu}} we show that $R^\Diff(M)$
determines the number of \gtstr s on $M$ modulo homotopies and diffeomorphisms.
By Corollary \ref{cor:p2image}, 
\[ R^\Diff(M) = \lcm(2^r\mdiv, 224) \]
for $2$-connected $M$,
and this allows us to generalise Theorem \ref{thm:class} to give a
classification of \mbox{2-connected} 7-manifolds equipped with a \gtstr, up to
diffeomorphisms and homotopies of \gtstr s \cite[Theorem 6.9]{nu}.

The proof of Theorem \ref{thm:i_and_r} gives subtle information about the mapping class group of~$M$.  
Let $I_H(M) \subseteq I(M)$ be the subgroup of the inertia
group of $M$ consisting of homotopy spheres $\Sigma$ such that there is a diffeomorphism
$f \colon M \sharp \Sigma \cong M$ where $H^*(f) = {\rm Id}$, considering $M \sharp \Sigma$ and $M$
as the same topological space.  Using the delicate algebra in Section \ref{ss:aut_q_on_gauss_functions}
we construct a surjective homomorphism 
\[ \wh P \colon \aut_{q^\circ}(H^4(M)) \to I(M)/I_H(M), \]
such that $F \in \aut_{q^\circ}(H^4(M))$ is realised by a diffeomorphism of $M$ if and only if $\wh P(F) = 0$.
Now by Theorem \ref{thm:a_class}, every $F \in \aut_{q^\circ}(H^4(M))$ is realised
by an almost diffeomorphism of $M$ and in Proposition \ref{prop:pairs_of_inertia_groups} we prove that every 
nested pair of subgroups $I_1 \subseteq I_2 \subseteq \Theta_7$ 
can be realised as
$I(M)_H \subseteq I(M)$ for some $2$-connected $M$.
As consequence we have 

\begin{thm} \label{thm:hatP_short}
There exist $2$-connected $M$ with automorphisms 
$F \in \aut_{q^\circ}(H^4(M))$ which are not realised by any diffeomorphism of $M$.
Necessarily every such $F$ is realised by an almost diffeomorphism of $M$.
\end{thm}

\begin{rmk}
Let us call an homeomorphism $f \colon M \to M$ exotic if it is not topologically isotopic 
to a diffeomorphism.  
Applying Theorem \ref{thm:hatP_short} gives examples 
of exotic homeomorphisms $f \colon M \to M$ whose exoticness is detected by 
their action on integral cohomology; see Example \ref{ex:mapping_inertia}.
To the best of our knowledge, these are the first examples of exotic homeomorphisms
of this kind.
\end{rmk}

\subsection{An overview of the proof of Theorem \ref{thm:class} and some remarks on surgery} \label{ss:surgery}
Every $2$-connected $M$ bounds a $3$-connected $8$-manifold $W$
and we define the {\em characteristic form} of $W$ to be the
triple $(H^4(W, \del W), \lambda_W, p_W)$, where
\[ \lambda_W \colon H^4(W, \del W) \times H^4(W, \del W) \to \Z, \]
is the intersection form of $W$, and $p_W \in H^4(W)$ is the spin
characteristic class of $W$ (see Section \ref{ss:basic_invariants}).
A key feature of dimension $8$ is that $p_W$ is
{\em characteristic for $\lambda_W$}, which means that
$\lambda_W(x, x) \equiv x \cup p_W$  mod $2$ for all $x \in H^4(W, \del W)$
(see Lemma \ref{lem:p_M}\ref{it:boundary}).
In \cite{wall62} Wall classified the manifolds $W$ by proving that
every iso\-morphism of characteristic forms is realised by a diffeomorphism.
He also proved that every abstract triple $(H, \lambda, \alpha)$,
where $\lambda \colon H \times H$ is
a symmetric bi-linear form and $\alpha \colon H \to \Z$ is characteristic for $\lambda$
in the sense above, is realised as the characteristic form of some $W$.

Following Wall, Wilkens \cite[Theorem 3.2]{wilkens71}
proved that any diffeomorphism $f \colon M_0 \to M_1$ 
between $2$-connected $M$ extends to a diffeomorphism 
$F \colon W_0 \to W_1$
for some $3$-connected coboundaries $W_i$ of $M_i$, $i = 0, 1$.
The results of Wall and Wilkens' reduced the classification of $2$-connected $7$-manifolds
to the classification of characteristic forms up to isometry
and orthogonal sum with {\em spherical} forms, where 
we call a characteristic form spherical if the boundary of the corresponding
handlebody is diffeomorphic to $S^7$; \cf \cite[\S 14]{wall67}.
This algebraic problem boils down to finding the 
correct notion of the {\em algebraic boundary} of a characteristic form.
Wilkens' base $(G, b, p)$ and the first author's refinement $(G, q^\circ, p)$ 
were partial solutions to this problem
and the mod~$28$ distillation $(G, q^\circ, \mu, p)$ of this paper
gives a complete solution.

The fundamental input to Wall and Wilkens' theorems is Smale's $h$-cobordism
theorem \cite{smale61}.  In addition, both proofs make use of handlebody
theory. Hence the topological inputs to our proofs are relatively elementary
from a modern perspective.
The reader may ask whether developments in manifold theory,
\eg the classical surgery theory 
of Browder-Novikov-Sullivan-Wall or the modified surgery
of Kreck,
give more powerful
tools to classify $2$-connected $7$-manifolds?

In the case of classical surgery, the answer to the above question is simply ``no''.
The homotopy classification of $2$-connected $7$-manifolds via 
the study of $CW$-decompositions and attaching maps, is 
surely harder than the smooth classification of these manifolds.
The 
reader may
consult \cite{sasao65} as a starting point.
Even if the homotopy classification is known,
the computation of the surgery structure set via the surgery exact sequence
and then the action of the self-equivalences on the structure set is 
a delicate problem.  Here the %
reader may consult \cite[Theorem 2.2]{crowley10}
for the case $M = S^3 \times S^4$.

The situation is different with modified surgery,
which provides a powerful tool for classifying $1$-connected
$7$-manifolds.  
Until recently the relevant results from modified surgery came from working
over the normal $2$-type and rested on the general classification theorem
\mbox{\cite[Theorem 6]{kreck99}}, which in the $2$-connected case
makes the very restrictive hypothesis that $H^4(M)$ is generated by $p_M$.
In \cite{kreck18} Kreck defines an enhanced normal $2$-type 
which applies to all $2$-connected $M$ and which he
uses to give an alternative proof of the uniqueness part of Theorem \ref{thm:class}.
The enhanced normal $2$-type encodes
what Kreck calls a {\em $d$-structure}
which, in the notation of this paper, is a pair $(M, k)$, where $k \in \divsp$.

\subsection{Dimensions 7 and 15} \label{ss:7_and_15}
Dimension $7$ and $15$ were exceptional for Wall's methods in \cite{wall67}
because the tangent bundles of $S^3$ and $S^7$ are trivial and this prevented
Wall from defining a quadratic refinement of the linking form.
We discuss this further
in Section \ref{ss:intrinsic}, 
where we show how Wall's methods can be extended to $2$-connected $7$-manifolds
by adding additional tangential structure.  In this way we are able to give an {\em intrinsic}
definition of the quadratic refinement of the linking form.

\begin{rmk} \label{rmk:15}
To discuss dimension $15$, let $String := O\an{6}$ denote the $6$-connected cover of 
the stable orthogonal group.  
This is a well-defined homotopy type with models which are topological groups 
(see \cite[Theorem 5.1]{stolz96}).
In particular there is a well-defined notion of a stable string structure on a manifold
and hence a well-defined notion of a stable string manifold.
For the almost diffeomorphism classification, a $15$-dimensional analogue of Theorem \ref{thm:a_class} was proven in \cite[Theorem B]{crowley02}.
For the smooth classification, we have $\Theta_{15} \cong \Z/8,\!128 \oplus \Z/2$,
where the $\Z/8,\!128$ summand is the subgroup of homotopy $15$-spheres which bound string manifolds and the $\Z/2$ summand maps isomorphically
to $\Omega_{15}^{\rm String}$, \cite{kervaire63, giambalvo71}.
The $15$-dimensional analogue of Theorem \ref{thm:class} requires the following modifications.
Firstly the universe of $6$-connected $15$-dimensional manifolds has two disjoint classes:
those manifolds which bound string manifolds and those which do not.
Secondly, within each of these classes a version of Theorem \ref{thm:class} holds, where
mod~$28$ Gauss refinements are replaced by mod~$8,\!128$ Gauss refinements, 
in the sense of Definition \ref{def:modngr}.
\end{rmk}

\subsection{Organisation} \label{ss:organisation}
The rest of this paper is organised as follows. In Section \ref{sec:invariants} we define the
invariants used in Theorems \ref{thm:a_class} and \ref{thm:class}.
In particular, families of quadratic refinements, 
Gauss refinements and the generalised Eells-Kuiper invariant are defined 
in Sections \ref{ss:families}, \ref{ss:gr} and \ref{ss:eek} respectively.
In Section \ref{sec:classification} we prove our main classification results
and we discuss the connected sum splitting of
$2$-connected $7$-manifolds in Theorems \ref{thm:almost_split1} and \ref{thm:split}.
Section \ref{sec:auto} is an algebraic section in which we analyse the automorphisms of refinements
and bases and the action of these automorphisms on Gauss refinements.
This section contains the proof of Theorem \ref{thm:i_and_r},
which follows from the computation of the reactivity of $M$ in Corollary \ref{cor:p2image}.

In Section \ref{sec:examples} we illustrate the classification of $2$-connected $M$ with examples
and we also present a refinement of Wilkens' identi\-fication of the set of indecomposable generators 
for the monoid of almost diffeomorphism classes of 
$2$-connected $7$-manifolds under the operation of connected sum;
see Theorem~\ref{thm:indecomposable}.
In Section \ref{sec:mcg_and_inertia} we investigate the relationship between
the inertia groups of $M$ and the mapping class groups of $M$ and prove
Theorem~\ref{thm:hatP_short}.

\smallskip
\noindent
{\bf Acknowledgements:}  We would like to thank Jim Davis,
Sebastian Goette, Matthias Kreck and Claude LeBrun for helpful comments and/or discussions. 
We would also like to thank Matthias Kreck for pointing out a mistake in the
formulation of Theorem \ref{thm:classification_minimal} in the first version of
this paper and for sharing drafts of his paper \cite{kreck18};
in particular the explicit form of the transformation rule in \eqref{eq:TR} is
based on \cite[\!(2)]{kreck18}.
In addition, we thank the referee for their careful comments, which have
improved the paper.

DC thanks the Mathematics Departments at Imperial College and the University of Bath for their hospitality 
and acknowledges support from EPSRC Mathematics Platform grant EP/I019111/1.
DC also acknowledges the support of the Leibniz Prize of Wolfgang L\"{u}ck, granted by the Deutsche Forschungsgemeinschaft.
JN thanks the Simons Foundation for its support under the Simons Collaboration
on Special Holonomy in Geometry, Analysis and Physics
(grant \#488631, Johannes Nordström).
\section{Invariants} \label{sec:invariants}
In this section we define the invariants needed to classify $2$-connected spin
$7$-manifolds~$M$.
In Section \ref{ss:basic_invariants} we introduce the linking form $b_M$ of
$M$ and the spin characteristic class $p_M \in 2H^4(M)$. 
In Section \ref{ss:coboundaries} we recall the
characteristic form 
$$(H^4(W, \del W), \lambda_W, p_W)$$ 
of a spin coboundary $W$ for $M$
and identify it as the salient algebraic model for $W$.
In sections
Sections \ref{ss:tfs_and_Gauss_sums}, \ref{ss:families} and \ref{ss:gr} we progressively
build algebraic ``boundary invariants'' of characteristic forms.
Section \ref{ss:tfs_and_Gauss_sums} recalls the theory of refinements of torsion forms and
Section \ref{ss:families} shows how a characteristic form defines a family of refinements
on its boundary.  In Section \ref{ss:gr} we define the generalised
Eells-Kuiper invariant $\EK_M$ of $M$ using Hirzebruch's
characteristic class formulae for the $\wh A$-genus and the $L$-genus.
The generalised Eells-Kuiper invariant is a reduced defect invariant of the $\wh A$-genus.
In Section \ref{ss:spin_c} we show how $\mu_M$ can be computed via a coboundary $W$
which is \spc rather than spin.  
Finally, in Section \ref{ss:intrinsic} we give an intrinsic definition of
the quadratic refinements defined via coboundaries in 
Section~\ref{ss:families}.
\subsection{Basic invariants} \label{ss:basic_invariants}
To any closed spin 7-manifold $M$ we associate its integral co\-ho\-mo\-logy
group $H^4(M)$, torsion linking form 
$b_M$ and spin characteristic class $p_M \in 2H^4(M)$.
We call the triple $(H^4(M), b_M, p_M)$ the base of $M$.
More generally, a base is a triple $(G, b, p)$ consisting of a finite
abelian group $G$,
a torsion form $b$ on the torsion subgroup of $G$ and an element $p \in 2G$.
For later use we introduce the category $\catb$ consisting of bases
with morphisms isomorphisms
\[ {\rm Ob}(\catb) = \{ (G, b , p) \}. \]
We now define in the invariants $b_M$ and $p_M$ in turn.

\subsubsection{The linking form \texorpdfstring{$b_M$}{b\_M}}
Recall that the linking form of a closed oriented 
$(4k{-}1)$-manifold $N$ is a nonsingular symmetric bilinear pairing
\[ b_N : TH^{2k}(N) \times TH^{2k}(N) \to \Q/\Z \]
defined on the torsion subgroup of $H^{2k}(N)$. 
Given $x, y \in TH^{2k}(N)$ and $\wh x \in H^{2k-1}(N; \Q/\Z)$, 
a lift of $x$ along the Bockstein 
$\beta \colon H^{2k-1}(N; \Q/\Z) \to TH^{2k}(N)$ associated to 
the coefficient sequence 
$\Z \to \Q \to \Q/\Z$, 
the linking form $b_N$ of $N$ is defined by the equation
$$ b_N(x, y) := \an{\wh x y, [N]} \in \Q/\Z. $$

If $N$ is the boundary of an oriented
$4k$-manifold $Y$, then the linking form of $N$ and the intersection form
of $Y$ are related, as explained in \cite[II]{alexander76}.
Let $i \colon H^{2k}(Y, N; \Q) \to H^{2k}(Y; \Q)$ be the natural map
and define a rational-valued intersection form on the image of $i$ by
$$ \rif_Y \colon \im(i) \times \im(i) \to \Q, \quad
\rif_Y(i(w), i(z)) := \an{w \cup i(z), [Y]}.$$
Let $j \colon H^{2k}(Y) \to H^{2k}(N)$ be the restriction map.
If $x \in TH^{2k}(N)$ and $\bar x \in H^{2k}(Y)$ is a preimage,
$j(\bar x) = x$, then the image of $\bar x$ in $H^{2k}(Y;\Q)$ is in the
kernel of the restriction $H^{2k}(Y;\Q) \to H^{2k}(N;\Q)$.
Thus the image of $j^{-1}(TH^{2k}(N)) \subset H^{2k}(Y)$ in $H^{2k}(Y;\Q)$
equals $\im(i)$. The linking form of $N$ satisfies
\begin{equation} \label{eq:b-by-cob}
b_N(x, y) = -\rif_Y(\bar x, \bar y)~~\text{mod}~\Z,
\end{equation}
whenever $\bar x, \bar y \in H^{2k}(Y)$ are lifts of $x$ and $y$
respectively.
Note that if the image of $j$ contains $TH^{2k}(N)$, then \eqref{eq:b-by-cob}
describes $b_M$ completely.
The appearance of the minus sign in \eqref{eq:b-by-cob} is explained
in \cite[Proof of Theorem 2.1]{alexander76}
and also in \cite[\S 3]{gordon78}.

\subsubsection{The spin characteristic class \texorpdfstring{$p_M$}{p\_M}}
The classifying space $B\spg$ is $3$-connected and has
$\pi_4(B\spg) \cong \Z$.
It follows that $H^4(B\spg) \cong \Z$ is infinite cyclic.  A generator is
denoted $\pm \frac{p_1}{2}$ and the notation is justified since for the
canonical map $\pi \colon B\spg \to BSO$ we have $\pi^*p_1 = 2 \frac{p_1}{2}$
where $p_1 \in H^4(BSO)$ is the first Pontrjagin class,
see \eg \cite[Lemma 2.2]{mclaughlin92}.

One way to explain the claims in the previous paragraph is to note that the 
canonical homomorphism $SU \to \spg$, which maps the stable special unitary
group to the stable spin group, induces an isomorphism
$H^4(B\spg) \to H^4(BSU)$. Since $H^4(BSU)$ is cyclic with a generator
(namely the universal second Chern class $c_2$) whose image in $H^4(BSO)$ is
$2p_1$, the same is true for $H^4(B\spg)$ (see also Lemma \ref{lem:pz} below).

Given a spin manifold $N$ we write
\[ p_{N} : = \frac{p_1}{2}(N) \in H^4(N) .\]

\begin{rmk}
\label{rmk:topological_p_M}
In order to prove the topological invariance of invariants we define in the
later subsections, we consider $p_Y$ for general topological spin
manifolds $Y$. 
We let $BTop$ denote the classifying space for
stable topological microbundles, see \cite[Essay IV, Proposition~8.1]{kirby77},
and $BTop\an{4}$ its $3$-connected cover.
Equivalently, $BTop\an{4}$ is the classifying space for stable spin topological microbundles.
By \cite[(3)]{hsiang69} there is a split short exact sequence
\[ 0 \to \pi_4(B\spg) \to \pi_4(BTop\an{4}) \to \Z/2 \to 0. \]
It follows that the canonical homomorphism $H^4(BTop\an{4}) \to H^4(B\spg)$ is an isomorphism 
and so $p_N \in H^4(N)$ is 
a homeomorphism invariant of
topological spin manifolds.
\end{rmk}

By \cite[Lemma 6.5]{kreck91}, the mod 2 reduction of $p_N$ is the $4th$
Stiefel-Whitney class $w_4$. 
This has the following consequences for the parity of $p_N$.

%
%
%
%

\begin{lem}  \label{lem:p_M}
\hfill
\begin{enumerate}
\item \label{it:closed7}
Let $M$ be a closed spin $7$-manifold. Then $p_M \in 2 H^4(M)$.
\item \label{it:closed8}
Let $X$ be a closed spin $8$-manifold.
For all $x \in H^4(X; \Z/2)$
\[ x^2 = x \cup p_X \in H^8(X;\Z/2). \]
\item \label{it:boundary}
Let $W$ be a compact spin $8$-manifold with boundary $M$.
For all $x \in H^4(W, M; \Z/2)$
\[ x^2 = x \cup p_W \in H^8(W,M; \Z/2). \]
\end{enumerate}
\end{lem}

\begin{proof}
By Wu's formula, see \eg
\cite[Theorem 11.14]{milnor74}, $w_4 = v_4$ for any closed spin manifold 
since the first three Wu classes of a spin manifold vanish.

\begin{enumerate}
\item Now $v_4(M) = 0$ since $M$ is $7$-dimensional, the Wu class satisfies
$v_4(M) \cup x = Sq^4(x)$ for all $x \in H^3(M; \Z_2)$ by definition,
and $Sq^4$ vanishes on classes of degree three.

\item $x^2 = Sq^4(x) = x \cup v_4(X) = x \cup p_X$.

\item Let $X := W \cup_{\Id_M} (-W)$. The push-forward
$i_* \colon H^*(W,M) \to H^*(X)$ of the inclusion $i \colon W \into X$ is dual under the
Poincar\'e pairing to the restriction $i^* \colon H^4(X) \to H^4(W)$.
Since $i^*p_X = p_W$, (ii) gives
\[ x^2 = (i_*x)^2 = i_*x \cup p_X = x \cup p_X , \]
where the equalities take place in
$\Z/2 \cong H^8(X; \Z/2) \cong H^8(W,M; \Z/2)$. \qedhere
\end{enumerate}
\end{proof}

\begin{rmk} \label{rem:pM_and_TM}
The characteristic class $p_M$ is the primary and final obstruction to the
triviality $TM$, the tangent bundle of $M$; 
\ie $TM$ is trivial if and only if $p_M = 0$.
This is because of Bott periodicity, 
which states that $\pi_5(BSO(7)) = \pi_6(BSO(7)) = 0$ and we
have the exceptional fact that $\pi_7(BSO(7)) = 0$ by 
\cite[p.\,162]{kervaire60}.
Hence all obstructions to the triviality of the tangent bundle of
$M$ after $p_M$ vanish.
Indeed, any rank $7$ (or higher) spin vector bundle $E$ over a $CW$-complex $X$
is trivial over the $7$-skeleton of $X$ if and only if $p(E) = 0$.
\end{rmk}
\subsection{Algebraic models of coboundaries} \label{ss:coboundaries}
Let $M$ be a closed spin 7-manifold. Since the bordism group $\Omega_7^{\spg}$
vanishes by \cite{milnor63}, there is a compact spin
8-manifold $W$ such that $\partial W = M$.  Applying surgery below the middle dimension to $W$ \cite[Theorem 3]{milnor61}, we can assume that $W$
is 3-connected.  
We define $FH^4(W, \del W) : = H^4(W, \del W)/TH^4(W, \del W)$
to be the torsion-free quotient of $H^4(W, \del W)$.
Since $W$ is $3$-connected, $H^4(W)$ is torsion-free and so
the relative cohomology sequence of $(W, M)$ gives exactness of
\begin{equation} \label{eq:M_and_W}
FH^4(W, \del W) \to H^4(W) \to H^4(M) \to 0.
\end{equation}

For an abelian group $H$, let $H^* := \Hom(H, \Z)$ be the dual of $H$.
Since $H^4(W)$ is torsion-free,
the composition $H^4(W) \to H_4(W, \del W) \to H^4(W, \del W)^*$
of the Kronecker homomorphism 
with the Poincar\'e-Lefschetz duality isomorphism 
is an isomorphism.
Hence
the first homomorphism in \eqref{eq:M_and_W} can be thought of as the adjoint
homomorphism,
\[  \hlambda_W \colon FH^4(W, \del W) \to FH^4(W, \del W)^*, \]
of the intersection pairing
$\lambda_W$ on $FH^4(W, \del W)$.
The principle we follow is to regard
the pair $(FH^4(W, \del W), \lambda_W)$ as a ``model'' for a coboundary~$W$.

Let us set up some terminology to deal with these models.
We say that $(H,\lambda)$ is an integral form if $H$ is a finitely generated
free abelian group and $\lambda \colon H \times H \to \Z$ is symmetric and bilinear.
Let $\hlambda$ denote the adjoint homomorphism $H \to H^*$.
The ``boundary''
of $(H,\lambda)$ is $G := \coker(\hlambda)$. We say that an element
$\alpha \in H^*$ is \emph{characteristic} for $\lambda$ if
$\lambda(x,x) = \alpha(x) \mmod 2$ for all $x \in H$. We then call
$(H,\lambda,\alpha)$ a \emph{characteristic form}.

If $W$ is a 3-connected coboundary of $M$ then 
the pair $(FH^4(W, \del W), \lambda_W)$
is an integral form with boundary $H^4(M)$.
By Lemma \ref{lem:p_M}\ref{it:boundary},
$(FH^4(W, \del W), \lambda_W, p_W)$ is characteristic.
(If $M$ is $2$-connected then Wall's classification of $3$-connected $8$-manifolds \cite{wall62}
ensures that the characteristic form of $W$ is a
complete invariant of $W$ under diffeomorphisms;
\ie every isomorphism of characteristic forms is realised by a diffeomorphism,
see \cite[Corollary~2.5]{crowley02}.)
In the next subsections we study the structures that an integral or
characteristic form induces on its boundary. By applying this to the
algebraic model of a 3-connected coboundary of $M$ we obtain the desired
algebraic invariants of $M$. To prove that they are independent of the
choice of $W$ we will combine a splitting result for the algebraic constructions
with the following lemma whose proof is a simple application of
the Mayer-Vietoris theorem.

\begin{lem}
\label{lem:split_mfd}
Let $W_i$ be compact 3-connected spin 8-manifolds with $2$-connected boundaries,
$f \colon \partial W_0 \to \partial W_1$ a homeomorphism,
and $X := (-W_0) \cup_f W_1$ (a closed spin topological manifold).
Then for $i = 0,1$ we have injections 
$H^4_{}(W_i, \del W_i) \into H^4(X)$ whose images are orthogonal to each
other with respect to the intersection form $\lambda_X$ of $X$.
Further, the restriction map $H^4(X) \to H^4(M)$ is surjective, with kernel
$H^4_{}(W_0, \del W_0) + H^4_{}(W_1, \del W_1)$. \qed
\end{lem}

\begin{rmk} \label{rmk:characteristic_gluing}
We note that by Lemma \ref{lem:p_M}\ref{it:closed8}, the triple
$(H^4(X), \lambda_X, p_X)$ of the manifold $X$ in Lemma \ref{lem:split_mfd} is
a (nonsingular) characteristic form.
Moreover, the image of $p_X$ under the restriction map $H^4(X) \to H^4(W_i)$ is of course $p_{W_i}$.
\end{rmk}

\subsection{Torsion forms and quadratic refinements on finite groups}
\label{ss:tfs_and_Gauss_sums}

Throughout this paper
$T$ is
a \emph{finite} abelian group.
We say that $b \colon T \times T \to \Q/\Z$ is a \emph{torsion form} on $T$
if it is symmetric, bilinear, and nonsingular in the sense that
the induced map $T \to \Hom(T, \Q/\Z)$ is an isomorphism.
We call a function $q \colon T \to \Q/\Z$ a \emph{quadratic refinement} of $b$
if
\[ q(x+y) = q(x) + q(y) + b(x, y), ~~ \forall\,x, y \in T . \]
The \emph{homogeneity defect} of $q$ is the unique element $\beta = \beta(q) \in 2 T$ such that
for all $x \in G$ $q(x)-q(-x) = b(x, \beta)$.
If $\beta = 0$ then $q(x) = q(-x)$ and $q$ is called \emph{homogeneous}.
We define
\[ \cal{Q}(b) := \{ q :  \text{$q$ is quadratic refinement of $b$} \} \]
and we let $\cal{Q}^0(b) \subseteq \cal{Q}(b)$ be the set of homogeneous quadratic 
refinements of $b$.
In this subsection we consider the problem of classifying the quadratic refinements in
$\cal{Q}(b)$ up to isomorphism.  For $\cal{Q}^0(b)$ this problem was solved by
Nikulin \cite{nikulin79} and the general solution was given independently by
Deloup and Massuyeau \cite{deloup05} and the first author \cite{crowley02}.

The first basic results \cite[Lemmas 2.30 \& 2.31]{crowley02} are that
$\cal{Q}(b)$ and $\cal{Q}^0(b)$ are both non-empty and that $T$ acts 
freely and transitively on $\cal{Q}(b)$ via the action
\[  \cal{Q}(b) \times T \to \cal{Q}(b), \quad (q, t) \mapsto q_t,\]
where we recall from the introduction that for all $t \in T$, 
\[ q_t(x) = q(x) + b(x, t) = q(x+t) - q(t). \]
It is clear that the homogeneity defects of $q$ and $q_t$ are related by
$\beta(q_t) = \beta(q) + 2t$.

\begin{ex} \label{ex:cyclic_b_and_q}
If $T \cong \Z/r\Z$ is cyclic then all torsion forms and refinements on $T$
are given by the following examples.
Given $\theta \in \Z/r$ coprime to $r$, let
$\anb{\theta}{r}$ denote $\Z/r$ 
equipped with the torsion form
\[ b(x,y) := \frac{\theta xy}{r} \in \Q/\Z . \]
Given $\theta \in \Z/2r$ coprime to $r$ and $\gamma \in \Z/r$
(so that $2 \gamma \in \Z/2r$), we define a quadratic refinement
$\anq{\theta}{2r}_\gamma$ of $\anb{\theta}{r}$ by
\[ q(x) := \theta \left(\frac{ x^2 + 2\gamma x}{2r}\right) \in \Q/\Z. \]

\end{ex}

Beyond the homogeneity defect,
we introduce two further equivalent invariants of $q$.
The first of these is the {\em Gauss sum of $q$} which is the complex number
\[ GS(q) := \sum_{x \in T} e^{2\pi i q(x)} \in \C,  \]
where $i = \sqrt{-1}$ and $e$ is Euler's number.
From the fact that $q_t(x) = q(x+t)-q(t)$
one easily obtains the following useful

\begin{lem}[{\cite[(4.1)]{deloup05}}] \label{lem:Gauss_sum_shift}
$GS(q_t) = e^{-2\pi iq(t)} GS(q)$. \qed
\end{lem}

It is a theorem of Milgram \cite[Theorem, p.\,127, Appendix 4]{milnor73} that if $q$ is homogeneous, then $GS(q)$
is a non-zero complex number with modulus $\sqrt{|T|}$: by Lemma \ref{lem:Gauss_sum_shift}, this
holds for all $q \in \cal{Q}(b)$.
We define the {\em Arf invariant} of $q$ to be the number $\Arf(q) \in \Q/\Z$ which is the argument of $GS(q)$ divided by $2 \pi$.  That is
\begin{equation} \label{eq:arf_def}
 GS(q) = \sqrt{|T|}e^{2\pi i \Arf(q)} \in \C. 
\end{equation}
Then Lemma \ref{lem:Gauss_sum_shift} is equivalent to
\begin{equation}
\label{eq:arf_shift}
\Arf(q_t) = \Arf(q) - q(t) .
\end{equation}

Before giving the classification theorems for $\cal{Q}(b)$, we review how elements of $\cal{Q}(b)$
can be presented as the boundaries of nondegenerate characteristic forms
$(H, \lambda, \alpha)$ and how $\Arf(q)$ is determined by
$(H, \lambda, \alpha)$ in this situation.
If $\lambda$ is nondegenerate then the boundary $T := \coker (\hlambda)$
of $(H, \lambda)$ fits into the short exact sequence
\begin{equation} \label{eq:HH*T}
0 \to H \xra{~\hlambda~} H^* \xra{~\ares~} T \to 0.
\end{equation}
We write
$\lambda_\Q \colon (H \otimes \Q)  \times (H \otimes \Q) \to \Q$
for the rational form induced by $\lambda$.
Its adjoint
$\hlambda_\Q \colon H \otimes \Q \to H^* \otimes \Q$
is an isomorphism, and we use the inverse
$(\hlambda_\Q)^{-1} \colon H^* \otimes \Q \to H \times \Q$ 
to pull back the form $\lambda_\Q$ on $H \otimes \Q$.
We obtain a rational symmetric bilinear form on $H^* \otimes \Q$,
and restricting to $H^* \subset H^* \otimes \Q$
gives the rational-valued bilinear form 
\[ \rif: = (\hlambda_\Q^{-1})^*(\lambda_\Q)|_{H^* \times H^*} \colon H^* \times H^* \to \Q.\]
Explicitly, if $y, z \in H^*$ 
and if $y = k \hlambda(\wt y)$ and 
$z = l \hlambda(\wt z)$ for some integers $k$ and $l$ then,
$$ \rif(y, z) =  
\lambda_{\Q}(\hlambda_{\Q}^{-1}(y), \hlambda_{\Q}^{-1}(z)) = 
\frac{\lambda(\wt y, \wt z)}{kl}
= \an{\hlambda_{\Q}^{-1}(y), z}.$$

\begin{rmk} \label{rem:rif}
In \cite{crowley02} the form $\rif \colon H^* \times H^* \to \Q$ is denoted
$\lambda^{-1}$.
\end{rmk}
\begin{rmk}
When the sequence \eqref{eq:HH*T} is the sequence
$H^4(W, \del W) \longrightarrow H^4(W) \xra{~j~} H^4(M)$
of a $3$-connected coboundary $W$ as in Section \ref{ss:coboundaries}
with $H^4(M) = TH^4(M)$,
then the form $\rif \colon H^* \times H^* \to \Q$ 
is precisely the restriction of the rational-valued intersection form of $W$,
$\lambda_W \colon H^4(W; \Q) \times H^4(W; \Q) \to \Q$,
to $j^{-1}(TH^4(M)) = H^4(W) \subset H^4(W; \Q)$.
\end{rmk}

Given a nondegenerate characteristic form $(H, \lambda, \alpha)$ and
$x, y \in T$, let $\bar x, \bar y \in H^*$ be such that $\ares(\bar x) = x$
and $\ares(\bar y) = y$.
We define the torsion form $b_\lambda$ 
\[ b_\lambda \colon T \times T \to \Q/\Z, \quad (x, y) \mapsto 
-\rif(\bar x, \bar y)~~\text{mod}~\Z, \]
and the quadratic refinement of $b_\lambda$
\begin{equation} \label{eq:bndry_of_char_form}
q_{\lambda, \alpha} \colon T \to \Q/\Z, 
\quad x \mapsto \frac{-\rif(\bar x, \bar x) - \rif(\bar x, \alpha)}{2}~~\text{mod}~\Z.
\end{equation}
We regard $\left(T, q_{\lambda, \alpha}, j(\alpha)\right)$ 
as the {\em boundary} of $(H, \lambda, \alpha)$
and note that the homogeneity defect of $q_{\lambda, \alpha}$ is exactly $\ares(\alpha)$.

\begin{rmk} \label{rmk:sign_of_q}
The minus signs in \eqref{eq:bndry_of_char_form} are introduced
to correspond to the sign in \eqref{eq:b-by-cob}.
The sign differs from \cite[Definition 2.32]{crowley02}.
As a consequence, the definition of the linking form and quadratic linking family in
\cite[Definition 2.50]{crowley02} have the wrong signs.
\end{rmk}

\begin{ex}
Let us discuss the calculation of $(T, b)$ 
from $(H, \lambda)$ in more detail.
If $H$ has basis $\{v_1, \dots, v_n\}$ and $\lambda$ is represented with
respect to this basis by the symmetric integer matrix $B$ 
then $B$ is invertible over $\Q$ and the rational symmetric matrix
$B^{-1}$ expresses $\lambda \colon H^* \times H^* \to \Q$
with respect to the dual basis $\{v_1^*, \dots, v_n^*\}$ of $H^*$.
It follows that the mod~$\Z$ values of $B^{-1}$ express
the linking form $b$ with respect to the generating 
set $\{j(v_1^*), \dots, j(v_n^*)\}$ of $T$.

For example, suppose that $H = \Z^2$ with basis $\{v_1, v_2\}$,
and $\lambda$ and $\alpha$ are given by
$$ B = \twotwo{0}{2^i}{2^i}{0},
\quad \alpha(v_1) = 2a_1,
\quad \alpha(v_2) = 2a_2,$$
where $a_1, a_2 \in \Z$.
Then $T = \Z/2^i \oplus \Z/2^i$
with generating set $\{(j(v_1^*), (j(v_2^*))\}$,
$b$ has linking matrix
$$ \twotwo{0}{-2^{-i}}{-2^{-i}}{0}$$
with respect to $\{j(v_1^*), j(v_2^*)\}$
and $q_{\lambda, \alpha}$ is given by the formula 
$$ q_{\lambda, \alpha}\bigl( k j(v_1^*) + l j(v_2^*) \bigr) = \frac{-kl - a_2l - a_1k}{2^{i}}.$$
\end{ex}

The following fundamental theorem of Wall states that
every linking form and quadratic refinement are realised as the boundary 
of some even nondegenerate form.

\begin{thm}[{\cite[Theorem 6]{wall63}}]
For all torsion forms $b$ and for every $q \in \cal{Q}^0(b)$, there is an even nondegenerate
form $(H, \lambda)$ and an isomorphism $q \cong q_{\lambda, 0}$.
\end{thm}

We now state Milgram's theorem on the Gauss sums of homogeneous quadratic torsion forms.

\begin{thm}[Milgram {\cite[Theorem, p.\,127, Appendix 4]{milnor73}}] \label{thm:Milgram}
Let $q \in \cal{Q}^0(b)$ and $(H, \lambda)$ be an even nondegenerate integral form
with signature $\sigma(\lambda)$.  Then
\begin{enumerate}
\item $8\Arf(q) \in \Z$;
\item $8 \Arf(q_{\lambda, 0}) \equiv -\sigma(\lambda)$~mod~$8$.
\end{enumerate}
\end{thm}

Following Milgram's theorem, we can restate Nikulin's classification of homogeneous quadratic
refinements of $b$ as follows.

\begin{thm}[{\cite[Theorem 1.11.3]{nikulin79}}] \label{thm:classification_of_homogeneous_q}
If $q_0, q_1 \in \cal{Q}^0(b)$ then $q_0$ is isomorphic to $q_1$ if and only if $\Arf(q_0) = \Arf(q_1)$.
\end{thm}

For general quadratic refinements of $b$ we have the following results.

\begin{prop}[{\cite[Proposition 5.19]{crowley02}}]
\label{prop:arf}
Any nondegenerate characteristic form $(H, \lambda, \alpha)$ has
\[ \Arf(q_{\lambda, \alpha}) = \frac{\rif(\alpha, \alpha) - \sigma(\lambda)}{8} \in \Q/\Z . \]
\end{prop}

\begin{thm}[{\cite[Theorem 5.22]{crowley02}}, {\cite[Theorem 4.1]{deloup05}}] \label{thm:classification_of_q}
Let $q_0, q_1 \in \cal{Q}(b)$ be quadratic refinements with homogeneity defects 
$\beta_0$ and $\beta_1$ respectively.  Then $q_0$ and $q_1$
are isomorphic if and only if the following hold:
\begin{enumerate}
\item There is an automorphism $f \colon T \cong T$ of $b$ such that $f(\beta_0) = \beta_1$;
\item $\Arf(q_0) = \Arf(q_1)$.
\end{enumerate}
\end{thm}

\begin{rmk}
The proof of Theorem \ref{thm:classification_of_homogeneous_q} in \cite{nikulin79} 
and the proof of Theorem \ref{thm:classification_of_q} 
in \cite{crowley02} both apply classification results for torsion forms and case by case checking.
In contrast, the proof of Theorem \ref{thm:classification_of_q} in \cite{deloup05}
is short and general, with one elegant argument covering all cases.
\end{rmk}

\subsection{Families of quadratic refinements} \label{ss:families}

Let $G$ be a finitely generated abelian group, $p$~an element of $2G$,
and $b$ a torsion form on the torsion subgroup $T$; 
\ie $(G,b,p)$ is a base and so is an object in the category $\catb$.
Define
\[ \divs_2 := \{ h \in G : p - 2h \in T \} , \]
and for $h \in \divs_2$ write $\beta_h := p_M - 2h$. Note that $T$ acts
simply transitively on $\divs_2$ by addition.

\begin{defn} \label{def:foqr}
A \emph{family of quadratic refinements} of a base $(G,b,p)$ is defined to be a function
$q^\circ \colon \divs_2 \to \cal{Q}(b), \; h \mapsto q^h$, such that:
\begin{enumerate}
\item \label{it:q_hom-defect}
the homogeneity defect of $q^h$ is $\beta_h$;
\item \label{it:q_trans}
$q^{h+t} = q^h_{-t}$ for any $t \in T$.
\end{enumerate}
The triple $(G, q^\circ, p)$ is called a {\em refinement} of $(G, b, p)$.
\end{defn}
An isomorphism $F \colon G \to G'$ obviously maps $\divs_2 \to \divs'_2$,
and $F$ pulls back a family of quadratic refinements $q'^\circ$ on
$G'$ to one on $G$ by setting
\begin{equation*}
(F^\# q')^h := q'^{F(h)} \circ F_{|T} .
\end{equation*}
In this case $q^\circ$ and $q'^\circ$
are isomorphic via $F$, and so are $(G, q^\circ, p)$ 
and $(G', q'^\circ, F^{-1}(p))$.

The orthogonal sum of two 
refinements 
$(G_0, q^\circ_0, p_0)$ and $(G_1, q^\circ_1, p_1)$ 
is the refinement
$(G_0 \oplus G_1, q^\circ_0 \oplus q^\circ_1, p_0 \oplus p_1)$
as defined in Section \ref{ss:elaboration}. 
The negative of a 
refinement $(G, q^\circ, p)$ of $(G, b, p)$ is the 
refinement $(G, -q^\circ, p)$ of $(Q, -b, p)$
defined by $(-q)^h = -q^h$.  
For later use we introduce the category $\catq$ consisting of refinements
with morphisms isomorphisms
\[ {\rm Ob}(\catq) = \{ (G, q^\circ\!, p) \}. \]

Refinements as in Definition \ref{def:foqr} 
are defined naturally on the boundaries of characteristic forms
$(H, \lambda, \alpha)$ when $\lambda$ is allowed to be degenerate.
First we define the base $(G, b, p)$.
Let $G := \coker(\hlambda)$, 
let $\rad :=\ker(\hlambda) \subset H$ be the radical of $\lambda$ and let
$\res \subseteq H^*$ the annihilator of $\rad$.
Then $\res \cong (H/K)^*$, and the form $\lambda$ descends to a nondegenerate
form ${\lambda_{/K} \colon H/\rad \times H/\rad \to \Z}$
with $R / \im(\hlambda_{/K}) \cong T$.
Hence we obtain a torsion form $b = b_\lambda$ on $T$ as in 
Section \ref{ss:tfs_and_Gauss_sums}.
To define $p$ we let $j \colon H^* \to G$ be the projection
and set $p := \ares(\alpha)$.
Regardless of whether $\lambda$ is degenerate or not, the classification of
$\Z_2$-valued bilinear forms implies that there is always
an $x \in H$ such that $\lambda(x,y) = \lambda(y,y) \mmod 2$ for any $y \in H$.
Then $\alpha - \hlambda(x) \in H^*$ is even, so $p = j(\alpha) \in G$ is
even as required.

\begin{defn} \label{def:bbocf}
The {\em boundary base} of a characteristic form $(H, \lambda, \alpha)$
is defined to be the triple 
$(G, b, p) := \bigl( \coker(\hlambda), b_\lambda, j(\alpha) \bigr)$.
\end{defn}

Next we define the induced family of quadratic refinements.
For any $h \in \divs_2$, pick $m \in H^*$ such that $\ares(m) = h$ and
set $\alpha_m = \alpha - 2m$.
Then $\ares(\alpha_m)$ is a torsion element
and so $\alpha_m \in \res$ which
is characteristic for $(H/\rad, \lambda_{/K})$ and we let
$$q^h_{(H, \lambda, \alpha)} := q_{\lambda_{/K}, \alpha_m}$$
be the quadratic refinement of $b$ defined in \eqref{eq:bndry_of_char_form} 
in the previous subsection, 
\ie if $\bar x \in R$ and $\ares(\bar x) = x$ then
\begin{equation}
\label{eq:alg_qdef}
q^h(x)_{(H, \lambda, \alpha)} = \frac{\rif(\bar x, \bar x) + \rif(\alpha_m, \bar x)}{2}
= \frac{\gen{\hlambda^{-1}_{\Q}(\bar x), \bar x + \alpha_m}}{2} \in \Q/\Z.
\end{equation}
This is independent of the choice of $m$, since if $m' = m + \hlambda(r)$ 
then
\[ \rif(2m', \bar x) - \rif(2m, \bar x) = 2\gen{r, \bar x} \in 2\Z . \]
That (i) of Definition \ref{def:foqr} is satisfied is immediate from
$\beta_h = j(\alpha_m)$.
Meanwhile, if $h' = h + t$ for some $t \in T$ then
$\ares(m' - m) = t$, so 
\[ q^{h'}_{(H, \lambda, \alpha)}(x) - q^h_{(H, \lambda, \alpha)}(x) = 
\gen{-\hlambda^{-1}_\Q(\bar x), m - m'} = -b_\lambda(x, t) , \]
which shows that (ii) of Definition \ref{def:foqr} holds.  

\begin{defn} \label{def:bocf}
The {\em boundary} of a characteristic form $(H, \lambda, \alpha)$
is the triple
$$ \del (H, \lambda, \alpha) : = 
(\coker(\wh \lambda), q^\circ_{(H, \lambda, \alpha)}, j(\alpha)).$$
It is a refinement of the boundary base $(\coker(\wh \lambda), b_\lambda, j(\alpha))$
of Definition \ref{def:bbocf}.
\end{defn}

It is clear that an isomorphism of characteristic
forms $E \colon (H_0, \lambda_0, \alpha_0) \cong (H_1, \lambda_1, \alpha_1)$
induces an isomorphism 
$\del E \colon \del(H_0, \lambda_0, \alpha_0) \cong \del(H_1, \lambda_1, \alpha_1)$
of the boundary refinements.
It is also clear that the boundary of an orthogonal sum of characteristic forms is
the orthogonal sum of the boundaries and that 
$\del(H, -\lambda, \alpha) = -\del(H, \lambda, \alpha)$.

We call a characteristic form $(H,\lambda,\alpha)$ {\em nonsingular} if
$\lambda$ is, \ie if the adjoint $\hlambda \colon H \to H^*$ is an isomorphism.
Suppose that $(H,\lambda, \alpha)$ is a nonsingular characteristic form,
$H_0$ is some primitive subgroup of~$H$ and $H_1$ is the $\lambda$-orthogonal
subspace to $H_0$.
Let $\alpha_i \in H_i^*$ be the restrictions of $\alpha$ to $H_i$.
Let $\lambda_1$ be the restriction of $\lambda$ to $H_1$, and $\lambda_0$
the restriction of $-\lambda$ to~$H_0$. 
In this case we say that $(H_0, \lambda_0, \alpha_0)$ and
$(H_1, \lambda_1, \alpha_1)$ are {\em orthogonal in $(H, \lambda, \alpha)$}.
For the groups $G_i = \coker(\wh \lambda_i)$ of the boundaries of 
$(H_i, \lambda_i, \alpha_i)$, the restriction maps $H^* \to H_i^*$ 
and the isomorphism $\wh \lambda \colon H \cong H^*$ give rise to homomorphisms
$H \to H^* \to H_i^* \to G_i$
which induce isomorphisms 
\[ \Pi_i \colon H/(H_0 \oplus H_1) \cong G_i.\]
Given $(H_0, \lambda_0, \alpha_0)$ orthogonal to
$(H_1, \lambda_1, \alpha_1)$ in $(H, \lambda, \alpha)$,
we thus have the canonical isomorphism
\begin{equation} \label{eq:canonical_iso}
F_\lambda : = \Pi_1 \circ \Pi_0^{-1} \colon G_0 \cong G_1. 	
\end{equation}
(There is a slight asymmetry in the definition of orthogonal forms: 
If $(H_0, \lambda_0, \alpha_0)$
is orthogonal to $(H_1, \lambda_1, \alpha_1)$ in $(H, \lambda, \alpha)$,
then $(H_1, \lambda_1, \alpha_1)$ is orthogonal to $(H_0, \lambda_0, \alpha_0)$
in $(H, -\lambda, \alpha)$. However, the isomorphisms $F_\lambda : G_0 \to G_1$
and $F_{-\lambda} : G_1 \to G_0$ are precisely inverse to each other.)
The following lemma is a routine calculation using \eqref{eq:alg_qdef} and
the fact that $(H, \lambda, \alpha)$ is nonsingular: see \cite[Lemma 3.10]{crowley02}
for the case where $(H_i, \lambda_i, \alpha_i)$ are nondegenerate.

\begin{lem} \label{lem:splitq}
Let $(H_0, \lambda_0, \alpha_0)$ and $(H_1, \lambda_1, \alpha_1)$ 
be orthogonal characteristic forms in the nonsingular
characteristic form $(H, \lambda, \alpha)$.  
If for $i = 0, 1$, $(G_i, q^\circ_i, p_i)$ denotes the boundary of $(H_i, \lambda_i, \alpha_i)$,
then the canonical isomorphism $F_\lambda$ of \eqref{eq:canonical_iso} 
induces an isomorphism of the boundaries:
\[
\pushQED{\qed} 
F_\lambda^\#(q^\circ_1, p_1) = (q^\circ_0, p_0). \qedhere
\popQED
\]
\end{lem}

As discussed in Section \ref{ss:coboundaries},
if $W$ is a 3-connected coboundary of a closed
spin 7-manifold $M$ then $(FH^4(W, \del W), \lambda_W, p_W)$ is a
characteristic form.
Note that the associated boundary in $\catb$ 
in the sense of Definition \ref{def:bbocf} 
is precisely the base of $M$,
$(H^4(M), b_M, p_M)$, where $b_M$ the torsion linking form of $M$
as described in \eqref{eq:b-by-cob}.
Hence we have the following

\begin{lem} \label{lem:bboM}
The base of a spin $7$-manifold $M$ is the boundary base
of the characteristic form of any 3-connected coboundary $W$ of $M$;
\ie
\[
 \pushQED{\qed} 
(H^4(M), b_M, p_M) = \del(H^4(W, \del W), \lambda_W, p_W).
\qedhere
\popQED
\]
\end{lem}

\begin{defn}
\label{def:qlf}
The \emph{quadratic linking family} $q^\circ_M$ of $M$ is the family of
quadratic refinements of $(H^4(M), b_M, p_M)$ defined by the characteristic
form $(FH^4(W, \del W), \lambda_W, p_W)$, 
where $W$ is any 3-connected coboundary $W$ of $M$.
Explicitly, applying \eqref{eq:alg_qdef} we obtain for all $h \in S_2$ that
$$ q_M^h(x) = 
\frac{-\lambda_W(\bar x, \bar x) - \lambda_W(\alpha_m, \bar x)}{2},$$
where $\bar x \in H^4(W)$ is a lift of $x$,
$\alpha_m = p_W - 2m$ and $m \in H^4(W)$ is a lift of $h$.

Moreover, if $d_\pi = 0$---in particular if $M$ is a rational homology sphere---then 
we have the preferred element $0 \in S_2 = TH^4(M)$
and $q_M := q^0_M$ is {\em the quadratic refinement of $M$}.
\end{defn}

If $W_0$ and $W_1$ are 3-connected coboundaries of $M_0$ and $M_1$
respectively and $f : M_0 \to M_1$ is a homeomorphism, 
let $X$ be the closed topological spin
manifold $(-W_0) \cup_f W_1$.
Then Lemma \ref{lem:split_mfd} and Remark \ref{rmk:characteristic_gluing}
imply that the characteristic forms
$(FH^4(W_0, M_0), \lambda_{W_0}, p_{W_0})$ and
$(FH^4(W_1, M_1), \lambda_{W_1}, p_{W_1})$ are orthogonal in
$(FH^4(X), \lambda_X, p_X)$ and also that the induced isomorphism
$F_{\lambda_X} : H^4(M_0) \to H^4(M_1)$ is precisely $(f^*)^{-1}$.
Together with Lemma \ref{lem:splitq}, this implies
that $q^\circ_M$ is independent of the choice of $W$ and natural under
homeo\-morphisms (in the sense that $(f^*)^\# q_{M_0}^\circ = q_{M_1}^\circ$ for
any homeomorphism $f \colon M_0 \to M_1$).

\begin{rmk} \label{rmk:t_invariant}
If $d_\pi = 0$ %
then by \cite[Definition 1.4 and Theorem~2.4]{crowley13},
the function $q_M$ can be defined analytically using the eta invariant of a 
Dirac operator on $M$, twisted by appropriate quaternionic line bundles.
This definition is intrinsic to $M$, in the sense that no co-boundary is required.
For an alternative intrinsic definition of $q^\circ_M$ in the case of
$2$-connected $M$, see Section \ref{ss:intrinsic} below.
\end{rmk}

\begin{rmk} \label{rmk:topological_classification}
The proof following Definition \ref{def:qlf} that $q^\circ_M$
is a homeomorphism invariant relies on Remark \ref{rmk:topological_p_M} and
Lemma \ref{lem:p_M}.
It is simpler than the proof given 
in \cite[Theorem 6.1]{crowley02} which used the full apparatus of smoothing theory.

Notice, however, that smoothing theory and Theorem \ref{thm:a_class} imply
that every $2$-connected $M$ with $H^2(M; \Z/2) \neq 0$ admits exotic self-homeomorphisms;
by which we mean homeo\-morphisms 
which are not isotopic to piecewise linear homeomorphisms.
Self-homotopy equivalences which are homotopic to exotic self-homeomorphisms were defined on certain rational
homotopy spheres in \cite[\S 2.b]{crowley13}, see \cite[Lemma 2.17]{crowley13}.
\end{rmk}

\begin{rmk}
\label{rmk:projections}

Given a section $\sigma \colon G/T \to G$ of the projection $\pi \colon G \to G/T$,
the image of $\sigma$ is isomorphic to the free part of $G$, and there is a
unique $k(\sigma) \in \divsp \cap \im(\sigma)$.
We can therefore define the family of quadratic refinements as a
function on the set of sections $\sect(\pi)$ of $\pi$ so that
$q^\bullet \colon \sect(\pi) \to \cal{Q}(b),~q^{\sigma} := q^{e_\pi k(\sigma)}$.
This presentation is relevant for considering connected-sum
splittings of $M$ and is discussed further
in Section \ref{ss:alsmost_classification}.
\end{rmk}

\subsection{Gauss refinements}\label{ss:gr}
We can associate a further boundary invariant to a characteristic form
which we refer to as a Gauss refinement of the family of quadratic refinements.
Let $(G, b, p) \in \catb$, \ie $G$ is a finitely generated abelian group,
$p \in 2G$ and $b \colon T \times T \to \Q/\Z$ is a torsion form.
Let $\pi \colon G \to G/T$ be the projection 
and define $d_\pi$ to be the greatest integer dividing $\pi(p)$ if
$\pi(p) \neq 0$ and set $d_\pi : = 0$ if $\pi(p) = 0$.
If $d_\pi \neq 0$, we define
\begin{equation*}
\divsp := \{ k \in G : p - d_\pi k \in T \} 
\end{equation*}
and if $d_\pi = 0$ set $\divsp := T$.
Given $k \in \divsp$ write $\beta_k := p - d_\pi k$
and note that $T$ acts simply transitively on $\divsp$ by addition.
As in the introduction, we abbreviate $d_\pi/2$ as $e_\pi$.

Given $(G,q^\circ\!,p) \in \catq$, \ie if $q^\circ$ is a family of
quadratic refinements of $(G,b, p)$,
let us define $\tdelta(k, t) \in \Q/2\td_\pi \Z$ by
\begin{equation}
\label{eq:tdelta}
\tdelta(k,t) = 4d_\pi q^{e_\pi k}(t) - d_\pi(d_\pi {+} 2) \, b(t,t)
\end{equation}
(note that if $k \in \divsp$ then $e_\pi k \in \divs_2$,
so $q^{e_\pi k}$ is a well-defined quadratic refinement of $b$).

\begin{defn}
Given $(G,q^\circ\!,p) \in \catq$, we call a function
$\gr \colon \divsp \to \Q/\tdf \Z$ a \emph{Gauss refinement} of $q^\circ$ if
\begin{subequations}
\begin{equation}
\label{eq:gr_def}
\gr(k) = \Arf(q^{e_\pi k}) \mod \bbz
\end{equation}
for all $k \in \divsp$, and the transformation rule
\begin{equation} \label{eq:gr_rule}
\gr(k+t) - \gr(k) \; = \; \frac{\tdelta(k, t)}{8}
\end{equation}
\end{subequations}
holds for all $k \in \divsp$ and $t \in T$.
\end{defn}

A Gauss refinement is completely determined by its value at any single
$k \in \divsp$, using \eqref{eq:gr_rule}.
The difference between two Gauss refinements of the same family of quadratic
refinements is a constant, and by \eqref{eq:gr_def} the constant takes values
in $\Z/\tdf\Z$.

Now suppose that $(H, \lambda, \alpha)$ is a characteristic form.
Given $k \in \divsp$, pick $n \in H^*$ such that $j(n) = k$,
and set %
$\alpha_n := \alpha - d_\pi n$.
Note that $j(\alpha_n) = \beta_k$, and that $\alpha_n \in R \cong (H/K)^*$ is a
characteristic element for the intersection form on $H/K$. Let
\begin{equation}
\label{eq:grchat}
 \gr_H(k) := \frac{\rif(\alpha_n, \alpha_n) - \sign(\lambda)}{8}
\in \Q/\tdf \Z . 
\end{equation}

\begin{lem}
\label{lem:gr_well}
$\gr_H$ is well-defined, independent of the choices of $n$.
\end{lem}

\begin{proof}
Replacing $n$ by $n' := n + \hlambda(r)$ for some $r \in H^*$,
so $\alpha_{n'} = \alpha_n  -d_\pi \hlambda(r)$,
changes the value of $\gr_H(k)$ by
\[ \frac{-2d_\pi \rif\bigl( \alpha_n, \hlambda(r) \bigr) +
d_\pi^2 \rif \bigl( \hlambda(r), \hlambda(r) \bigr)}{8} = \frac{-d_\pi}{4}
\left(\gen{r, \alpha_n} - \frac{d_\pi}{2}\lambda(r, r)\right) . \]
The last factor is an integer, and it is even when $d_\pi$ is not divisible
by 4 (\ie when $\td_\pi = 2d_\pi$) because $\alpha_n$ is characteristic
for $\lambda$.
\end{proof}

\begin{lem}
$\gr_H$ is a Gauss refinement of $q^\circ_{(H, \lambda, \alpha)}$.
\end{lem}

\begin{proof}
First we check the condition \eqref{eq:gr_def}.
The $\alpha_n$ used in the definition of $\gr_H(k)$ co-incides with the
$\alpha_m$ used in the definition of $q^{e_\pi k}$ in
\eqref{eq:alg_qdef}.
Since $\alpha_n$ is characteristic for $\lambda$, 
Proposition \ref{prop:arf} immediately gives \eqref{eq:gr_def}.

Next we check the transformation law \eqref{eq:gr_rule}.
Given $k \in \divsp$ and $t \in T$, pick $n'$ such that $j(n') = k+t$.
Then $\alpha_{n'} - \alpha_n = -d_\pi(n'-n)$, and $j(n'{-}n) = t$, so
\begin{align*}
&\; \rif(\alpha_{n'}, \alpha_{n'}) - \rif(\alpha_n, \alpha_n)
\; = \;
-2d_\pi \rif(\alpha_n, n' {-} n) + 
d_\pi^2 \rif(n' {-} n, n'{-} n) \\
=& \; -4d_\pi\frac{
\rif(n' {-} n, n'{-} n) + \rif(\alpha_n, n' {-} n)}{2} + 
d_\pi(d_\pi+2)\rif(n' {-} n, n'{-} n) \\
=& \;\tdelta(k, t) \mod 2\td_\pi . \qedhere
\end{align*}
\end{proof}

An isomorphism $F \colon G' \to G$ with $F(p') = p$ maps $\divsp' \to \divsp$.
If $F^\# q^\circ = q'^\circ$ then we have $\Delta(F(k), F(t)) = \Delta'(k, t)$
for all $k \in \divsp'$ and $t \in T'$, so if
$\gr \colon \divsp \to \Q/\frac{\td_\pi}{4}\Z$
is a Gauss refinement of $q^\circ$ then
\begin{equation*}
F^\# \gr := \gr \circ F
\end{equation*}
is a Gauss refinement of $q'^\circ$.

\medskip

\subsubsection{Gauss refinements of orthogonal characteristic forms}
We recall that if $(H, \lambda, \alpha)$ is non\-singular and $H_0 \subset H$
is primitive with orthogonal complement $H_1$, then for the characteristic
forms $(H_i, \lambda_i, \alpha_i)$ defined by restriction from
$(\pm\lambda, \alpha)$, there is a canonical isomorphism 
$F_\lambda \colon \del(H_0, \lambda_0, \alpha_0)
\cong \del(H_1, \lambda_1, \alpha_1)$
of the associated %
refinements.

\begin{lem}
\label{lem:alg_compare}
Let $(H_0, \lambda_0, \alpha_0)$ and $(H_1, \lambda_1, \alpha_1)$ 
be orthogonal characteristic forms in the nonsingular
characteristic form $(H, \lambda, \alpha)$.  
The canonical isomorphism $F_\lambda$ of \eqref{eq:canonical_iso} 
pulls back $\gr_{H_1}$ to a Gauss refinement of the linking family $q_0^\circ$
of $(H_0, \lambda_0, \alpha_0)$,
and
\[  F_\lambda^\#\gr_{H_1} - \gr_{H_0} \; = \;
\frac{\rif(\alpha,\alpha) - \sign(\lambda)}{8} \mod \tdf. \]
\end{lem}

\begin{proof}
Note that since $(H, \lambda)$ is nonsingular,
$(\wh \lambda)^{-1} \colon H^* \cong H$
is an isomorphism from $\rif$ to~$\lambda$.
Also, we have homomorphisms $H^* \to H_i^* \to G_i$ where we recall that
$G_i = \coker(\wh \lambda_i)$.
Pick a $k \in \divsp(G_0)$, and then pick $n \in H^*$ whose image in
$G_0$ equals $k$ (the set-up means that the image of $n$ in $G_1$ is $F_\lambda(k)$).
Let $n_i$ be the image of $n$ in $H_i^*$, and set
$\alpha_{n_i} := \alpha_i - d_\pi n_i \in R_i$
as in the definition of~$g_{H_i}$. Since
$\sigma(\lambda) = \sigma(\lambda_1)- \sigma(\lambda_0)$, 
it suffices to show that
\begin{equation}
\rif(\alpha,\alpha) =
\rif_1(\alpha_{n_1}, \alpha_{n_1})
- \rif_0(\alpha_{n_0}, \alpha_{n_0})
\mod 2\td_\pi .
\end{equation}

The image of $\alpha - d_\pi n$ in $H^* \otimes \Q$ can be written as a sum
$\wh \lambda_0(\gamma_0) + \wh \lambda_1(\gamma_1)$
where $\gamma_i \in H_i \otimes \Q$ and $\hlambda_i(\gamma_i) = \alpha_{n_i}$.
Thus, since the hypothesis involves $\lambda$ restricting to $\lambda_1$
on $H_1$ and $-\lambda_0$ on~$H_0$, 
\begin{multline*} \rif(\alpha, \alpha)
= \lambda_1(\gamma_1, \gamma_1) - \lambda_0(\gamma_0, \gamma_0) + 2d_\pi \rif(n, \alpha - d_\pi n) 
+ d_\pi^2 \rif(n, n) \\
= \rif_1(\alpha_{n_1}, \alpha_{n_1})
- \rif_0(\alpha_{n_0}, \alpha_{n_0})
\mod 2\td_\pi ; 
\end{multline*}
that equality holds mod $4d_\pi$ when $d_\pi$ is not divisible by 4 follows
from $\alpha - d_\pi n$ being a characteristic element for $\lambda$.
\end{proof}

\begin{rmk} \label{rmk:neutral_forms}
Let us call a characteristic form $(H,\lambda,\alpha)$ \emph{neutral} if it is
nonsingular and $\lambda(\alpha,\alpha) = \sign(\lambda)$
and say that two
characteristic forms are \emph{neutrally isomorphic} if they become isomorphic
after addition of neutral forms (so this is a sharper condition than stable
isomorphism).
Lemma \ref{lem:alg_compare} implies that Gauss refinements are invariant
under neutral isomorphism. 
The gluing and splitting arguments for characteristic forms reviewed in
Section~\ref{ss:alsmost_classification},
in particular Theorem \ref{thm:stable_iso_and_qlfs},
can be used to show that characteristic forms are classified up to neutral
isomorphism by their boundary distillations $(G, q^\circ\!, \gr, p)$.
\end{rmk}

\subsubsection{Linked functions}
There is a certain redundancy in the definition of a Gauss refinement~$\gr$,
in that the constraint \eqref{eq:gr_def} on $\gr \mmod \Z$ forces
the transformation rule \eqref{eq:gr_rule} to hold mod~$\Z$.
In the analysis of the action of automorphisms on Gauss refinements
in \S\ref{sec:auto}, it will prove convenient to replace \eqref{eq:gr_rule}
with a condition that can be expressed purely in terms of the base
$(G, b, p)$ rather than the refinement $(G, q^\circ, p)$, but nevertheless
implies \eqref{eq:gr_rule} when \eqref{eq:gr_def} is assumed.

We call a function
$\lf \colon \divsp \to \Q/\frac{d_\pi}{4} \Z$ \emph{$(b,p)$-linked} if
for all $k \in \divsp$ and $t \in T$
\begin{equation}
\label{eq:lf_def}
\lf(k + t) = \lf(k) + \frac{\Delta(k, t)}{8} ,
\end{equation}
where
\begin{equation} \label{eq:Delta}
\Delta(k, t) := - d_\pi^2 b(t, t) + 2d_\pi b(\beta_k, \, t) \in \Q/2d_\pi \Z . 	
\end{equation} 

\begin{lem}
\label{lem:gr_via_linked}
$\gr : \divsp \to \Q/2\td_\pi\Z$ is a Gauss refinement of
$(G, q^\circ, p)$ if and only if \eqref{eq:gr_def} holds for some 
$k \in \divsp$, and the mod $2d_\pi$ reduction of $g$ is
$(b,p)$-linked.
\end{lem}

\begin{proof}
For $q^{e_\pi k}$ to be a refinement of $b$ with inhomogeneity $\beta_k$
implies from the definitions that
\[ 2q^{e_\pi k}(t) = b(t, t) + b(\beta_k, \, t) \in \Q/\Z , \]
which in turn gives
\[ \tdelta(k, t) = \Delta(k, t) \mod 2d_\pi . \]
Similarly, combining
\[ q^{e_\pi k}(2t) = 2b(t, t) + b(\beta_k, \, t)  \in \Q/\Z \]
and
\[ q^{e_\pi k}(-e_\pi t) = -e_\pi q^{e_\pi k}(t)
+ {\textstyle {e_\pi + 1 \choose 2}}\, b(t,t)  \in \Q/\Z \]
gives that
\[ \tdelta(k, t) = -8 q^{e_\pi k} (-e_\pi t) \mod 8 . \]
By the Chinese remainder theorem, these two constraints completely
characterise $\tdelta(k,t)$ as an element of $\Q/2\td_\pi\Z$.
Now observe that
\[ \Arf(q^{e_\pi (k+t)}) \; = \; \Arf(q^{e_\pi k}_{-e_\pi t})
\; = \; \Arf(q^{e_\pi k}) - q^{e_\pi k}(-e_\pi t) \in \Q/\Z\]
by \ref{def:foqr}\ref{it:q_trans} and \eqref{eq:arf_shift}.
Thus \eqref{eq:gr_rule} is equivalent to requiring that
$g(k) - \Arf(q^{e_\pi k}) \mmod \Z$ is constant and that
\eqref{eq:lf_def} holds.
\end{proof}

\begin{rmk}
We could make an analogy with factors of automorphy of automorphic forms and
think of $\Delta$ as a ``term of automorphy''. For any linked functions to
exist is clearly equivalent to the cocycle condition
\begin{equation}
\label{eq:cocycle}
\Delta(k, s+t) = \Delta(k+s, t) + \Delta(k, s) ,
\end{equation}
which can be checked directly from the definition in \eqref{eq:lf_def}. 
The difference of two
functions with the same term of automorphy is invariant under the $T$ action;
since $T$ acts transitively on $\divsp$ that simply means that the difference
between two $(p,b)$-linked functions is a constant in $\Q/2d_\pi\Z$.
\end{rmk}

\subsection{The generalised Eells--Kuiper invariant} \label{ss:eek}
Let $M$ be a spin 7-manifold and $W$ a \mbox{3-connected} coboundary of $M$.
Let $\gr_W : \divsp \to \Q/\tdf\Z$ be the Gauss refinement of
$(H^4(M), q_M^\circ, p_M)$ defined by
the characteristic form $(FH^4(W, \del W), \lambda_W, p_W)$. 
Applying \eqref{eq:grchat}, this means that
for $n \in H^4(W)$ such that $j(n) \in \divsp \subseteq H^4(M)$,
\begin{equation} \label{eq:g_W-explicit}
g_W(j(n)) =
\frac{\rif_W(\alpha_n, \alpha_n) - \sigma(\lambda_W)}{8}
= \frac{(p_W - d_\pi n)^2 - \sigma(W)}{8},
\end{equation}
as defined in \eqref{eq:grW} in the introduction.
We pointed out before that if 
$f \colon M_0 \to M_1$ is a spin homeomorphism 
then $X := (-W_0) \cup_f W_1$ is a closed topological spin 8-manifold, %
Lemma \ref{lem:split_mfd} means that 
$(FH^4(W_0, M_0), \lambda_{W_0}, p_{W_0})$ and
$(FH^4(W_1, M_1), \lambda_{W_1}, p_{W_1})$ are orthogonal in
the nonsingular form $(FH^4(X), \lambda_X, p_X)$, and the induced isomorphism
${F_{\lambda_X} : H^4(M_0) \to H^4(M_1)}$ is precisely $(f^*)^{-1}$.
Hence Lemma \ref{lem:alg_compare} implies
\begin{equation}
\label{eq:compare}
 \gr_{W_1} - (f^*)^\# \gr_{W_0}  = \frac{p_X^2 - \sign(X)}{8} \mod \tdf \Z .
\end{equation}
If $f$ is a diffeomorphism then $X$ is smooth, the RHS of \eqref{eq:compare}
equals $28\ahat(X)$, and $\ahat(X)$ is an integer;
this proves Lemma \ref{lem:ek_defined}. Letting
\begin{equation}
\dM := \gcd \left( \textstyle \frac{\wt d_\pi}{4}, 28 \right)
\end{equation}
as in the introduction, it follows that
\begin{equation}
\label{eq:EK_of_M}
\begin{aligned}
\EK_M \colon \divsp \to \Q/\dM \Z, \\
\EK_M := \gr_W \mmod \dM
\end{aligned}
\end{equation}
is independent of the choice of $W$ and natural under diffeomorphisms:
If $f \colon M_0 \to M_1$ is a 
diffeomorphism then
$(f^*)^\# \EK_{M_0} = \EK_{M_1}$ 
by \eqref{eq:compare}.
Now $\EK_M$ satisfies a transformation rule that is a mod 28 reduction of
\eqref{eq:gr_rule}, and we say that this 
makes $\EK_M$ a mod 28 Gauss refinement of $q_M^\circ$.

\begin{defn}
\label{def:modngr}
Given $(G,q^\circ\!,p) \in \catq$ and a positive integer $N$, we call a function
$\EK \colon \divsp \to \Q/\gcd(\tdf,N) \Z$ a \emph{mod $N$ Gauss refinement} of
$q^\circ$ if
\[
\EK(k) = \Arf(q^{e_\pi k}) \mod \bbz
\]
for all $k \in \divsp$, and the transformation rule
\[
\EK(k+t) - \EK(k) \; = \; \frac{\tdelta(k, t)}{8} \mod \gcd(\tdf, N) 
\]
holds for all $k \in \divsp$ and $t \in T$.
\end{defn}

For $N = 28$ this transformation rule is equivalent to \eqref{eq:TR} stated in
the introduction.

If $W$ is not 3-connected, then \eqref{eq:grW} defines $g_W$ only on the
subset $\divsp \cap j(H^4(W))$ of~$\divsp$. However, as long as
that set is non-empty, this completely determines $\EK_M$ by the transformation
rule, so the description of $\EK_M$ from the introduction is valid.
This point can be seen as a special case of Proposition \ref{prop:ek_via_spc} below.

\begin{rmk} \label{rmk:splitting_definition_of_eek}
Analogously to Remark \ref{rmk:projections},
we can define Gauss refinements (and $\EK_M$) as functions of sections $\sigma \colon G/T \to G$
rather than on $\divsp$,
$\gr_W(\sigma) := \gr_W(k(\sigma))$.
Then $\gr_W(\sigma) = \Arf(q^\sigma) \mmod \bbz$, and the transformation
rule \eqref{eq:lf_def} can also be rewritten in these terms.
\end{rmk}

\begin{rmk} \label{rmk:analytic_eek}
Recall Remark \ref{rmk:original_ek} saying that if $p_M$ is torsion then
$\divsp = T$ contains the distinguished element $0$ and 
$\frac{1}{28}\EK_M(0) \in \Q/\Z$ recovers the original Eells-Kuiper
invariant $\mu(M)$.

Although defined extrinsically using spin co-boundaries, the original Eells-Kuiper invariant $\mu(M)$ 
was shown by Donnelly \cite[Theorem 4.2]{donnelly75} to have an intrinsic definition in terms of 
the eta invariant of the Dirac operator of $M$.  
It would be interesting to find an intrinsic definition of the generalised
Eells-Kuiper invariant when $p_M \neq 0 \in H^4(M; \Q)$.

For further information about the role of eta invariants in the classification of $7$-manifolds,
we refer the reader to \cite[\S 4]{goette12}. 
\end{rmk}

\begin{rmk} \label{rmk:neutral_forms_mod_224}
In \cite[\S 4.4]{crowley02}, a pair of characteristic forms
$(H, \lambda, \alpha)$ are called \emph{smoothly equivalent}
if they become isomorphic after addition of nonsingular characteristic forms
with  $\lambda(\alpha,\alpha) \equiv \sign(\lambda)$~mod~$224$
(so this is a weakening of the notion of neutral equivalence from
Remark \ref{rmk:neutral_forms}).
In algebraic terms, the definition of the generalised Eells-Kuiper invariant
can be used to show that the mod $28$ distillation %
of $M$, $(H^4(M), q^\circ_M, \EK_M, p_M)$,
is a complete invariant of the smooth equivalence class of the characteristic
form $(FH^4(W, \del W), \lambda_W, p_W)$ of a $3$-connected coboundary for $M$.
Hence Theorem \ref{thm:class} is a development of the dimension 7 case of 
\cite[Theorem 4.9]{crowley02}, which classifies $2$-connected
$7$-manifolds up to diffeomorphism by the smooth equivalence
class of the characteristic form of a $3$-connected coboundary.
\end{rmk}

\begin{rmk}
\label{rmk:maxinfo}
Let us conclude this subsection by considering how the information captured
by the function $\EK_M : \divsp \to \Q/\tdf\Z$ can in some special cases be
presented more simply. If $p_M$ is torsion or if the greatest divisor of $p_M$
is the same as $d_\pi$ (the greatest divisor modulo torsion), then
$\divsp$ contains the distinguished element 0, and the function $\EK_M$ can be 
naturally identified by the value $\EK_M(0) \in \Q/\tdf \Z$.

More generally, for any divisor $\dv$ of $d_\pi$
we can relate $\EK_M$ to functions defined on
${\divsdv = \{ k \in G : p_M - \dv k \in T \}}$.
Let us focus on the case when $\dv$ is even---because that is more subtle than
when $c$ is odd---and let $\tdv = \lcm(4,\dv)$.
We can then define a function
$\bar \gr_W : \divsdv \to \Q/\tdvf \Z$ analogously to $\gr_W$.
If $d_\pi = r\dv$, then $r\divsp$ is a non-empty subset of $\divsdv$.
For any $k \in \divsp$, the mod $\tdvf$ reduction of $\gr_W(k)$ equals
$\bar \gr_W(rk)$.

Thus the mod $\tdvf$ reduction of $\gr_W$ is completely
determined by $\bar \gr_W$. In particular, if we take $\dv = \gcd(28, d_\pi)$,
then $\bar \gr_W$ determines $\EK_M$.
Meanwhile $\bar \gr_W$ can sometimes be
easier to describe.

In particular, if $\dv$ divides $p_M$, then $\divsdv$ contains 0, and the
function $\bar \gr_W$ can be naturally identified with its value at 0.
In fact more is true: $\bar \gr_W$ must be constant, except when $\mdiv$
is an odd multiple of $\dv$ \emph{and} the
parameter $r$ from Theorem \ref{thm:i_and_r} is 0.
That $\dv$ divides $p_M$ means that the
image of $p_W$ in $H^4(W; \Z_\dv)$ is contained in the image
of $H^4(W,M; \Z_\dv)$, and thus has a well-defined square in
$H^8(W,M; \Z_\dv) \cong \Z_\dv$.
The mod $\tfrac{\tdv}{8}$ reduction of $\bar \gr_W$ is always constant,
determined by
\[ \bar \gr_W = \frac{p_W^2 - \sign(W)}{8} \mod \tfrac{\dv}{8} . \]
One can attempt to compute $\bar \gr_W$ itself in a similar way using
the Pontrjagin square $\wp(\bar p) \in H^8(W,M;\Z_{2c})$ of a
pre-image $\bar p \in H^4(W,M;\Z_c)$ of $p_W$. This is independent of the
choice of $\bar p$ if and only if $\wp(\partial x) = 0$ for
all $x \in H^3(M;\Z_c)$.
Because the suspension of the Pontrjagin square is the Postnikov square,
that is equivalent to requiring that, for $j = \ord_2 c$, there are
no $2^j$-torsion classes $y \in H^4(M)$ with $2^jb(y,y)$ odd.
Thus---in the terminology of \S\ref{ss:auto_notation}---if there are no split
$2^j$-torsion elements in $H^4(M)$ then there is a well-defined Pontrjagin
square $\wp(p_W) \in H^8(W,M; \Z_{2\dv}) \cong \Z_{2\dv}$,
and $\bar \gr_W$ is determined by
\[ \bar \gr_W = \frac{\wp(p_W) - \sign(W)}{8} \mod \tfrac{\dv}{4} . \]
(This is compatible with the claim above that $\bar \gr_W$ could be
non-constant if $r = 0$, because Lemma \ref{lem:autb} means that if $H^4(M)$
lacks certain split summands then $r \not= 0$.)

On the other hand, for any divisor $\dv$ of $d_\pi$, $\bar \gr_W$ is determined
by its value at a single element of~$\divsdv$, so $\bar \gr_W$ is completely
determined by $\gr_W$. Regardless of the possible convenience in some special
settings of considering divisors $\dv$ other than $d_\pi$, using $d_\pi$ 
captures the maximal possible amount of information. For the purposes of
studying the general classification theory there is thus no advantage to
considering anything other than $\gr_W$ and $\EK_M$ as functions of~$\divsp$,
and that is therefore what we do in the rest of the paper.
\end{rmk}

\subsection{The computation of \texorpdfstring{$\EK_M$}{the Eells--Kuiper invariant} via \texorpdfstring{\spc}{spin-c} coboundaries} \label{ss:spin_c}
Inspired by calculations of Kreck and Stolz for their $s_1$ invariant \cite{kreck91}, 
we derive an expression for $\EK_M$ in terms of coboundaries that are not
necessarily spin (never mind 3-connected) but just \spc.

For a principal \spc bundle we use the canonical homomorphisms
$\spcg(n) \to SO(n)$ and $\spcg(n) \to U(1)$ to define an associated real
vector bundle $E$ together
with a complex line bundle $L$ such that $c_1(L) = w_2(E) \mmod 2$.
We can then define the characteristic classes
\begin{align*}
z & := c_1(L) , \\
\pc & := p(E \oplus L) , \\
\pd & := \pc - z^2 .
\end{align*}
So $2\pc = p_1(E \oplus L) = p_1(E) + z^2$ and $2\pd = p_1(E) - z^2$.
Recall that any $U$-bundle has a natural \spc structure, defined as follows:
if $i : U(n) \to SO(2n)$ is the natural inclusion then the homomorphism
$i \times \det : U(n) \to SO(2n) \times U(1)$
has a lift under the double cover $\spcg(2n) \to SO(2n) \times U(1)$.
If $E$ is a complex vector bundle then the fundamental line bundle $L$
of the corresponding \spc bundle is $L := \det E$.

\pagebreak[2]

\begin{lem}\hfill \label{lem:pz}
\begin{enumerate}
\item \label{it:pz_basis}
$\pd$ and $z^2$ form a basis for $H^4(B\spcg)$.
\item \label{it:cx}
$\pd(E) = -c_2(E)$ for any complex bundle $E$.
\item \label{it:w}
$\pd(E) = w_4(E) \mmod 2 $ for any \spc bundle $E$.
\end{enumerate}
\end{lem}

\begin{proof}
Observe that $\spcg/U \cong \spg/SU$ is $5$-connected, since
$\spg(6)/SU(3) \cong S^7$ and $\spg(6)$ and $SU(3)$ have the same homotopy
groups as $\spg$ and $SU$ in degree $\leq 5$. Letting $\pi : BU \to B\spcg$
denote the classifying map for $EU$ considered as a \spc-bundle,
the maps $\pi^* \colon H^4(B\spcg) \to H^4(BU)$ and
$\pi^* \colon H^2(B\spcg) \to H^2(BU)$ are therefore isomorphisms
(with both $\bbz$ and $\bbz_2$ coefficients).
Patently $\pi^* z = c_1$.

We know that $H^4(BU)$ has basis $\{c_2, c_1^2\}$. 
Because there is no 2-torsion,
the equation $2\pi^* \pd  = p_1 - (\pi^*z)^2 = (-2c_2 + c_1^2) - c_1^2$
implies $\pi^* \pd = -c_2$, proving \ref{it:pz_basis} and \ref{it:cx}.

The isomorphism on $H^4(-;\bbz_2)$ implies that it suffices to check that
\ref{it:w} holds when $E$ is complex. But that follows from \ref{it:cx}.
\end{proof}

\begin{cor}
If $X$ is a compact \spc 8-manifold then
$\pc_X$ is characteristic for the intersection form $\lambda_X$ of $X$.
\end{cor}

\begin{proof}
Lemma \ref{lem:pz} gives
\[ \pc = w_4 + w_2^2 \mod 2 . \]
Wu's formula implies that for any closed orientable
manifold $X$ the fourth Wu class is $v_4(X) = w_4(X) + w_2(X)^2$,
and by definition $v_n$ is characteristic for the intersection form of a closed
$2n$-manifold.
The compact case follows from the closed case, as in the proof of
Lemma \ref{lem:p_M}\ref{it:boundary}.
\end{proof}

\begin{lem}
\label{lem:ind_L}
If $X$ is a closed \spc 8-manifold then the Dirac operator
of the fundamental complex spinor bundle has
\begin{equation}
\label{eq:ind_L}
28 \ind \dirac^+
\; = \; \frac{\pc_X^2 - \sign(X)}{8} - \frac{5z^2\pc_X}{12} + \frac{z^4}{4} .
\end{equation}
\end{lem}


\begin{proof}
\cite[Theorem D.15]{lawson89} expresses $\ind \dirac^+$ as
the integral of $\exp\left(\frac{z}{2}\right) \ahat(X)$, whose degree 8 part expands to
\[ \frac{-4p_2 + 7p_1^2}{2^7.45} - \frac{z^2 p_1}{24.8} + \frac{z^4}{24.16}
\; = \; \frac{p_1^2}{2^7.7}  - \frac{L}{2^5.7} - \frac{z^2 p_1}{2^6.3} +
\frac{z^4}{2^7.3} . \]
Then substitute $p_1 = 2\pc - z^2$ to obtain \eqref{eq:ind_L}.
\end{proof}

Now suppose $M$ is a spin 7-manifold and $W$ a \spc coboundary,
such that the restriction of $z \in H^2(W)$ to $M$ is trivial. Then $z$ has a
pre-image $\bar z \in H^2_{}(W, M)$, and $\bar z^2 \in H^4(W, M)$ is
independent of the choice of $\bar z$.

\begin{defn}
\label{def:grc}
Given $k \in \divsp$, suppose there is $n \in H^4(W)$ such that
$j(n) = k$. Then let $\wh\alpha_n := \pc_W - d_\pi n$, and
\[ \gr^c_W(k) :=
\frac{\rif_W(\wh\alpha_n, \wh\alpha_n) - \sign(W)}{8}
- \frac{5\bar z^2 \pc_W}{12} + \frac{z^4}{4} \in \Q/\tdf \Z . \]
\end{defn}

If $z = 0$ then of course $\gr^c_W = \gr_W$. The proof that $\gr^c_W(k)$ does
not depend on the choice of $n$ is analogous to Lemma \ref{lem:gr_well},
using that $\wh\alpha$ is characteristic for intersection form $\lambda_W$.

\enlargethispage{\baselineskip}

\begin{prop}
\label{prop:ek_via_spc}
Let $(W_1,z_1)$ be a \spc coboundary of $M$ and $j_1 \colon H^4(W_1) \to H^4(M)$
the natural homomorphism.
Then
\begin{enumerate}
\item \label{it:grc}
$\gr^c_{W_1}(k) = \EK_M(k)\mmod 28$ for all $k\in S_{d_\pi} \cap j_1(H^4(W_1))$;
\item \label{it:grc_transform}
The defined values of $\gr^c_{W_1}$ satisfy the transformation rule
\eqref{eq:gr_rule}, \ie
\[ \gr^c_{W_1}(k') = \gr^c_{W_1}(k) + \frac{\tdelta(k, k'{-}k)}{8} \]
whenever $k, k' \in S_{d_\pi} \cap j_1(H^4(W_1))$, where $\tdelta$ in
\eqref{eq:tdelta} is defined in terms of $q_M^\circ$ and~$b_M$.
\end{enumerate}
\end{prop}

\begin{proof}
For part \ref{it:grc}, let $W_0$ be a 3-connected coboundary for $M$,
and ${X := (-W_0) \cup_{\Id_M} W_1}$.
Then $X$ is a smooth \spc manifold, possibly with more than one choice
of $z \in H^2(X)$ restricting to $z_1$ on $W_1$ and $0$ on $W_0$.
While we do not trouble ourselves with
separating the algebra from the topology in this case, we essentially
adapt the proof of Lemma \ref{lem:alg_compare} to show
\begin{equation}
\label{eq:compare^c}
\gr^c_{W_1} - \gr_{W_0} =
\frac{\pc_X^2 - \sign(X)}{8} - \frac{5z^2\pc_X}{12} + \frac{z^4}{4}
\mod \tdf \Z .
\end{equation}
Since the RHS equals $28 \ind \dirac^+$ by
Lemma \ref{lem:ind_L}, while
$\gr_W = \EK_M \mmod \gcd(28,\tdf)$ by definition, the result then follows.

Pick some $n_1 \in H^4(W_1)$ such that $j_1(n_1) \in \divsp$ as in Definition
\ref{def:grc}.  As $W_0$ is \mbox{3-connected},
there is some $n \in H^4(X)$ whose restriction to $W_1$ equals $n_1$.
Then $\pc_X - d_\pi n$ is a sum
of push-forwards of $\gamma_i \in H^4_{}(W_i, M;\Q)$ and
$\hlambda_{W_i}(\gamma_i) = \wh\alpha_i$. Meanwhile, note that regardless of
the choice of $z$, $z^2 \in H^4(X)$ is the push-forward of
$\bar z_1^2 \in H^4_{}(W_1, M)$. Hence
\begin{align*}
 \pc_X^2 - \frac{10 z^2 \pc_X}{3} + 2z^4 
\; & = \; \gamma_1^2 + \gamma_0^2 + 2d_\pi n (\pc_X - d_\pi n) + d_\pi^2 n^2
- \frac{10}{3}\bar z_1^2 \pc_{W_1} +2z_1^4 \\
& = \; \rif_{W_1}(\wh\alpha_1, \wh\alpha_1)
- \frac{10}{3}\bar z_1^2 \pc_{W_1} + 2z_1^4
- \rif_{W_0}(\wh\alpha_0, \wh\alpha_0) 
\mod 2\td_\pi .
\end{align*}
The fact that the equality holds $\mmod 4d_\pi$ when $d_\pi$ is not divisible
by 4 is due to $\pc_X - d_\pi n$ being a characteristic element for the
intersection form on $X$.

Part \ref{it:grc_transform} follows from \eqref{eq:compare^c} since $\gr_{W_0}$
satisfies \eqref{eq:gr_rule} and the RHS of \eqref{eq:compare^c} is constant.
\end{proof}

As a consequence of Proposition \ref{prop:ek_via_spc}~\ref{it:grc_transform}, we can extend
$\gr^c_W$ to a well-defined Gauss refinement so long as
$S_{d_\pi} \cap j(H^4(W))$ is non-empty. Hence the generalised Eells-Kuiper
invariant $\EK_M$ can be computed in terms of any \spc coboundary $W$ of $M$
where the intersection $S_{d_\pi} \cap j(H^4(W))$ is non-empty.

\subsection{An intrinsic definition of \texorpdfstring{$q^\circ_M$}{the linking family}} \label{ss:intrinsic}
In this subsection we define $\spg\an{4, 2}$-structures on spin manifolds and use them
to give an intrinsic definition of the linking family $q^\circ_M$ for $2$-connected $M$.

Recall from the proof of Lemma \ref{lem:p_M} that the
mod~$2$ reduction of the universal spin class $\rho_2(p) \in H^4(B\spg; \Z/2)$
is identified with the $4$th Wu class of the universal bundle over $B\spg$.  
We regard $v_4$ as a map
$$ v_4 \colon B\spg \to K(\Z/2, 4) $$
and define $B\spg\an{4, 2}$ to be the homotopy fibre of $v_4$.  
By construction there is a map $\gamma^{\an{4, 2}} \colon B\spg\an{4, 2} \to B\spg$
and a sequence of maps 
\begin{equation*} 
K(\Z/2, 3) \longrightarrow B\spg\an{4, 2} \xra{~\gamma^{\an{4, 2}}~} B\spg \xra{~v_4~} K(\Z/2, 4),
\end{equation*}
where both successive pairs of maps defines a fibration sequence.
Let $N$ be a spin manifold and let $\nu_N \colon N \to B\spg$
the classifying map for the stable normal bundle of $N$.
A {\em $\spg\an{4, 2}$-structure} on $N$ is an vertical homotopy class of lift 
$\bar \nu_N \colon N \to B\spg$.  In particular, there is a commutative diagram
\[ \xymatrix{
&
B\spg\an{4, 2} \ar[d]^{\gamma^{\an{4, 2}}} \\
N \ar[r]^(0.45){\nu_N} \ar[ur]^{\bar \nu} &
B\spg.
}\]
The diagram above ensures that $\bar \nu \colon N \to B\spg\an{4, 2}$ 
is canonically covered by a map of stable vector bundles
from the normal bundle of $N$ to the pull-back of the universal 
bundle over $B\spg$ along $\gamma^{\an{4, 2}}$.

\begin{lem} \label{lem:Spin2}
\hfill
\begin{enumerate}
\item Every spin $7$-manifold $M$ admits a $\spg\an{4, 2}$-structure. 
\item The set of equivalence classes of $\spg\an{4, 2}$-structures on $M$
is a torsor for $H^3(M; \Z/2)$.
\item The induced map $\gamma^{\an{4, 2}*} \colon H^4(B\spg) \to H^4(B\spg\an{4, 2})$ is 
isomorphic to $\times 2 \colon \Z \to \Z$.
\end{enumerate}
\end{lem}

\begin{proof}
By Lemma \ref{lem:p_M}(i) we have $\rho_2(p_M) = 0$ and
so Part (i) follows from the right-hand fibration in 
the sequence of maps defining $B\spg\an{4, 2}$ above.
Part (ii) and (iii) follow from the left-hand fibration in 
the sequence of maps defining $B\spg\an{4, 2}$ above.
\end{proof}

For later use, we point out that Lemma \ref{lem:Spin2}(iii)
shows that $B\spg\an{4, 2}$-manifolds $(X, \bar \nu)$ have a naturally defined 
characteristic class $p^{\bar \nu}_X \in H^4(X)$ such that $2 p_X^{\bar \nu} = p_X$.
For spin $7$-manifolds $M$ we set
$$\bar \divs_2 := \{ h \in G : p_M = 2h \} \subset \divs_2$$
and then Lemma \ref{lem:Spin2}(ii) shows that every $h \in \bar \divs_2$ arises
as $p^{\bar \nu}$ for some $\spg\an{4, 2}$-structure $\bar \nu$ on $M$.
Of course, $\bar \divs_2$ is a torsor for $_{2}TH^4(M)$,
the subgroup of $2$-torsion elements of $TH^4(M)$.

In the usual way, we define the bordism groups of closed $n$-manifolds with
$\spg\an{4, 2}$-structure,
$$ \Omega_n^{\spg\an{4, 2}} = \{ [N, \bar \nu] \},$$
where $[N, \bar \nu]$ denotes the $B\spg\an{4, 2}$-bordism class of 
$(N, \bar \nu)$.

\begin{lem} \label{lem:OmegaSpin2}
$\Omega_7^{\spg\an{4, 2}}  = 0$;
\ie every closed spin 7-manifold has a spin coboundary $W$ with $p_W$ even.
\end{lem}

\begin{proof}
Consider a $B\spg\an{4, 2}$-manifold $\bar \nu\colon M \to B\spg\an{4, 2}$.
Since the space $B\spg\an{4, 2}$ is $3$-connected, 
surgery below the middle dimension as in by
\cite[Proposition]{kreck99} ensures that we may replace
$(M, \bar \nu)$
in its $B\spg\an{4, 2}$-bordism class by a homotopy sphere 
with $\spg\an{4, 2}$-structure $(\Sigma, \bar \nu_\Sigma)$.
By Lemma \ref{lem:Spin2}(ii), $\Sigma$ has a unique $\spg\an{4, 2}$-structure.
By \cite[Theorem 3.1]{kervaire63} $\Sigma$ is stably parallelisable and
and so its $\spg\an{4, 2}$-structure is induced from a stably framing.
By \cite{kervaire63} $\Sigma$ bounds a parallelisable manifold.
Hence $(\Sigma, \bar \nu_\Sigma)$ bounds a $B\spg\an{4, 2}$-manifold
and so $\Omega_7^{\spg\an{4, 2}}  = 0$.
\end{proof}

Fix a $\spg\an{4, 2}$-structure $\bar \nu$ on $M$ and
recall the characteristic class 
$p^{\bar \nu}_M \in H^4(M)$.
For $2$-connected $M$ we show how to define a homogeneous quadratic form
$$q^{\bar \nu} \colon TH^4(M) \to \Q/\Z $$
using just $(M, \bar \nu)$ and in particular no coboundary.
Moreover this definition recovers the quadratic form obtained 
by evaluating the quadratic linking family $q^\circ_M$ at
$h = p^{\bar \nu}_M \in \bar \divs_2$;
\ie $q^{\bar \nu} = q^{p^{\bar \nu}_M}$.
The idea is to repeat Wall's definition 
of the quadratic refinement of the linking form for $(s{-}1)$-connected $(2s{+}1)$-manifolds
for $s \neq 3, 7$ from \cite[\S12A]{wall67}.
We assume the reader is familiar with this definition,
recalling only its essential features.

Following Wall we work with the dual group $TH_3(M)$, the torsion subgroup of $H_3(M)$.
For brevity, we write $\wh x \in TH_3(M)$ 
for the Poincar\'{e} dual of $x \in TH^4(M)$.
Since $M$ is \mbox{$2$-connected} every element $\wh x \in TH_3(M)$ is represented by
an embedding $S^3 \hookrightarrow M$ and since every linear bundle over $S^3$ is trivial,
this extends to an embedding $f_{\wh x} \colon D^4 \times S^3 \hookrightarrow  M$.
To compute the self-linking number $b_M(x, x)$ we need to push $f_{\wh x}(\{0\} \times S^3)$
off itself and this can be achieved by taking a section $s \colon S^3 \to S^3 \times S^3$
of the unit normal bundle $S^3 \times S^3 \to S^3$.
Following Wall, 
we set $X : = M \setminus \mathrm{Int}(f_{\wh x}(D^4 \times S^3))$,
note that $M$ is obtained from $X$ by attaching a $4$-handle and a $7$-handle
and let $y_1 := [s] \in H_3(X)$ and $y_2 \in H_3(X)$ be the homology class
of the meridian $S^3 \times \{\ast\}$.
For $i \colon X \to M$ the inclusion,
$y_2$ generates the kernel of $i_* \colon H_3(X) \to H_3(M)$
and $i_*(y_1) = x$ has order $r$ for some positive integer~$r$.
Hence $i_*(ry_1) = 0$ and so $ry_1 = \lambda(s) y_2$ for $\lambda(s) \in \Z$.
The homological definition of the linking form gives
\[ b_M(x, x) = \frac{\lambda(s)}{r}.\]

Wall defined a refinement of $b_M$ by restricting the choice of section $s$,
and hence the possible integers $\lambda(s)$ appearing in the description 
of $b_M$ above.
To achieve a similar restriction on the choice of sections in dimension $7$ we use the 
$\spg\an{4, 2}$-structure $\bar \nu$ on $M$.
The codimension-$0$ submanifold $f_{\wh x}(D^4 \times S^3) \subset M$ inherits a $\spg\an{4, 2}$-structure from
$(M, \bar \nu)$ and this induces a $\spg\an{4, 2}$-structure
on $S^3 \times S^3$, which we also denote by $\bar \nu$.
By construction the universal bundle on $B\spg\an{4, 2}$ in Wu $4$-oriented in the sense
of Brown \cite[Definition~1.10]{brown72}.  
Now for any closed $6$-manifold $Y$ with a Wu $4$-orientation
$\bar \nu_Y$, Brown \cite[Corollary 1.11]{brown72} defines a quadratic refinement,
$\phi^{\bar \nu_Y} \colon H^3(Y; \Z/2) \to \Z/2$,
of the mod $2$-intersection form of~$Y$.
In particular, we have the quadratic refinement
$$ \phi^{\bar \nu_{}} \colon H^3(S^3 \times S^3; \Z/2) \to \Z/2.$$
We then define $q^{\bar \nu} \colon TH^4(M) \to \Q/\Z$ by the equation
$$ q^{\bar \nu}(x) := \frac{\lambda(s)}{2r} \in \Q/\Z,$$
where we restrict to sections $s \colon S^3 \to S^3 \times S^3$ such that $\phi^{\bar \nu_{}}(s^*(u)) = 0$
for $u \in H^3(S^3; \Z/2)$ the generator.

\begin{lem} \label{lem:q_intrinsic}
$q^{\bar \nu} \colon TH^4(M) \to \Q/\Z$
is well-defined and refines $b_M$.  Moreover 
$q^{\bar \nu} = q^{p^{\bar \nu}_M}$.
\end{lem}

\begin{proof}
That $q^{\bar \nu}$ is a well-defined refinement of $b_M$ 
follows from the proof of \cite[Lemma 26]{wall67}, using the fact 
that Brown's form is a refinement of the mod~$2$ intersection for a $6$-manifold.

To see that $q^{\bar \nu} = q^{p^{\bar \nu}_M}$, we let $(W, \bar \nu_W)$ be a 
$B\spg\an{4, 2}$-coboundary for $(M, \bar \nu)$, which exists by 
Lemma \ref{lem:OmegaSpin2}.
As for spin coboundaries, we may assume that $W$ is $3$-connected
and consider the characteristic form
$(H^4(W, \del W), \lambda_W, p_W)$ of $W$.  
In the definition of $q^{p^{\bar \nu}_M}$ in \eqref{eq:alg_qdef},
we may take $m = p^{\bar \nu_W}$ so that $\alpha_m = 0$.  Then for $x \in TH^4(M)$ and
$y \in H^4(W)$ with $j(y) = x$, we have
\begin{equation} \label{eq:qbarnu}
q^{p^{\bar \nu}_M}(x) = \frac{-\rif_W(y, y)}{2},
\end{equation}
where we note that $\lambda_W$ is even since $\rho_2(p_W) = 0$.
But in the proof of \cite[Theorem 8]{wall67} Wall identifies his topologically
defined refinement with the algebraically defined refinement
appearing in \eqref{eq:qbarnu}.
It follows that Wall's arguments in the proof of \cite[Theorem 8]{wall67} can be repeated
to show that $q^{\bar \nu} = q^{p^{\bar \nu}_M}$.
\end{proof}

\begin{rmk}
Using Lemma \ref{lem:q_intrinsic} we can define $q^h$ intrinsically on
$2$-connected $M$ for every $h \in \bar \divs_2 \subset \divs_2$
and then use the transformation rule of Definition \ref{def:foqr}(ii) to
determine~$q^\circ_M$.
For example, if $H^4(M)$ is torsion then
for each $\spg\an{4, 2}$-structure $\bar \nu$ on $M$ we have
\[ q_M = q^0_M = q^{\bar \nu}_{p^{\bar \nu}}.\]
\end{rmk}

\section{The classification of 2-connected 7-manifolds} \label{sec:classification}
In this section we classify closed smooth spin $2$-connected
$7$-manifolds $M$ up to diffeomorphism.
Recall that a homotopy $7$-sphere $\Sigma$ is a spin manifold which is homotopy equivalent to $S^7$.
In Section \ref{ss:almost_diffeomorphism} we recall that 
an almost diffeomorphism $f \colon M_0 \acong M_1$ defines a 
diffeomorphism $f \colon M_0 \sharp \Sigma \cong M_1$,
for some $\Sigma$.
In Section \ref{ss:alsmost_classification} we relate the algebra of Section \ref{sec:invariants} 
to the algebra used in \cite{crowley02} and so give the almost diffeomorphism classification of 
$2$-connected $M$ in terms of their refinements $(H^4(M), q^\circ_M, p_M)$.

With the almost diffeomorphism classification in hand, we 
consider the inertia group
of $M$, which is the group of homotopy spheres $\Sigma$ such that $M \sharp \Sigma \cong M$.
In Section \ref{ss:inertia_and_reactivity2} we establish basic
facts relating the inertia group of $M$, the reactivity of $M$ and certain mapping class groups of $M$.
We also construct an important family of almost diffeomorphisms $f \colon M \acong M$
in Proposition \ref{prop:p2null}.  
The almost diffeomorphisms of Proposition \ref{prop:p2null} allow us to show that 
the generalised Eells-Kuiper invariant of $M$,~$\mu_M$, precisely measures the gap between
the almost diffeomorphism classification and the diffeomorphism classification.  
In Section \ref{ss:classification_theorem} we prove that
the mod $28$ distillation %
of $M$, $(H^4(M), q^\circ_M, \EK_M, p_M)$, is a complete invariant
of diffeomorphisms.

%
\subsection{Almost diffeomorphisms}  
\label{ss:almost_diffeomorphism}
In this  subsection we briefly review the almost smooth spin category in dimension $7$.
An {\em almost diffeomorphism} $f \colon M_0 \acong M_1$
is a homeomorphism which is smooth except perhaps at a finite set of singular points 
$\{m_0, \dots, m_a\} \subset M_0$.  Notice that {\em we do not require $f$ to be non-smooth}
at $m_i$, but we rather {\em allow} it.
The composition of almost diffeomorphisms is again
an almost diffeomorphism and so almost diffeomorphism defines an 
equivalence relation on smooth spin $7$-manifolds.

Let $f \colon M_0 \acong M_1$ be an almost diffeomorphism with singular set 
$\{m_0, \dots, m_a\}$.  We shall associate a 
homotopy $7$-sphere to each singular point $m_i$.  
For $i = 0, \dots, a$, let $D^7_i \ni m_i$ be a small
disc containing $m_i$ and disjoint from $D^7_j$ if $i \neq j$. 
The manifold $f(D^7_i) \subset M_1$
is a co-dimension zero submanifold of $M_1$ and so inherits a smooth structure from $M_1$
such that
\[ \wh f_i : = f|_{D^7_i - \{m_i\}} \colon D^7_i - \{m_i\} \cong f(D^7_i - \{m_i\})  \]
is a diffeomorphism.  We can therefore define the smooth homotopy $7$-sphere
\[ \Sigma_i : = D^7_i \cup_{\wh f_i} (-f(D^7_i)) \]
by gluing $D^7_i$ and $-f(D^7_i)$ together along $\wh f_i$

We set $\Sigma_f : = \Sigma_0 \sharp \Sigma_1 \sharp \dots \sharp \Sigma_a.$
If $D^7 \subset M_0$ contains the singular points of $f$ 
in its interior, then by \cite[Proposition 2.1]{crowley02} 
there is a diffeomorphism 
$f' \colon M_0 \sharp \Sigma_f \to M_1$
such that $f'|_{M_0 - D^7} = f|_{M_0 - D^7}$.
It follows that $M_0$ is almost diffeomorphic to $M_1$
if and only if there is a homotopy sphere $\Sigma$ and a diffeomorphism
$M_0 \sharp \Sigma \cong M_1$.

Before defining pseudo-isotopy for almost diffeomorphisms with one singular point
we recall the definition for diffeomorphisms.
Let $\Diff_{}(M)$ be the group of diffeomorphisms of~$M$.
A pseudo-isotopy
between $f_0, f_1 \in \Diff(M)$
is a diffeomorphism
$F \colon M \times I \cong M \times I$ which restricts to $f_i$ on $M \times \{i\}$.
We define
\[ \wt \pi_0\Diff_{}(M) := \{ [f \colon M \cong M] \}, \]
the group of the pseudo-isotopy classes of diffeomorphisms of $M$.

For $m_0 \in M$,
let $\ADiff(M, m_0)$ be the group of almost diffeomorphisms of $M$ 
with singular point $m_0$.
A pseudo-isotopy between $f_0, f_1 \in \ADiff(M, m_0)$ is a homeomorphism
$F \colon M \times I \to M \times I$
with $F|_{{M_0} \times \{i\}} = f_i$ and which is smooth, except possibly
along $\{m_0\} \times I$.
We define
\[ \wt \pi_0\ADiff_{}(M, m_0) := \{ [f \colon M \acong M] \}, \]
the group of pseudo-isotopy classes of almost diffeomorphisms of $M$ with singular point $m_0$.

\subsection{The almost diffeomorphism classification}
\label{ss:alsmost_classification}
In this subsection we show how Theorem \ref{thm:a_class}
follows from the classification results of \cite{crowley02}.
The almost diffeomorphism classification given in \cite{crowley02} used a different but closely
related definition of a quadratic linking family.  We begin by explaining
the relationship between the two definitions of linking family and showing that
Theorem \ref{thm:a_class} is equivalent to \cite[Theorem B]{crowley02}.  
We then describe the main ideas of the proof of \cite[Theorem B]{crowley02} and interpret linking families
in terms of connected sum splittings.
Throughout this subsection $M$ is $2$-connected
and we have the global notation $G = H^4(M)$ with torsion subgroup 
$T \subseteq G$
and free quotient $F = G/T$.

Let us start with some elementary algebra for the group $G$.
Let $\iota \colon T \to G$ be the inclusion 
and $\pi \colon G \to F$ be the canonical projection.
We let $\sect(\pi) := \{ \sigma \colon F \to G\}$ be the set of sections of $\pi \colon G \to F$ and 
we let $\proj(\iota) := \{ \tau \colon G \to T \}$ be the set of projections over $\iota$;
\ie $\tau \circ \iota = {\rm Id}_T$.
The sets $\sect(\pi)$ and $\proj(\iota)$ are in bijection by mapping 
$\sigma \mapsto \tau_\sigma$, where $\im(\sigma) = \ker(\tau_\sigma)$.  
Both sets admit simple transitive actions of
$\hom(F, T)$ via addition of functions.
For $\phi \in \hom(F, T)$, $f \in F$ and $g \in G$ we have
\[ (\sigma + \phi)(f) = \sigma(f) + \phi(f) \quad \text{and} \quad
(\tau + \phi)(g) = \tau(g) + \phi(\pi(g)).   \]
Notice that $\tau_{\sigma + \phi} = \tau_\sigma - \phi$.

\begin{rmk*}
The action of $\hom(F, T)$ on $\sect(\pi)$ used above differs by a sign from the corresponding
action in \cite[p.\,39]{crowley02}.
\end{rmk*}

Let $(G, b, p)$ be a base so that $b$ is a torsion form on $T$ and $p \in 2G$.
Recall that $\cal{Q}(b)$ is the set of refinements of $b$ and given $q \in \cal{Q}(b)$,
let us write $\beta(q)$ for the homogeneity defect of $q$: see Section \ref{ss:tfs_and_Gauss_sums}.
In \cite[Definition 2.39]{crowley02} a quadratic linking
family on a base $(G, b, p)$ was defined as a function
\[ q^\bullet \colon \sect(\pi) \to \cal{Q}(b)\]
such that for all $\sigma \in \sect(\pi)$ and for all $\phi \in \hom(G, T)$,
\[  q^{\sigma + \phi} = q^\sigma_{-\phi(\pi(p)/2)} \quad \text{and} \quad
\beta(q^\sigma) = \tau_\sigma(p). \]
We explain the topological significance of these conditions below,
focussing for now on the algebra.

In this paper we work with linking families which are functions on $S_2$
and we now explain how to pass between linking families defined on $S_2$
and linking families defined on $\sect(\pi)$.
Given a section $\sigma \in \sect(\pi)$ there is a unique element
$k(\sigma) := \sigma(\pi(p)/d_\pi)$
of $S_{d_\pi}$ which lies in ${\rm Im}(\sigma)$
and so we obtain the function
\[ \sect(\pi) \to S_{d_\pi}, 
\quad \sigma \mapsto k(\sigma) \in {\rm Im}(\sigma) \cap S_{d_\pi}. \]
Now multiplication by $e_\pi = \frac{d_\pi}{2}$ gives a map $S_{d_\pi} \to S_2$ 
and we set
$\wh S_2 := e_\pi S_{d_\pi} \subset S_2$.
Given a refinement $q^\circ \colon S_2 \to \cal{Q}(b)$ we define 
\begin{subequations}
\begin{equation} \label{eq:q_bullet_and_q_sect}
q^\bullet \colon \sect(\pi) \to \cal{Q}(b), \quad q^\sigma := q^{e_\pi k(\sigma)} .
\end{equation}
Conversely, given $q^\bullet \colon \sect(\pi) \to \cal{Q}(b)$ we define
\begin{equation} \label{eq:q_sect_and_q_bullet}
q^\circ \colon \wh S_2 \to \cal{Q}(b), \quad q^{e_\pi k(\sigma)} := q^\sigma
\end{equation}
\end{subequations}
and extend $q^\circ$ to all of $S_2$ by the transformation rule of Definition \ref{def:foqr} \ref{it:q_trans}.
The transformation rules for $q^\bullet$ and $q^\circ$ ensure that they are determined
by their value on a single section or element of $S_2$.
Moreover, these transformation rules are compatible since
$k(\sigma + \phi) = k(\sigma) + \phi(\pi(p)/d_\pi)$,
$ q^{\sigma + \phi} = q^\sigma_{-\phi(\pi(p)/2)} $
and
$$ q^{e_\pi k(\sigma + \phi)} = 
q^{e_\pi ( k(\sigma) + \phi ( \pi(p)/d_\pi ) )} = 
q^{e_\pi k(\sigma) + \phi(\pi(p)/2)} = 
q^{e_\pi k(\sigma)}_{-\phi(\pi(p)/2)}.$$
Hence we have
\begin{lem} \label{lem:2_types_of_qlf}
The mappings $q^\circ \mapsto q^\bullet$ and $q^\bullet \mapsto q^\circ$ of 
\eqref{eq:q_bullet_and_q_sect} and \eqref{eq:q_sect_and_q_bullet} define inverse equivalences 
of categories between linking families defined on $S_2$ and linking families defined on $\sect(\pi)$. \qed
\end{lem}

\begin{proof}[Proof of Theorem \ref{thm:a_class}]
Let $M$ be $2$-connected, $q_M^\circ$ the linking family of $M$ as defined in Definition \ref{def:qlf}
and $-q^\bullet_M$ the linking family of $M$ as defined in \cite[Definition 2.39]{crowley02}
(we have introduced the sign to correct the mistake in \cite[Definition 2.50]{crowley02}: see 
Remark \ref{rmk:sign_of_q}.)
Comparing these definitions, we see that for each $\sigma \in \sect(\pi)$
\begin{equation} \label{eq:qs_related}
q^{e_\pi k(\sigma)}_M = q^\sigma_{M}.
\end{equation}
Now \cite[Theorem B]{crowley02} states that
all linking families defined on $\sect(\pi)$ arise as the quadratic linking families of $2$-connected $M$ and 
that any isomorphism of linking families defined on $\sect(\pi)$ is realised by an almost diffeomorphism.  
Hence Theorem \ref{thm:a_class} follows by combining \cite[Theorem B]{crowley02}, \eqref{eq:qs_related}
and Lemma \ref{lem:2_types_of_qlf}.
\end{proof}

We now explain the proof of \cite[Theorem B]{crowley02}.  
Recall that every $2$-connected $M$ is the boundary
of a $3$-connected $W$ and that the characteristic form of $W$, $(H^4(W, \del W), \lambda_W, \alpha_W)$,
is a complete invariant of $3$-connected $W$ under diffeomorphisms by \cite{wall62};
see \cite[Corollary~2.5]{crowley02}.  
Let $\natural$ denote the boundary connected sum of manifolds with boundary.
A foundational theorem of Wilkens 
\cite[Theorem 3.2]{wilkens71} (see also \cite[Theorem 2.24]{crowley02})
states that for any diffeomorphism $f \colon \del W_0 \cong \del W_1$, 
there are $W_2$ and $W_3$ with
$\del W_2 = \del W_3 = S^7$ and a diffeomorphism 
\[ g \colon W_0 \natural W_2 \cong W_1 \natural W_3 \]
extending $f$.
The boundary
of $W$ is a homotopy sphere, if and only if
$(H^4(W, \del W), \lambda_W, \alpha_W)$ is nonsingular.  Hence we 
say that two characteristic forms are \emph{stably isomorphic} if they
become isomorphic after addition of nonsingular characteristic forms.  The above discussion shows that
classifying $2$-connected $7$-manifolds up to almost diffeomorphism is equivalent to classifying
characteristic forms up to stable isomorphism.  
This was achieved in \mbox{\cite[Theorem~3.4]{crowley02}} by extending ideas of Wall 
\cite[Theorem p.\,156]{wall72} from the setting of even
forms to the setting of characteristic forms.  The point is that an isomorphism 
$F \colon \del(H_0, \lambda_0, \alpha_0) \cong \del(H_1, \lambda_1, \alpha_1)$ of the boundaries of characteristic
forms can be used to glue them together to obtain a nonsingular characteristic form
\[ (H_0, -\lambda_0, \alpha_0) \cup_F (H_1, \lambda_1, \alpha_1). \]
It is then possible to explicitly write down an isomorphism of characteristic forms
\begin{multline*}
 E \colon (H_0, \lambda_0, \alpha_0) \oplus \bigl( (H_0, -\lambda_0, \alpha_0) \cup_F (H_1, \lambda_1, \alpha_1) \bigr) \to \\
(H_1, \lambda_1, \alpha_1) \oplus \bigl( (H_0, -\lambda_0, \alpha_0) \cup_{\rm Id} (H_0, \lambda_0, \alpha_0) \bigr), 
\end{multline*}
such that $\del E = F$.  Combined with Lemma \ref{lem:splitq}, these methods give the following theorem, which 
is a refinement of a special case of \cite[Theorem 3.4]{crowley02}.

\begin{thm}[\cf\,{\cite[Theorem 3.4]{crowley02}}] \label{thm:stable_iso_and_qlfs}
For $i = 0, 1$, let $(H_i, \lambda_i, \alpha_i)$ be two characteristic forms.
The following are equivalent:
\begin{enumerate}
\item 
There is an isomorphism of refinements
$F \colon \del(H_0, \lambda_0, \alpha_0) \cong \del(H_1, \lambda_1, \alpha_1)$;
\item
There are nonsingular characteristic forms $(H_j, \lambda_j, \alpha_j)$, $j = 2, 3$, and
an isomorphism 
\[ E \colon (H_0, \lambda_0, \alpha_0) \oplus (H_2, \lambda_2, \alpha_2) \cong (H_1, \lambda_1, \alpha_1) \oplus
(H_3, \lambda_3, \alpha_3) \]
such that $\del E  = F$;
\item
There is a nonsingular characteristic form $(H, \lambda, \alpha)$ containing 
$(H_0, -\lambda_0, \alpha_0)$
and $(H_1, \lambda_1, \alpha_1)$ as orthogonal summands. 
\end{enumerate}
In addition, there is a canonical isomorphism 
$F_\lambda = F \colon \del(H_0, \lambda_0, \alpha_0) \cong \del (H_1, \lambda_1, \alpha_1)$
in case~$(\rm{iii})$.
\end{thm}

\begin{rmk*}
By Lemma \ref{lem:2_types_of_qlf} the statement of
Theorem \ref{thm:stable_iso_and_qlfs} and the discussion before it applies
equally well to linking families defined over $\sect(\pi)$ and linking families
defined over~$S_2$.
\end{rmk*}

\begin{proof}[Proof of Theorem \ref{thm:stable_iso_and_qlfs}]
This follows from \cite[Lemma 3.12]{crowley02} and the proof
of \cite[Theorem~3.4]{crowley02}.
\end{proof}

We now explain how the linking family $q^\circ_M$ of $M$ parametrises connected 
sum decompositions of $M$ into a summand with torsion-free homology and a summand
which is a rational homotopy sphere.
By Theorem \ref{thm:a_class}, $2$-connected
rational homotopy spheres $M$ with torsion linking form $(T, b)$ are classified up to 
almost diffeomorphism by their quadratic refinements $q^0_M \in \cal{Q}(b)$.  
We shall write $M(q)$ for any rational homotopy sphere with linking form isomorphic to $q$. 
The simplest
examples of $2$-connected 
$M$ with $H^4(M)$ torsion-free are given in the following
\begin{defn} \label{def:MF}
Let $F \cong \Z^b$ be a free abelian group of rank $b$ and $d_\pi$ be an even integer.
We define the $2$-connected $7$-manifold
$$ M(F, d_\pi) := \sharp_b(S^3 \tilde \times_{d_\pi} S^4),$$
where $S^3 \tilde \times_{d_\pi} S^4$ is the total space of the
$S^3$-bundle over $S^4$ with trivial Euler class and
first Pontrjagin class equal to $2d_\pi$ 
times the the preferred generator of $H^4(S^4)$
and $\sharp_b$ denotes the $b$-fold connected sum of 
a manifold with itself.
The base of $M(F, d_\pi)$ is identified with $(\Z^b, 0, (d_\pi, \dots, d_\pi))$.
\end{defn}

We define an {\em almost splitting} of $M$ to be an almost diffeomorphism with a singular 
point $m_0 \in M$
\[ f \colon M \acong M(q^f) \sharp M(F, d_\pi), \]
where $q^f$ is some quadratic refinement of $b_M$
and $f_0(m_0) \in M(q^f)$.
Two almost splittings $f_0$ and $f_1$ are called
{\em $H^*$-equivalent} if there is an almost diffeomorphism
$g \in \ADiff(M, m_0)$
with $g(m_0) = m_0$ and $H^*(g) = \Id$,
an almost diffeomorphism 
$g_T \colon M(q^{f_0}) \acong 
M(q^{f_1})$ with singular point $f_0(m_0)$ and a diffeomorphism
$g_F \colon M(F, d_\pi) \cong M(F, d_\pi)$ 
such that $g_T(f_0(m_0)) = f_1(m_0)$ and the following diagram commutes up to pseudo-isotopy:

\[ \xymatrix{ M \ar[d]^g  \ar[r]^(0.25){f_0} & M(q^{f_0}) \sharp M(F, d_\pi) \ar[d]^{g_T \sharp g_F} \\
M \ar[r]^(0.25){f_1} & M(q^{f_1}) \sharp M(F, d_\pi) }  \]
We define $\Asplt(M) := \{ [f] : \text{$f$ an almost splitting of $M$} \}$ 
to be the set of $H^*$-equivalence classes of almost splittings of
$M$ and note that there is a well-defined map
\[ \Asplt(M) \mapsto \sect(\pi), \quad [f] \mapsto \sigma(f), \]
where $\im(\sigma(f)) = f^*\bigl(H^4(M(F, d_\pi))\bigr)$.  The following theorem is implicit in \mbox{\cite[Definition 2.50]{crowley02}}. 


\begin{thm} \label{thm:almost_split1}
Let $M$ have linking family $q^\bullet_M \colon \sect(\pi) \to \cal{Q}(b)$.
For each $\sigma \in \sect(\pi)$ there is a unique $H^*$-equivalence class of almost splitting 
\[ f_\sigma \colon M \acong M(q_M^\sigma) \sharp M(F, d_\pi). \]
Consequently the map $\Asplt(M) \to \sect(\pi)$ is a bijection.
\end{thm}

\begin{proof}
Let $W$ be a $3$-connected coboundary of $M$ with characteristic form 
$(H, \lambda, \alpha) = (H^4(W, \del W), \lambda_W, \alpha_W)$.
We recall from the definition of the linking family defined by $W$
in~\eqref{eq:alg_qdef}, that there are orthogonal splittings
of $(H, \lambda, \alpha)$
\[ \psi \colon (H, \lambda, \alpha) \cong (R, \lambda_R, \alpha_\psi) \oplus (F, 0, \alpha|_F), \]
where $(R, \lambda_R, \alpha_\psi)$ is nondegenerate.  For every such splitting $\psi$,
the classification of $3$-connected coboundaries (see \cite[Corollary 2.5]{crowley02}) implies that 
there is a corresponding boundary connected sum splitting $g_\psi \colon W \cong W_\psi \natural W_F$.
In addition, there is a corresponding section $\sigma = \sigma(\psi) \in \sect(\pi)$ where $\im(\sigma) = j(H^4(W_f))$,
for $j$ the natural homomorphism $H^4(W, M) \to H^4(M)$.
By definition, (\cf \cite[Definition 2.50]{crowley02}),
\[ q_M^{\sigma} = \del(R, \lambda_R, \alpha_\psi), \]
and we define $f_\sigma \colon M \cong M(q_M^\sigma) \sharp M(F, d_\pi)$ to be the diffeomorphism on the boundary induced by the
the splitting $g_\psi$.  
This shows that $\Asplt(M) \to \sect(\pi)$ is onto.  

Suppose that $f_0$ and $f_1$ are two splittings of $M$ defining the same section $\sigma$.
Then the $H^*$-equivalence class of $f_i$ is determined by the almost diffeomorphism type
of $M(q^{f_i})$.
Now the Poincar\'{e} dual of $\im(\sigma)$ is a finitely generate free abelian group $\wh F \subset H_3(M)$.
We choose a basis $\{x_1, \dots, x_b\}$ for $\wh F$ and this is represented by a set of disjoint embeddings 
$\phi \colon \sqcup_{i=1}^b \colon D^4 \times S^3 \subset M$.
We let $M_\phi$ be the outcome of surgery on $\phi$.
Clearly there are choices $\phi_0$ and $\phi_1$ for $\phi$ so that $M_{\phi_i} \acong M(q^{f_i})$.
We claim that the almost diffeomorphism type of $M_\phi$ is independent of the 
choice of $\phi$ and this 
implies that $\sigma \colon \Asplt(M) \to \sect(\pi)$ in injective.

To prove the claim, let $W_\phi$ be the trace of surgeries on $\phi$ and let 
$W_1 : = W \cup_M W_\phi$ be the union of $W_\phi$ and our original $3$-connected coboundary.
By construction, we see that there is a fixed $\alpha_\sigma \in R^*$ such that the characteristic form of 
$W_1$ is isomorphic to the orthogonal sum
$(R, \lambda_R, \alpha_\sigma) \oplus (H_1, \lambda_1, \alpha_1)$ where 
$(H_1, \lambda_1, \alpha_1)$ is nonsingular.  It follows that the almost diffeomorphism type of
$M_\phi$ is well-defined.
\end{proof}

We conclude this subsection by identifying a simpler complete almost diffeomorphism invariant of $2$-connected $M$.  
Recall that the quadratic refinements $q \in \mathcal{Q}(b)$ of a torsion form $(b, T)$
are classified by their homogeneity defect $\beta \in 2T$ and 
Arf invariant $A(q) \in \Q/\Z$.
For a refinement $(G, q^\circ, p)$ with torsion form $(b, T)$, 
$\beta_h = p-2h$ is the homogeneity defect of
$q^h$ and $\aut(b)$, the group of automorphisms of $b$, 
acts on $\cal{Q}(b)$, the set of refinements of~$b$.
We define the {\em almost smooth splitting set} of $M$ to be the set
\[ \bar{\cal{Q}}^{\mathrm{as}}(M) := \{ \bigl( [\beta_h], A(q^h) \bigr) : h \in S_{2} \}
\subset \bigl(2 TH^4(M)/\!\aut(b) \bigr) \times \Q/\Z .\]
The following classification theorem is a direct corollary of Theorem \ref{thm:almost_split1}
and \eqref{eq:qs_related}.

\begin{cor} \label{cor:minima_almost_classification}
Let $F \colon (H^4(M_1), b_{M_1}, p_{M_1}) \to (H^4(M_0), b_{M_0}, p_{M_0})$
be an isomorphism of the bases of $M_0$ and $M_1$.
The following are equivalent:
\begin{enumerate}
\item \label{it:ad1}
$M_0$ is almost diffeomorphic to $M_1$;
\item \label{it:ad2}
$(F^{\#} \times \Id)(\bar{\cal{Q}}^{\mathrm{as}}(M_0)) =
\bar{\cal{Q}}^{\mathrm{as}}(M_1)$; 
\item \label{it:ad3}
$(F^{\#} \times \Id)(\bar{\cal{Q}}^{\mathrm{as}}(M_0)) \cap
\bar{\cal{Q}}^{\mathrm{as}}(M_1) \neq \emptyset$.
\qed
\end{enumerate} 
\end{cor}

%
%
%

\subsection{The homotopy classification} \label{ss:homotopy}
In \cite[\S 6]{crowley02} $2$-connected $M$ were classified up to homotopy equivalence using
{\em $J$-quadratic linking families}, as we now review.
For a torsion form $(b, T)$, $\mathcal{Q}_J(b) \subset \mathcal{P}({\mathcal{Q}(b)})$ 
was defined to be the 
set of subsets of $\mathcal{Q}(b)$ of the form
$$ S(q) := \{ q_{12t} : t \in T\}.$$
Notice that for $q_0, q_1 \in S(q)$, $\beta(q_0) - \beta(q_1) \in 24 T$
and so $S(q)$ has a well-defined homogeneity defect $\beta(S(q)) \in T \otimes \Z/24$.
For a group $G$, recall that $T \subseteq G$ is the torsion subgroup,
$\pi \colon G \to F = G/T$ is the map to the torsion-free quotient of $G$,
$\sigma \colon F \to G$ denotes a section of $\pi$ and $\tau_\sigma \colon G \to T$
denotes the projection defined by $\sigma$.
A {\em $J$-quadratic linking family} was defined to be a triple
$(G, q^\bullet_J, \rho_{24}(p))$,
where $\rho_{24}(p) \in G \otimes \Z/24$ is an even element and 
$q^\circ_J \colon \sect(\pi) \to \mathcal{Q}_J(b)$ is a function 
such that for all $\phi \in \mathrm{Hom}(F, T)$ we have
$q^{\sigma + \phi}_J = (q^\sigma_J)_{-\phi(\pi(p))}$ 
and $\beta(q^{\sigma}_J) = (\tau_\sigma \otimes \Id)(\rho_{24}(p))$.
The $J$-quadratic linking family of $M$,
written $(H^4(M), q^\bullet_{J,M}, \rho_{24}(p_M))$ is induced from its 
quadratic linking family in the obvious way.

We now up-date the notion of a $J$-quadratic linking family to that of a $J$-refinement.
Recall from Section \ref{ss:elaboration} that 
$\rho_{12} \colon \mathcal{Q} \to J\mathcal{Q}(b)$ is the quotient map
which identifies quadratic refinements $q \simeq q_{12t}$ and that 
a $J$-refinement of a base $(G, b, p)$ 
is a triple $(G, Jq^\circ, \rho_{24}(p))$ 
where $Jq^\circ \colon S_2 \to J\mathcal{Q}(b)$ is a function
satisfying $Jq^{h + t} = (Jq^j)_{-t}$ and
$\rho_{24}(\beta_h) = \rho_{24}(p-2h)$.
The $J$-refinement of $M$ is the triple
$(H^4(M), \rho_{12} \circ q^\circ_M, \rho_{24}(p_M))$.

Given a $J$-refinement $(G,  Jq^\circ, \rho_{24}(p))$
we define the corresponding $J$-quadratic linking family by setting
\begin{subequations}
\begin{equation} \label{eq:J1}
q^\bullet \colon \sect(\pi) \to \cal{Q}_J(b), \quad 
q_J^\sigma := Jq^{e_\pi k(\sigma)} 
\end{equation}
and given the function $q^\bullet \colon \sect(\pi) \to \cal{Q}_J(b)$ of a $J$-quadratic
linking family we define the corresponding $J$-refinement by setting
\begin{equation} \label{eq:J2}
q^\circ \colon \wh S_2 \to \cal{Q}(b), \quad q^{e_\pi k(\sigma)} := q^\sigma
\end{equation}
\end{subequations}
and we then extend the definition of $Jq^\circ$ to all of $S_2$ using the 
transformation rule for $J$-refinements.
The correspondence between quadratic linking functions define on $\sect(\pi)$ and 
on $S_2$ identified in Lemma \ref{lem:2_types_of_qlf} is easily modified to give

\begin{lem} \label{lem:J-refinements}
The mappings $Jq^\circ \mapsto q^\bullet_J$ and $q^\bullet_J \mapsto Jq^\circ$ of 
\eqref{eq:J1} and \eqref{eq:J2} define inverse equivalences 
of categories between $J$-refinements defined on $S_2$ and 
$J$-quadratic linking families defined on $\sect(\pi)$. \qed
\end{lem}

\begin{proof}[Proof of Theorem \ref{thm:homotopy}]
Let  $(G, Jq^{\circ}, \rho_{24}(p))$ be a $J$-refinement.
The transformation rule for $J$-refinements ensures that a $J$-refinement is determined by 
$[q^h]$ for any $h \in S_2$.
Since $\mathcal{Q}(b) \to \mathcal{Q}_{12}(b)$ is onto,
it follows that every $J$-refinement $(G, [q^\circ], \rho_{24}(p))$ 
is the mod~$24$ reduction of a refinement $(G, q^\circ, p)$.
Theorem \ref{thm:a_class} then entails that every $J$-refinement
is realised as $(H^4(M), \rho_{12} \circ q^\circ_M, \rho_{24}(p_M)$,
for a $2$-connected $M$.

If $F \colon (H^4(M_1), b_{M_1}, p_{M_1}) \to (H^4(M_0), b_{M_0}, p_{M_0})$ is an isomorphism of bases where
$M_0$ and $M_1$ are $2$-connected, then by
\cite[Theorem 6.11]{crowley02}, $F = f^*$ for a homotopy equivalence
$f \colon M_0 \to M_1$ if and only if 
$(q_{J, M_0}^\bullet, \rho_{24}(p_{M_0})) = 
F^{\sharp}(q_{J,M_1}^\bullet, \rho_{24}(p_{M_1}))$
and Lemma \ref{lem:J-refinements}, this is happens if and only if
$(Jq_{M_0}^\circ, \rho_{24}(p_{M_0})) = 
F^{\sharp}(Jq_{M_1}^\circ, \rho_{24}(p_{M_1}))$.
\end{proof}

\subsection{Inertia and reactivity in more detail} 
\label{ss:inertia_and_reactivity2}
Recall that $I(M)$, the inertia group of $M$, is 
the subgroup of the group of homotopy spheres $\Sigma \in \Theta_7$
such that 
$M \sharp \Sigma \cong_{} M$,
and that
\[ I_H(M) \subseteq I(M), \]
is the subgroup of homotopy spheres $\Sigma$ for which there is a diffeomorphism 
$f \colon M \sharp \Sigma \cong M$
such that $H^*(f) = \Id$, where we regard $M \sharp \Sigma$ and $M$ as the same topological space.

One might expect that a complete understanding of $I(M)$ is needed to pass from the almost diffeomorphism classification of 2-connected 7-manifolds
to the diffeomorphism classification, but it turns that that a lower bound on the order of $I_H(M)$
suffices.
The main result of this subsection, Proposition \ref{prop:p2null}, establishes
this required lower bound on $I_H(M)$ for $2$-connected $M$: see Remark \ref{rmk:p2null}.
In general, computing $I^{}(M)$ is a delicate problem which we take up 
for $2$-connected $M$ in Section \ref{sec:auto}.
We begin this section by relating the groups $I_H(M)$ and $I(M)$ to certain mapping class
groups of $M$.

Given an almost diffeomorphism $f \in \ADiff(M, m_0)$,
we consider the problem of deciding whether $f$ is pseudo-isotopic to a diffeomorphism.
From Section \ref{ss:almost_diffeomorphism} we recall the homotopy
sphere $\Sigma_0$ which measures the singularity of $f$ at $m_0$.
From the definition
of pseudo-isotopy in Section \ref{ss:almost_diffeomorphism}
we see that the diffeomorphism class of the homotopy sphere
\[ \Sigma_f : = \Sigma_0 \]
is invariant under pseudo-isotopies.  Moreover, 
it is clear that $f$ defines a diffeomorphism
\begin{equation} \label{eq:Sigma_f}
f \colon  M \sharp \Sigma_f \cong M,	
\end{equation}
and that $\Sigma_{f \circ g} = \Sigma_f \sharp \Sigma_g$.
Further, an application of the Alexander trick---see Rourke and Sanderson
\mbox{\cite[Proposition~3.22]{rourke72}}---shows that $f$ is pseudo-isotopic to
a diffeomorphism if and only if~$\Sigma_f \cong S^7$.
It follows that there is a {\em singularity homomorphism},
\[ \del \colon \wt \pi_0\ADiff(M, m_0) \to \Theta_7, \quad [f] \mapsto \Sigma_f, \]
with kernel isomorphic to the image of $\wt \pi_0 \Diff(M)$ in $\wt \pi_0\ADiff(M, m_0)$.
We define the subgroup 
$\wt \pi_0\ADiff_H(M, m_0) \subseteq \wt \pi_0\ADiff_{}(M, m_0)$ 
of pseudo-isotopy classes inducing the identity on $H^*(M)$ and define the
singularity homomorphism 
$\del_H \colon \wt \pi_0\ADiff_{H}(M, m_0) \to \Theta_7$
to be the restriction of~$\del$.
From \eqref{eq:Sigma_f} we see that
\begin{equation} \label{eq:I_and_del}
 I_H(M) = \im(\del_H) \quad \text{and} \quad I(M) = \im(\del).
\end{equation}
Given $f \in \ADiff(M, m_0)$
we now show how to determine 
$\Sigma_f \in \Theta_7$ using the mapping torus of $f$, $T_f$,
which is the almost smooth manifold constructed from
the cylinder $M \times I$ by using $f$ to identify points at either end:
\[ T_f : = (M \times [0, 1])/ (m,0) \sim (f(m),1) \]
Since $f$ is an almost diffeomorphism, the closed $8$-manifold $T_f$ admits a smooth
structure, except perhaps at the point $\ol m_0 = [m_0, 0]$ corresponding
to the singular point of $f$.  Indeed if $B^8_0 \ni \ol m_i$ is a small open
ball containing $\ol m_0$, then
\begin{equation} \label{eq:W_f}
 W_f : = T_f - B^8_0
\end{equation}
is a compact smooth manifold with boundary
\[ \del W_f \cong \Sigma_f. \]
We choose a spin structure on $T_f$ and denote the corresponding $8$-dimensional almost smooth 
spin manifold by $T_f$ also: 
no confusion shall arise since we are interested only in the characteristic number
\[  p^2(f) : = \langle p_{T_f}^2, [T_f] \rangle \in \Z, \]
which depends only on the oriented almost diffeomorphism type of $T_f$ since
$2p_{T_f} = p_1(T_f)$ and $H^8(T_f) \cong \Z$ (in fact $p_{T_f}$ is independent
of the choice of spin structure by \cite[p.\,170]{cadek08}).
It follows that $p^2(f)$ is an invariant of the pseudo-isotopy class of $f$.
For the statement of the next lemma, we recall the renormalised Eells-Kuiper
invariant of a homotopy sphere~$\Sigma$, $\mu(\Sigma) \in \Z/28$,
defined in \eqref{eq:classical}.
By \cite[(13)]{eells62}, $\mu(\Sigma_1) = \mu(\Sigma_2)$ if and only if
$\Sigma_1 \cong \Sigma_2$.

\begin{lem} \label{lem:mapping_torus}
For every almost diffeomorphism $f \in \ADiff(M, m_0)$ the following hold:
\begin{enumerate}
\item \label{lem:mapping_torus:8}
$p^2(f) \in 8 \Z$,
\item \label{lem:mapping_torus:mu}
$\mu(\Sigma_f) = \frac{p^2(f)}{8} \in \Z/28$,
\item \label{lem:mapping_torus:smooth}
$f$ is pseudo-isotopic to a diffeomorphism if and only if $p^2(f) \in 224 \Z$.
\end{enumerate}
\end{lem}

\begin{proof}
\ref{lem:mapping_torus:8} This follows since by
Lemma \ref{lem:p_M}\ref{it:boundary}, $p_{T_f}$ is characteristic for
the intersection form of $T_f$.  Hence by \cite[Lemma 5.2, \S 5]{milnor73},
$p^2(f) \equiv \sigma(T_f)$~mod~$8$.  But by Novikov additivity, the signature of
$T_f$ is zero.

\noindent
\ref{lem:mapping_torus:mu} This follows since $W_f$ defined in \eqref{eq:W_f} above
is a smooth spin coboundary for $\Sigma_f$ and so can be used to compute $\mu(\Sigma_f)$.
Since $\sigma(W_f) = \sigma(T_f) = 0$, applying \eqref{eq:classical} gives the result.

\noindent
\ref{lem:mapping_torus:smooth} The almost diffeomorphism $f$ is pseudo-isotopic to a diffeomorphism if and only 
if $\Sigma_f \cong S^7$.  Hence \ref{lem:mapping_torus:smooth} follows directly from \ref{lem:mapping_torus:mu}.
\end{proof}

In the light of Lemma \ref{lem:mapping_torus}, we define the function
\[ p^2 \colon \wt \pi_0 \ADiff_{}(M, m_0) \to \Z, \quad [f] \mapsto p^2(f). \]
Since the image of $p^2$ plays a key role, we define non-negative integers called
the {\em reactivity of $M$}, $R(M)$, and the \emph{(co)homologically fixed reactivity}
of $M$, $R_H(M)$, by the following equations
\[ p^2(\wt \pi_0 \ADiff_{}(M, m_0)) = 
 R(M) \Z \quad \text{and} \quad
p^2(\wt \pi_0 \ADiff_{H}(M, m_0)) = R_H(M) \Z .\]
By Lemma~\ref{lem:mapping_torus}~\ref{lem:mapping_torus:8}, $R(M)$ and $R_H(M)$ are
both divisible by 8. 
By Lemma \ref{lem:mapping_torus}~\ref{lem:mapping_torus:mu} and the definition
of reactivity we have
\begin{prop} \label{prop:role_of_reactivity} \hfill
\begin{enumerate}
\item \label{it:r_of_r}
$I(M) = \frac{R(M)}{8} \Theta_7$,
\item $I_H(M) = \frac{R_H(M)}{8} \Theta_7$. \qed
\end{enumerate}
\end{prop}

For other problems, for example counting the number of deformation equivalence classes of $G_2$-structures
on $M$ as in \cite{nu}, it important to know the value of $p^2$ for diffeomorphisms.
Hence we defined the {\em smooth reactivity of $M$}, $R^\Diff(M)$, and the
{\em smooth (co)homologically fixed reactivity of $M$}, $R^\Diff_H(M)$ by the equations
\[ p^2(\wt \pi_0\Diff_{}(M)) = R^\Diff(M) \Z
\quad \text{and} \quad 
 p^2(\wt \pi_0\Diff_{H}(M)) = R^\Diff_H(M) \Z, \]
where $\wt \pi_0\Diff_H(M) \subseteq \wt \pi_0\Diff(M)$ is
the subgroup of pseudo-isotopy classes acting trivially on $H^*(M)$.
By Lemma~\ref{lem:mapping_torus}~\ref{lem:mapping_torus:smooth} we have

\begin{lem} \label{lem:smooth_reactivity} \hfill
\begin{enumerate}
\item $R^\Diff(M) = \lcm(R(M), 224)$,
\item $R^\Diff_H(M) = \lcm(R_H(M), 224)$. \qed
\end{enumerate}
\end{lem}

We next construct almost diffeomorphisms $f \!\in\! \ADiff_{}(M, m_0)$ 
on $2$-connected $M$
with $p^2(f) \neq 0$.
Recall that $d_\pi$ is the divisibility of $\pi(p_M) \in H^4(M)/TH^4(M)$ and
$\tilde d_\pi = \lcm(4, d_\pi)$.

\begin{prop} \label{prop:p2null}
If $M$ is 2-connected then $R_H(M) | 2\tilde d_\pi;\!\,$ 
\ie~$2\td_\pi\Z \subseteq p^2(\wt \pi_0 \ADiff_H(M, m_0))$.
\end{prop}

\begin{rmk} \label{rmk:p2null}
If $M$ is $2$-connected, Propositions \ref{prop:role_of_reactivity} and \ref{prop:p2null} together give
$\frac{\tilde d_\pi}{4} \Theta_7 \subseteq I_H(M)$.
In Corollary \ref{cor:p2image}~\ref{it:p2null} below we will show that $R_H(M) = 2 \tilde d_\pi$
and hence $\frac{\tilde d_\pi}{4} \Theta_7 = I_H(M)$.
\end{rmk}

For the proof of Proposition \ref{prop:p2null} it will be useful to compute the
characteristic number $p^2(T_f)$ using a co-bounding spin $8$-manifold $W$.
We define
the closed almost smooth \mbox{$8$-manifold}
\[ X_f : = (-W) \cup_f W. \]

\begin{lem} \label{lem:T_and_X}
With the notation above, $p^2(f) = \an{p_{X_f}^2, [X_f]}$.
\end{lem}

\begin{proof}
Since $p^2(f) = \an{p_{T_f}^2, [T_f]}$ and $\an{p_{X_f}^2, [X_f]}$ are characteristic numbers,
if suffices to prove that $T_f$ is oriented bordant to $X_f$.
Consider the manifolds $M \times I$ and $W \sqcup -W$.  Both have boundary 
$-M \sqcup M$,
$T_f$ is formed from $M \times I$ by gluing $-M$ to $M$ via $f$ and $X_f$ if formed
from $-W \sqcup W$ by gluing $-M$ to $M$ via $f$.  It therefore suffices to prove
that $-W \sqcup W$ is bordant relative to the boundary to $M \times I$.  But the manifold
$W \times I$ is a rel.~boundary bordism from $-W \sqcup W$ to $M \times I$, and
we are done.
\end{proof}

\begin{proof}[Proof of Proposition \ref{prop:p2null}]
We assume $d_\pi \not= 0$, since otherwise there is nothing to prove.
By Theorem \ref{thm:almost_split1} or by \cite[Theorem 1]{wilkens72},
we may decompose $M$ as a connected sum of spin manifolds
\[ M \cong_{} M_0 \sharp M_1 \]
where $M_0 = M(\Z, d_\pi) = S^3 \tilde \times_{d_\pi} S^4$ 
is the total space of a $3$-sphere bundle over $S^4$ from Definition \ref{def:MF}.
We shall produce the required almost diffeomorphisms on the manifold $M_0$ and
then extend by the identity to $M$.  Let 
\[ M_0^\bullet : = M_0 - \Int(D^7) \]
be $M_0$ minus a small open disc.  Since $M_0$ is the total space of an $S^3$-bundle over $S^4$, 
there is a diffeomorphism
\[ M_0^\bullet \cong (D^3 \tilde \times_{d_\pi} S^4) \cup_{S^2 \times D^4} (D^3 \times D^4), \]
where $D^3 \tilde \times_{d_\pi} S^4$ is a tubular neighbourhood of a section of 
$M_0 \to S^4$ and $D^3 \times D^4$ is a $3$-handle added to $D^3 \tilde \times_{d_\pi} S^4$ along the tubular neighbourhood of a fibre $2$-sphere,
$S^2 \times D^4 \subset S^2 \tilde \times_{d_\pi} S^4$.

By \cite[p.\,171\,(2)]{wall62} we may identify $\pi_3(SO(4))$ as the group of
pairs of integers $(n, p)$ where $n \equiv p \mmod 2$, so that the
corresponding $S^3$-bundle over $S^4$ has Euler class $n \in H^4(S^4) = \Z$ and
first Pontrjagin class $2p$.
Let $\gamma_{n, p} \colon (D^3, S^2) \to (SO(4), {\rm Id})$ be a smooth
function representing $(n, p)$.  We define a diffeomorphism
\[ f^\bullet_{n, p} \colon M^\bullet_0 \cong M^\bullet_0 \]
where $f^\bullet_{n, p}|_{D^3 \tilde \times_\alpha S^4}$ is the identity and on
the $3$-handle we use the $D^3$ co-ordinate to twist the $D^4$-coordinate
using $\gamma_{n, p}$.  To be explicit:
\[ f^\bullet_{n, p}|_{D^3 \times D^4}(u, v) = (u, \gamma_{n, p}(u)(v)) .\]
We observe that
there is a subspace $S^3 \vee S^4 \subset M^\bullet_1$
such that the restriction $f^\bullet_{n, p}|_{S^3 \vee S^4}$ is the identity and 
$M^\bullet_1$ deformation retracts to $S^3 \vee S^4$.
It follows that $f^\bullet_{n, p}$ acts trivially on cohomology. 

Let $m_0 \subset D^7 \subset M_0$ be the centre of the 
disc removed to make $M^\bullet_0$.
By coning the restriction of $f^\bullet_{n, p}$ to the boundary of $M_0^\bullet$,
we extend $f^\bullet_{n, p}$ to 
an almost diffeomorphism $f_{n, p}$ of $M_0$ with a single singular point $m_0$.
Since $f^\bullet_{n, p}$ acts trivially on cohomology, so does $f_{n, p}$.
Since $M_0$ admits a unique spin structure for each
orientation and since $f_{n, p}$ is orientation preserving, $f_{n, p}$ is a
spin almost diffeomorphism.
By construction, $f_{n, p}$ is the identity on any 7-disc contained in $D^3 \tilde \times_{d_\pi} S^4$ and 
hence we may we extend $f_{n, p}$ to $M$ by taking the connected
sum with the identity on~$M_1$. Thus we define the spin almost diffeomorphism
\[ g_{n, p} := f_{n, p} \sharp \Id_{M_1} \colon M \cong M \]
with single singularity at $m_0$ and which acts trivially on cohomology.
We claim that 
\begin{equation} \label{eq:spin}
p^2(g_{n, p}) = p^2(f_{n, p}) = d_\pi(2p -nd_\pi).
\end{equation}
The manifold $M_0 \cong S^3 \tilde \times_{d_\pi} S^4$ bounds the
$8$-dimensional $D^4$-bundle $W_0 : = D^4 \tilde \times_{d_\pi} S^4$,
and we let $W_1$ be any spin coboundary for $M_1$.  
We form the twisted doubles $X_{f_{n, p}} := (-W_0) \cup_{f_{n, p}} W_0$
and
\begin{equation} \label{eq:Xg}
 X_{g_{n, p}} := (-W_0 \natural -W_1) \cup_{g_{n, p}} (W_0 \natural W_1) \cong  
X_{f_{n, p}} \sharp \bigl( (-W_1) \cup_{{\rm id}} W_1 \bigr). 
\end{equation}
Applying Lemma \ref{lem:T_and_X} we have,
\[ p^2(g_{n, p}) = \an{p^2(X_{g_{n, p}}), [X_{g_{n, p}}]} = 
\an{p^2(X_{f_{n, p}}), [X_{f_{n, p}}]} = p^2(f_{n, p}),\]
where the second equality holds by \eqref{eq:Xg} since the characteristic number $p^2$
is a bordism invariant, 
which is additive for connected sums and 
$(-W_1) \cup W_1 = \del (W_1 \times I)$.
Writing $X_{n, p} := X_{f_{n, p}}$, it therefore remains to compute $\an{p^2(X_{{n, p}}), [X_{{n, p}}]}$.
From the construction of $X_{n, p}$ we see that $H_4(X_{n, p}) \cong \Z(x) \oplus \Z(y)$ 
where $x$ is represented by the zero section of $W_1$ and $y = [D^4 \cup D^4]$ is represented 
by an embedded $4$-sphere obtained by gluing two fibres of the $D^4$-bundle $W_1$ together, 
one from each copy of $W_0$.  
By construction, the normal bundle of the $4$-sphere $D^4 \cup D^4$ has
characteristic function $\gamma_{n, p}$ and hence Euler number~$n$.
It follows that the intersection form of $X_{n, p}$ 
with respect to the basis $\{x, y\} $ is given by the following matrix:
\[ \left( \begin{array}{cc} 0 & 1 \\ 1 & n  \end{array} \right) \]
Moreover since $x$ is represented by an embedded $4$-sphere with tubular neighbourhood diffeomorphic
to $D^4 \tilde \times_{d_\pi} S^4$
and since $y$ is represented by an embedded $4$-sphere with normal bundle $\gamma_{n, p}$, 
we have $p_{X_{n, p}}(x) = d_\pi$ and $p_{X_{n,p}}(y) = p$.  We conclude that the Poincar\'{e} 
dual of $p_{X_{n, p}}$ is given by
\[ PD(p_{X_{n, p}}) = (p - nd_\pi) x + d_\pi y. \]
It follows that 
$\an{p_{X_{n, p}}^2, [X_{n, p}]}= 2d_\pi(p-nd_\pi) + n d_\pi^2 = d_\pi(2p - n d_\pi)$, 
and the claim \eqref{eq:spin} is proven.

Finally we need to choose $n$ and $p$ so that $d_\pi(2p-nd_\pi) = 2\td_\pi$.
Recall that we may choose $n$ and $p$ freely subject to the constraint that 
$n \equiv p$~mod~$2$.  If $d_\pi = 4k+2$, then $2 \tilde d_\pi = 4 d_\pi$ and we choose 
$(n, p) = (0, 2)$.  If $d_\pi = 4k$, then $2 \tilde d_\pi = 2d_\pi$ and we set 
$(n, p) = (1, 2k+1)$.
\end{proof}
\subsection{The proof of the main classification theorem} \label{ss:classification_theorem}
The mod 28 distillation of $M$ is the quadruple $(H^4(M), q_M^\circ, \EK_M, p_M)$
where $q^\circ_M$ is the quadratic linking family of $M$ as in
Definition \ref{def:qlf} and the generalised Eells-Kuiper invariant
\[ \EK_M \colon \divsp \to \Q/ \dM \Z \]
is the mod 28 Gauss refinement of $q^\circ_M$ defined by \eqref{eq:EK_of_M}.  
In this subsection we prove Theorem \ref{thm:class} which states, in part, that
mod 28 distillations give a complete invariant of diffeo\-morphisms
of $2$-connected $M$. For the remainder of the subsection $M$ is $2$-connected.

Recall that the (renormalised) classical Eells-Kuiper invariant, as defined
by \eqref{eq:classical}, gives a group isomorphism
\[ \Theta_7 \cong \Z/28 \Z, \quad \Sigma \mapsto \mu(\Sigma) : = \mu_\Sigma(0) .\]
The following lemma is obvious from the definitions of $q^\circ_M$ and $\EK_M$.

\begin{lem} \label{lem:g_M_sharp_Sigma}
For all $\Sigma \in \Theta_7$, $q^\circ_{M \sharp \Sigma} = q^\circ_M$ and
$\EK_{M \sharp \Sigma} = \EK_M + [\mu(\Sigma)]$, where $[\mu(\Sigma)]$ is the mod~$\dM$
reduction of $\mu(\Sigma)$. \qed
\end{lem}

\begin{proof}[Proof of Theorem \ref{thm:class}]
The existence of a smooth $M$ with mod 28 distillation isomorphic
to a prescribed $(G,q^\circ,\EK,p) \in \catek$ follows from the corresponding
existence statement in Theorem \ref{thm:a_class}, since
Lemma \ref{lem:g_M_sharp_Sigma} lets us freely change the Eells-Kuiper
invariant of a manifold with a prescribed refinement $(G,q^\circ,p) \in \catq$.

By Lemma \ref{lem:ek_defined}, which is proven in \eqref{eq:compare},
the generalised Eells-Kuiper invariant is a diffeomorphism invariant.
Now we suppose
$F^\#(q^\circ_{M_0}, \EK_{M_0}, p_{M_0}) = (q^\circ_{M_1}, \EK_{M_1}, p_{M_1})$.
As explained in~\eqref{eq:Sigma_f}, Theorem \ref{thm:a_class}
means that there is a homotopy sphere $\Sigma$ and a diffeomorphism 
\[ f \colon M_0 \sharp \Sigma \cong M_1 \]
such that  $H^*(f) = F \colon H^4(M_1) \cong H^4(M_0)$.
It remains to show that $\Sigma \in I_H(M_0)$.  For if so, there is a diffeomorphism 
\[ h \colon M_0 \cong M_0 \sharp \Sigma \]
with $H^*(h) = \Id$ and then $f \circ h \colon M_0 \cong M_1$ is a diffeomorphism 
with $H^*(f \circ h) = H^*(f) = F$.

Since $f$ is a diffeomorphism it preserves the mod~28 Gauss refinements.  
Applying Lemma \ref{lem:g_M_sharp_Sigma} we have
\[ \EK_{M_1} = F^\#(\EK_{M_0} + \mu(\Sigma)) = F^\#(\EK_{M_0}) + \mu(\Sigma) .\]
On the other hand, our assumption is that $F^\#(\EK_{M_0}) = \EK_{M_1}$.  Since $d_{M_0} = d_{M_1}$,
\[
\mu(\Sigma) = \EK_{M_1} - F^\#(\EK_{M_0}) = 0 \in \Z/d_{M_1} \Z = \Z/d_{M_0} \Z.
\]
By Remark \ref{rmk:p2null}, $\Sigma \in I_H(M_0)$ and this completes the proof.
\end{proof}

\begin{proof}[Proof of Theorem \ref{thm:classification_categorical}]
That the functor $\cal{D} \colon \cal{M}_{7,2}^{\spg} \to \catek$ is surjective and faithful
is a restatement of Theorem \ref{thm:class}.
To see that $\cal{D}$ is additive and compatible with orientation reversal,
let $i = 0, 1$, and let $M_i = \del W_i$ where $W_i$ has characteristic from $(H_i, \lambda_i, \alpha_i)$
with boundary refinement $q^\circ_i$.
The mod $28$ Gauss refinement of $M_i$ is $(q^\circ_i, \del g_{W_i})$
where we set $\del g_{W_i} := g_{W_i}$~mod~$\dM$ as in \eqref{eq:EK_of_M}.
Since the characteristic form of $-W_i$ is $(H_i, -\lambda_i, \alpha_i)$ and
the characteristic form of the boundary connected sum $W_0 \natural W_1$ is $(H_0, \lambda_0, \alpha_0) \oplus (H_1, \lambda_1, \alpha_1)$,
it follows that mod $28$ Gauss refinements of $-M_i$ and $M_0 \sharp M_1$ are
$(-q^\circ_i, -\del g_{W_i})$ and $(q_1^\circ, \del g_{W_0}) \oplus (q_1^\circ, \del g_{W_1})$ respectively.
\end{proof}

\subsection{Smooth splitting functions} \label{ss:splitting_functions}
In this subsection we consider connected sum splittings of $2$-connected $M$
in the smooth category and we prove a smooth analogue of Theorem \ref{thm:almost_split1}.
We also prove Theorem \ref{thm:classification_minimal}, which is the smooth analogue
of Corollary \ref{cor:minima_almost_classification}. 
Throughout this subsection $M$ is 2-connected.

Let $(T, b)$ be a torsion linking form and $d$ an even integer.  We define the set
\[ \wh{\cal{Q}}_{d}(b) := \{(q, s)\} \subset \cal{Q}(b) \times \Q/d \Z, \]
which consists of pairs of quadratic refinements $q$ of $b$ and rational residues mod $d$ where 
$A(q) \equiv s$~mod $\Z$.
By Theorem \ref{thm:class}, rational homotopy spheres $M$ with torsion linking forms 
$(H^4(M), b_M) = (T, b)$ are classified up to diffeomorphism by the pair 
$(q_{M}, \EK(M)) \in \wh{\cal{Q}}_{28}(b)$.  
We denote this rational homotopy sphere by
\[ M = M(q, s), \]
where $q = q_M$ and $s = \EK(M)$.
Suppose that we are given a base $(G, b, p)$ with $F = G/T \cong \Z^b$ and
$\pi(p) \in F$ of divisibility $d_\pi$; as in \S\ref{ss:gr}, $\pi$ denotes
the projection $G \to F$. 
By Theorem~\ref{thm:almost_split1}, if $M$ has base
$(H^4(M), p_M, b_M) \cong (G, b, p)$,
then for some rational homotopy sphere $M(q, s(f))$ and 
base-point $m_0 \in M$, there is a connected sum splitting
\[ f \colon M \cong M(q, s(f)) \sharp M(F, d_\pi), \]
where $f(m_0) \in M(q, s(f))$.
Recall that $\sigma(f) \in \sect(\pi)$ is
given
by $\im(\sigma) = f^*(H^4(M(F, d_\pi)))$.  In considering the
uniqueness of $M(q, s)$ in such a splitting, we note that 
by Theorem \ref{thm:class}, $I(M(F, d_\pi)) = \dM\Theta_7$ 
(see also Remark \ref{rmk:p2null}).
As a consequence, for $i = 0, 1$, we see that if $(q, s_i) \in \cal{Q}_{28}(b)$
and $s_0 \equiv s_1$~mod~$\dM\Z$, then there is a diffeomorphism
\begin{equation} \label{eq:splitting_inertia}
h \colon M(q, s_0) \sharp M(F, d_\pi) \cong M(q, s_1) \sharp M(F, d_\pi) 
\end{equation} 
such that $H^*(h)$ preserves the induced splittings of $H^4$.  
We define two splittings $f_0$ and
$f_1$ to be $H^*$-equivalent if there is an almost diffeomorphism
$g \in \ADiff(M, m_0)$
with $g(m_0) = m_0$, $H^*(g) = \Id$
and $\Sigma_g \in \wh d_\pi \Theta_7$,
an almost diffeomorphism
$g_T \colon M(q, s_0) \acong M(q, s_1)$
with singular point $f_0(m_0)$ and $\Sigma_{g_T} = \Sigma_g$ and a 
diffeomorphism $g_F \colon M(F, d_\pi) \cong M(F, d_\pi)$
such that $g_T(f_0(m_0)) = f_1(m_0)$ and the following diagram commutes up to pseudo-isotopy:
\[ \xymatrix{ M \ar[d]^g  \ar[r]^(0.225){f_0} & M(q^{f_0}, s(f_0)) \sharp M(F, d_\pi) \ar[d]^{g_T \sharp g_F} \\
M \ar[r]^(0.225){f_1} & M(q^{f_1}, s(f_1)) \sharp M(F, d_\pi) }  \]

We define $\splt(M) := \{ [f] \}$ to be the set of $H^*$-equivalence classes of splittings of
$M$.  We also define the {\em smooth splitting function} of $M$ 
\[ \wh q^{\; \bullet}_M \colon \sect(\pi) \to \wh{\cal{Q}}_{\dM}(b), 
\quad \sigma \mapsto \wh q_M^{\;\sigma} := \bigl( q_M^{e_\pi k(\sigma)}, \EK_M(k(\sigma)) \bigr), \]
where we recall that $k(\sigma) \in S_{d_\pi}$ is defined by $k(\sigma) \in \im(\sigma) \cap\,S_{d_\pi}$.
From the diffeomorphism in \eqref{eq:splitting_inertia}, we see that 
$M\bigl( q^{e_\pi k(\sigma)}, \EK(k(\sigma)) \bigr) \sharp M(F, d_\pi)$ 
gives a well-defined diffeomorphism type for each section $\sigma$.

\begin{thm} \label{thm:split}
Let $M$ have smooth splitting function 
$\wh q^{\; \bullet}_M \colon \sect(\pi) \to \wh{\cal{Q}}_{\dM}(b_M)$.
For each $\sigma \in \sect(\pi)$ there is a unique $H^*$-equivalence class of splitting 
\[ f_\sigma \colon M \cong M(\wh q^{\; \sigma}_M) \sharp M(F, d_\pi). \]
Consequently the map $\splt(M) \to \sect(\pi), [f] \mapsto \sigma(f)$ is a bijection.
\end{thm}

\begin{proof}
The proof is a refined version of the proof Theorem~\ref{thm:almost_split1}
and we adopt the notation of that proof so that $M$ has $3$-connected coboundary $W$.
Specifically, the proof of the existence of $f_\sigma$ is verbally the same,
except that now by \cite[Definition 2.50]{crowley02} we have
\[ \wh q_M^{\; \sigma} = \left(\del(R, -\lambda_R, \alpha_\psi), 
\left[ \bigl( \rif_R\bigl( \alpha_\psi, \alpha_\psi \bigr) - \sigma(\lambda_R)\bigr)/8
\right] \right). \]
Hence the splitting $W \cong W_\psi \natural W_F$ defines the splitting 
$f_\sigma \colon M \cong M(\wh q_M^{\; \sigma}) \sharp M(F, d_\pi)$.	

To show that splittings $f_0$ and $f_1$ defining the same section $\sigma$ are
$H^*$-equivalent, we consider the nonsingular characteristic form $(H_1, \lambda_1, \alpha_1)$.
The symmetric form $(H_1, \lambda_1)$ has a Lagrangian $L \subset H_1$ corresponding to $H^3(M)$ and 
hence $\alpha_1(L) = d_\pi \Z$.  The proof of Proposition \ref{prop:p2null} now shows that
$\frac{\tilde d_\pi}{4}$ divides $(\rif_1(\alpha_1, \alpha_1) - \sigma(\lambda_1))/8$.
Consequently the diffeomorphism type of $M_\phi$ is determined up to connected sum with
$\Sigma \in \dM \Theta_7$, and this shows that $f_0$ and $f_1$ are $H^*$-equivalent splittings.
\end{proof}

\begin{proof}[Proof of Theorem \ref{thm:classification_minimal}]
For a $2$-connected $M$,
define an action of $\aut(b_M)$ on $\wh{\cal{Q}}_{\dM}(b_M)$ by $F \cdot (q, s) = (q \circ F, s)$
and let $[q, s]$ denote the $\aut(b_M)$ orbit of $(q, s)$.
We define the map
\[ \beta \colon \wh{\cal{Q}}_{\dM}(b_M) \to 2T/\!\aut(b_M) \times \Q/\dM\Z, \quad (q, s) \mapsto ([\beta_q], s).  \]
Since the Gauss sum of each $q$ is given by $A(q) = s \mmod \Z$, we note that
Theorem~\ref{thm:classification_of_q} ensures that $[q, s] = [q', s']$ 
if and only if $\beta(q, s) = \beta(q', s')$.

Now for $i = 0, 1$, let $M_0$ and $M_1$ have smooth splitting functions 
$\wh q_i^{\; \bullet} \colon \sect(\pi_i) \to \wh{\cal{Q}}_{\dM}(b_{M_i})$, 
where $\pi_i \colon H^4(M_i) \to H^4(M_i)/TH^4(M_i)$ is the projection
and suppose there 
there is an isomorphism $F \colon (H^4(M_1), b_{M_1}, p_{M_1}) \to (H^4(M_0), b_{M_0}, p_{M_0})$
of their bases.
By Theorem~\ref{thm:split},
$M_0$ and $M_1$ are diffeomorphic if and only if there are sections $\sigma$ and $\sigma'$, a homotopy sphere
$\Sigma \in \dM\Theta_7$ and a diffeomorphism 
$M(\wh q_0^{\; \sigma}) \cong M(\wh q_1^{\; \sigma'}) \sharp \Sigma$.
But this happens if and only if there is an isomorphism $q_0^{\sigma} \cong q_1^{\sigma'}$ and
$\mu_0(k(\sigma)) = \mu_1(k(\sigma'))$; \ie if and only if
$(F^{\#} \times \Id)(\beta(\wh q_0^{\; \sigma})) = \beta(\wh q_1^{\; \sigma'})$.
With the notation above the smooth splitting set of $M_i$, defined
in the introduction just prior to the statement of Theorem \ref{thm:classification_minimal}, 
is the set
\[ \bar{\cal{Q}}(M_i) = \{ \beta(\wh q^{\; \sigma}_i) : \sigma \in \sect(\pi_i) \}.  \]
The above shows that $M_0$ and $M_1$ are diffeomorphic if and only if the sets
$(F^{\#} \times \Id)\bigl(\bar{\cal{Q}}(M_0)\bigr)$ and $\bar{\cal{Q}}(M_1)$ intersect 
and consequently this happens if and only if these sets coincide.
\end{proof}

\section{Automorphisms of \texorpdfstring{$H^4(M)$}{H\^{}4(M)}} \label{sec:auto}
The smooth classification Theorem \ref{thm:class} implies that the number
of different smooth structures on the same 2-connected almost-smooth 7-manifold
corresponds to the number of different mod 28 Gauss refinements
of the linking family $(H^4(M), q_M^\circ, p_M)$.
The first estimate of the number of smooth structures on $M$, $\dM = \gcd(\frac{\td_\pi}{4}, 28)$, 
only counts smooth structures on $M$ modulo almost diffeomorphisms that act
trivially on $H^4(M)$. To get the full picture, we need to understand how
automorphisms of the quadratic linking family $q^\circ_M$ act on the Gauss refinements.
We begin this process in Section \ref{ss:aut_b_on_gauss_functions}.

Conveniently enough, it turns out that this
problem can be reduced to understanding how automorphisms of the base
$(H^4(M), b_M, p_M)$ act on linked functions: see Proposition \ref{prop:imtp} in
Section \ref{ss:aut_q_on_gauss_functions}.
While we do not have a complete
description of this action in general we still have control up to a 
factor $2^r$, where $r \in \{0, 1, 2\}$ is explicitly defined in \eqref{eq:rdef}.  
Moreover it is feasible to understand it for explicit examples:
see Examples \ref{ex:r0}, \ref{ex:r2}, \ref{ex:split_vs_hyperbolics} and \ref{ex:r_for_split_sums}.
With Proposition \ref{prop:imtp} in hand, we proceed in Section \ref{ss:computation_of_reactiviy}
to determine the reactivity of $2$-connected $M$ in terms of $r$ and the integer
$\mdiv$ defined in \eqref{eq:mdiv} and recalled in Section \ref{ss:auto_notation}.
\subsection{Notation} \label{ss:auto_notation}

We begin by setting up some terminology.
Given a finitely generated abelian group $G$, $p \in 2G$, and
$b \colon T \times T \to \Q/\Z$ a torsion form,
let $\aut_b$ denote the group of isomorphisms $F \colon G \to G$ preserving $p$
and $b$. %
If $q^\circ$ is a family of quadratic refinements of $(G,b,p)$, let
$\aut_{q^\circ}$ be the subgroup of $\aut_b$ that preserves $q^\circ$ too.

Let $\pi \colon G \to G/T$ be the projection to the free quotient of $G$.  
Let $\shearb \subseteq \aut_b$ be the subgroup of ``pure shears'', \ie $F$
acting trivially on $T$ and $G/T$. In other words,
$F = \Id_G + \rho \circ \pi$ for some homomorphism
$\rho \colon G/T \to T$ such that $\rho\big(\frac{\pi(p)}{d_\pi}\big)$
is $d_\pi$-torsion (the last condition is equivalent to $F(p) = p$); so
actually $\shearb$ does not depend on $b$ at all.
Similarly let $\shearq = \shearb \cap \aut_{q^\circ}$, the subgroup of shears in
$\aut_{q^\circ}$. For $h \in \divs_2$,
$(F^\#q)^h = q^{F(h)} = q^{h + \rho\left(\frac{\pi(p)}{2}\right)}
= q^h_{-\rho\left(\frac{\pi(p)}{2}\right)}$, 
so for $F$ to preserve $q$ we need
$\rho\big(\frac{\pi(p)}{2}\big) = 0$. 
Hence $\shearq$ simply corresponds to
homomorphisms $\rho \colon G/T \to T$ such that 
$\rho\big(\frac{\pi(p)}{d_\pi}\big)$
is $\frac{d_\pi}{2}$-torsion. In particular: $\shearq$ actually depends
on neither $q^\circ$ nor $b$ but only on $p$, and if $F \in \shearb$
then $F^2 \in \shearq$.

We say that a cyclic subgroup $C \subseteq T$ is a \emph{split summand}
if $T$ is a direct sum of $C$ and its $b$-orthogonal complement.
We call $x \in T$ \emph{split} if it generates a split summand; this is equivalent to
\[ b(x,x) = \frac{m}{n} , \]
where $n$ is the order of $x$ and $m$ is coprime to $n$. 

Given the element $p \in 2G$ we consider the following notions of its
divisibility (if $p$ is a torsion element we set all three integers to be 0):
\begin{align*}
d_{\phantom \pi} & :=
{\rm Max} \{s \in \Z : s \textrm{ divides } p \in G \}, \\
d_\pi & :=
{\rm Max} \{s \in \Z : s \textrm{ divides } \pi(p) \in G/T \}, \\
\mdiv & :=
{\rm Max} \{ s : s, \mpl \in \Z, \; s\mpl^2 \text{~divides~} m p \in G \} .
\end{align*}
We have an obvious chain of divisibilities
\[ 2 \mid d \mid \mdiv \mid d_\pi. \]
Further $d = d_\pi$ if and only if $\mdiv = d_\pi$, since the latter
implies that the maximum in the definition of $\mdiv$ is attained with
$\mpl = 1$.
For an integer $s$, let $\ord_2 s$ denote the exponent of $2$ in
the prime factorisation of $s$; \eg~${\ord_2 2^j = j}$.

\begin{defn} \label{def:extremal_exponent}
A non-negative integer $e$ is a {\em $2$-extremal exponent} for $(G, p)$ if
for some $\mpl$ such that $\mdiv \mpl^2$ divides $\mpl p$, $\ord_2 \mpl = e$.
\end{defn}

\begin{ex}
Let $p = (2^a, 2^c) \in \bbz \times \bbz/2^b\Z$ with $a,  b \geq c \geq 1$.
Then $d_\pi = 2^a$, $d = 2^c$, and $\mdiv = \max(2^c, 2^{a-b+c})$.
The 2-extremal exponents are 0 for $a \leq b$, and $b-c$ for $a \geq b$.
\end{ex}

\subsection{The action of \texorpdfstring{$\aut_b$}{Aut\_b} on linked functions} \label{ss:aut_b_on_gauss_functions}

Given $F \in \aut_b$ and any $k \in \divsp$, set 
$t := F(k) -k \in T$ (not necessarily $d_\pi$-torsion, unless $F_{|T}$
is the identity) and $\beta_k := p - d_\pi k$, and let
\begin{equation}
\label{eq:pnaive}
P(F) := d_\pi^2 b(t, t) - 2d_\pi b(\beta_k, t) \in \Q/2d_\pi\Z .
\end{equation}
In other words, $P(F) = -\Delta(k, t)$ from \eqref{eq:Delta}.
Equivalently, we can characterise $P(F)$ by
\begin{equation}
\label{eq:pdiff}
F^\# \lf = \lf - \frac{P(F)}{8} \mod \frac{d_\pi}{4}\Z
\end{equation}
for any linked function $\lf$ (use that $(F^\# \lf)(k) = \lf(F(k)) =
\lf(k + t) = g(k) + \frac{\Delta(k,t)}{8}$ by the condition \eqref{eq:lf_def}
for $\lf$ to be a linked function.)
The first characterisation, \eqref{eq:pnaive}, is independent of $\lf$ and the
second, \eqref{eq:pdiff}, of $k$, so in fact $P$ depends on neither.
If $F$ preserves a family of quadratic refinements, then
taking $\lf$ to be a Gauss refinement of that family shows that $P$
takes values in $8\Z/2d_\pi\Z$ (in the next subsection we study a corresponding
$8\Z/2\td_\pi\Z$-valued function $\tp$).
Even if $F$ does not preserve a family of quadratic refinements, the
fact that the $\textrm{mod } \quart \Z$ reduction of the Arf invariant of a
quadratic refinement of $(b,p)$ depends only on $(b,p)$ itself shows that $P$
takes values in $2\Z$. It is also clear
from \eqref{eq:pdiff}, or from \eqref{eq:pnaive} together with
\eqref{eq:cocycle}, that $P$ is a homomorphism
$\aut_b \to 2\Z/2d_\pi\Z$.

Let $j_\pi = \ord_2 d_\pi$, and $j_\msub = \ord_2 \mdiv$.

\begin{lem}
\label{lem:autb}
$P(\aut_b) \subseteq \mdiv \Z/2d_\pi \Z$.
If $b$ lacks a split $2^{2e+j_\msub}$ summand for some $2$-extremal
exponent $e$, then $P(\aut_b) \subseteq 2\mdiv \Z/2d_\pi \Z$.
\end{lem}

\begin{proof}
Pick some $y \in G$ such that $\mpl^2\mdiv y = m p$. Then $s := Fy - y$ is
an $\mpl^2\mdiv$-torsion element. It suffices to show that
\begin{equation}
\label{eq:pdinfty}
P(F) = \mpl^2\mdiv^2 b(s, s) \mod 2\mdiv ,
\end{equation}
because the RHS is $\mdiv$ if the 2-primary part of $s$ is split, and $0$
otherwise.

Note that $\pio := \frac{d_\pi}{\mdiv}$ and $\frac{\mpl}{\pio}$ are
integers. Let $k := \frac{\mpl}{\pio}y$. Then
$k \in \divsp$, and $\pio t = \mpl s$, so \eqref{eq:pnaive} implies 
\[ P(F) = \pio^2 \mdiv^2 b(t, t) - 2 \pio \mdiv b(\beta_k, t)
= \mdiv^2 b(\mpl s, \mpl s) - 2\mdiv b(\beta_k, \mpl s) %
\mod 2\mdiv . \]
Since $\beta_k = p - d_\pi k$ is $\mpl$-torsion, \eqref{eq:pdinfty} and the
result follows.
\end{proof}

If $F \in \shearb \subseteq \aut_b$, \ie $F = \Id_G + \rho \circ \pi$ for some
homomorphism $\rho \colon G/T \to T$, then
$t = F(k) - k = \rho\big(\frac{\pi(p)}{d_\pi}\big)$
is independent of the choice of $k \in \divsp$.
Since $\frac{\pi(p)}{d_\pi} \in G/T$ is
a primitive element of a free abelian group, we can prescribe its image
under a homomorphism $\rho$ arbitrarily. Determining the image
$P(\shearb)$ therefore amounts to computing the RHS of \eqref{eq:pnaive}
for all $d_\pi$ torsion elements $t \in T$.

\begin{lem}
\label{lem:aut0} 
$4\mdiv \Z/2d_\pi \Z \subseteq P(\shearb)$.
Moreover, if $j_\pi \not = j_\msub + 1$ or if $b$ has no split $2^{j_\pi}$
summand, then $2\mdiv \Z/2d_\pi \Z \subseteq P(\shearb)$.
\end{lem}

\begin{proof}
The key claim is that there exists a $d_\pi$-torsion
element $t$ such that $b(\beta_k, t) = \frac{1}{\pio}$, where
$d_\pi = \pio \mdiv$.
By the non-degeneracy of $b$, this is equivalent to $\beta_k$ having order at
least~$\pio$, and not being divisible by more than $\mdiv$. That any divisor of
$\beta_k$ also divides $\mdiv$ is obvious, and if $\mpl\beta_k = 0$ then
$\mpl p$ is divisible by $\mpl d_\pi$, which indeed implies $\pio \mid \mpl$ by
the definition of $\mdiv$.

Let $\rho$ be any homomorphism $G/T \to T$ mapping
$\frac{\pi(p)}{d_\pi} \mapsto t$, and
$F := \Id_G + \rho \circ \pi \in \shearb$.
If the 2-primary part of $t$ does
not generate a split $2^{j_\pi}$ summand then $d_\pi^2 b(t, t)$ is divisible by
$2d_\pi$, so 
\[ P(F) = d_\pi^2b(t,t) - 2d_\pi b(\beta_k, t) = 2\mdiv \mod 2d_\pi , \]
and we are done. Otherwise
$P(F) = d_\pi - 2\mdiv = (\pio-2)\mdiv \mmod 2d_\pi$. The subgroup this
generates is precisely $n\mdiv\Z/2d_\pi\Z$, where $n = \gcd(\pio{-}2, 2\pio) =
\gcd(\pio{-}2,4)$. Clearly $n$ is 1 or 2 except when $j_\pi = j_\msub + 1$,
in which case $n = 4$.
\end{proof}

Lemmas \ref{lem:autb} and \ref{lem:aut0} imply that the following is well-defined.

\begin{defn} \label{def:r}
Define $r = r(G,p,b) \in \{0,1,2\}$ by
\begin{equation}
\label{eq:rdef}
\im P = 2^r \mdiv \Z/2d_\pi \Z .
\end{equation}
\end{defn}

\begin{rmk} \label{rmk:values_of_r}
Lemmas \ref{lem:autb} and \ref{lem:aut0} provide necessary conditions for $r = 0$ or $r = 2$.
In particular, if $G$ has no 2-torsion then $r = 1$.  
The next examples show that there are bases with $r = 0$ and bases with $r=2$.
\end{rmk}

\begin{ex}
\label{ex:r0}
Let $G = \Z \oplus \Z/2^j$, $b = \an{\frac{1}{2^j}}$ and $p = (2^j, 0)$
(so $d_\pi = \mdiv = 2^j$).
Then the shear $F \colon (x,y) \mapsto (x, x+y)$ has $P(F) = 2^j \mmod 2^{j+1}$,
\ie $P(F) = \mdiv \mmod 2d_\pi$. Thus $r = 0$.
\end{ex}

\begin{ex}
\label{ex:r2}
Let $G = \Z \oplus \Z/2^j $, $b = \an{\frac{1}{2^j}}$ and
$p = (2^j, 2^{j-1})$ (so $d_\pi = 2^j$,
while $\mdiv = 2^{j-1}$). Now any $t \in T$ has
$d_\pi^2 b(t,t) + 2d_\pi(\beta_k, t) = 0 \mmod 2^{j+1}$, so $r = 2$.
\end{ex}

\subsection{The action of \texorpdfstring{$\aut_{q^\circ}$}{Aut\_q} on Gauss refinements} \label{ss:aut_q_on_gauss_functions}
Now let $q^\circ$ be a family of quadratic refinements of the base $(G, b, p)$, and let
$\aut_{q^\circ}$ denote its group of automorphisms.
For an automorphism $F \in \aut_{q^\circ}$ we define
$\tp(F) \in 8\Z/2\td_\pi \Z$ by
\[ \wt P(F) := -4d_\pi q^{e_\pi k}(t) + d_\pi(d_\pi {+} 2) \, b(t,t) \]
for any $k \in \divsp$ and $t := F(k) - k$; equivalently,
$\tp(F) = -\tdelta(k, t)$.
Now $\tp$ is a homo\-morphism
$\wt P \colon \aut_{q^\circ} \to 8\Z/2\td_\pi\Z$, such that
\begin{equation} \label{eq:P_tilde}
F^\# \gr = \gr - \frac{\tp(F)}{8} \mod \frac{\td_\pi}{4}\Z 
\end{equation}
for any Gauss refinement $\gr$ of $q^\circ$.
Notice that, similarly to the proof of Lemma \ref{lem:gr_via_linked},
$\tp$ can alternatively be characterised by
\begin{gather*}
\tp(F) = P(F) \mod 2d_\pi,
\\ \tp(F) = 0 \mod 8 .
\end{gather*}
We can therefore get some control on the image of the shear subgroup
$\shearq \subseteq \aut_{q^\circ}$ just from the observation that 
$F^2 \in \shearq$ for any $F \in \shearb$.
 
\begin{lem}
\label{lem:imtp0}
$\tp(\shearq) \supseteq 4\mdiv\Z/2\td_\pi\Z$.
\end{lem}

\begin{proof}
The proof of Lemma \ref{lem:aut0} showed that we can achieve
$P(F) = 2\mdiv$ or $P(F) = 2\mdiv + d_\pi$ for some $F \in \shearb$.
Then $F^2 \in \shearq$ has $P(F^2) = 4\mdiv$.
\end{proof}

\noindent
Conveniently, it turns out that the image of $\tp$ can be determined
directly from the image of $P$.

\begin{prop}
\label{prop:imtp}
$\im \tp = \{ n \in 8\Z/2\td_\pi\Z : n \mmod 2d_\pi \in \im P \}
= \lcm(8, 2^r\mdiv)\Z/2\td_\pi \Z$.
\end{prop}

\begin{proof}
If $d_\pi$ is not divisible by 4 then $\lcm(8,2^r \mdiv) = 4\mdiv$,
so the result follows from Lemma~\ref{lem:imtp0}.

If $4 \mid d_\pi$ and $F \in \aut_b$, then note
that for any $k \in \divsp$, $t := F(k)- k$ and $h := \frac{d_\pi}{2}k$
we have
\[ q^{h}\left(-\tfrac{d_\pi}{2}t\right) = \frac{P(F)}{8} \mod \Z . \]
Thus $P(F) = n \in 8\Z/2d_\pi\Z$ implies
that $(F^\# q)^h = q^{F(h)} = q^h_{-\frac{d_\pi}{2}t}$ and $q^h$,
which have equal inhomogeneity $\beta_h = p - 2h$
by definition, also have equal Arf invariant by \eqref{eq:arf_shift}.
Therefore by Theorem \ref{thm:classification_of_q} there is an
automorphism $F_T$ of $(T,b)$ such that $(F^\# q)^h \circ F_T = q^h$
(necessarily $F_T$ fixes~$\beta_h$).

Now suppose that $\sigma$ is a section of $\pi$, and
$k \in \im \sigma \cap \divsp$ (\cf Remark \ref{rmk:projections}).
Then $G \cong \im \sigma \oplus T$, and we may define
$\Id_{\im \sigma} + F_T \in \aut_b$. This fixes $k$ and $h$, so the composition
$F' := F \circ (\Id_{\im \sigma} + F_T)$ has
\[ (F'^\#q)^h = \left((\Id_{\im(\sigma)} + F_T)^\#F^\#q\right)^h
= (F^\#q)^h \circ F_T = q^h . \]
Hence $F' \in \aut_{q^\circ}$, and $F'(k) = F(k)$ implies $P(F') = P(F)$.
\end{proof}

\begin{ex}
\label{ex:r0'}
For the base $(\Z \oplus \Z/2^j, \an{\frac{1}{2^j}}, (2^j, 0))$ of Example \ref{ex:r0}
let $q^\circ$ be the refinement with $q^{(2^{j-1}, 0)} = \anq{1}{2^{j+1}}$.
The isomorphism $F$ of the base in Example \ref{ex:r0} does not
preserve~$q^\circ$: 
if $j > 1$ then $F$ alters the homogeneity defect of $q^{(2^{j-1}, 0)}$ and if
$j=1$ then $F$ alters the Arf invariant.  However, if $j \geq 3$ then 
$F' \colon (x,y) \mapsto (x, \; x + (2^{j-1} {+} 1) y)$ is an
isomorphism of $q^\circ$ with $\tp(F') = P(F) = 2^m \mmod 2^{j+1}$.
\end{ex}

\noindent
The following examples illustrate that $r$, and hence $\im \tp$, can depend
on $b$ as well as $(G,p)$.

\begin{ex} \label{ex:split_vs_hyperbolics}
Let $G = \Z \oplus (\Z/2^j)^2$ and $p = (2^j, 0, 0)$
(so $d_\pi = \mdiv = 2^j$).
Choosing the torsion form
$b_0 = \an{\frac{1}{2^j}} \oplus \an{\frac{1}{2^j}}$ on $TG$,
using Example \ref{ex:r0} shows that $r_0 = 0$. 
Let $b_1$ be the hyperbolic torsion form on $TG$ with matrix
\[ \begin{pmatrix} 0 & 2^{-j} \\ 2^{-j} & 0 \end{pmatrix}. \]
Since $d_\pi = \mdiv$, it follows that $r_1 = 0$ or $1$.  But $b_1$ contains no split
cyclic summands, and so by Lemma \ref{lem:aut0}, we conclude that $r_1 = 1$.
\end{ex}

The next example shows that $r$ cannot be determined merely from the type
of splitting of~$b$ (cyclics versus hyperbolics), but can depend on the
isomorphism classes of split cyclic summands.
 
\begin{ex} \label{ex:r_for_split_sums}
Let $G = \Z \oplus \Z/8 \oplus \Z/64 \oplus \Z/512$ with torsion form
$\anb{1}{8}\oplus \anb{1}{64} \oplus \anb{\epsilon}{512}$ 
($\epsilon = \pm 1$),
and $p = (64, 0, 8, 0)$. Then $d_\pi = 64$ and $\mdiv = 8$,
so $r$ is 0 or 1. The 2-extremal exponents are 0 and~3. If $F \in \aut_b$ then
by \eqref{eq:pdinfty}
\begin{align*}
P(F) = \mdiv \!\mmod 2\mdiv
&\Leftrightarrow (\Id-F)(1,0,0,0) \textrm{ split 512-torsion} \\
& \Leftrightarrow (\Id-F)(8,0,1,0) \textrm{ split 8-torsion} .
\end{align*}
Thus if $r = 0$ there must be some automorphism $f$ of $(T,b)$ such that
$(\Id - f)(0,1,0)$ plus a split 8-torsion element is divisible by 8, \ie
$f(0,1,0) = (a, 8b + 1, 8c)$ with $a$ odd. If $\epsilon = +1$ then the would-be
image has norm $\frac{17}{64}$ for any $a, b, c$, so there can be no such $f$;
hence $r = 1$. On the other hand, if $\epsilon = -1$ we can define
such an $f$ by the matrix
\[ \begin{pmatrix} 1 & -8 & 0 \\ 1 & 1 & 8 \\ 1 & 1 & 1\end{pmatrix}. \]
Setting $F = \Id_\Z + \rho + f$ with $\rho \colon \Z \to T, n \mapsto (0, 0, -n)$
makes $\Id - F$ map $(8,0,1,0) \in G$ to the split 8-torsion element
$(1,0,0) \in T$ (and $(1,0,0,0)$ to $(0,0,-1)$),
so $P(F) = \mdiv \mmod 2\mdiv$, and $r = 0$.
\end{ex}

\begin{cor} \label{cor:counting_Gauss_refinements}
Modulo the action of $\aut_{q^\circ}$, the number of possible Gauss refinements of
$(G,q^\circ\!,p)$ is
\[ \NumB{\frac{2^r \mdiv}{8}} , \]
and the number of possible mod 28 Gauss refinements is
\[ \gcd\left(28, \NumB{\frac{2^r \mdiv}{8}} \right) . \]
\end{cor}

\begin{rmk}
Notice that Corollary \ref{cor:counting_Gauss_refinements} combined with 
Theorems \ref{thm:a_class} and \ref{thm:class} gives the computation of
the inertia group $I(M)$ for $2$-connected $M$ from Theorem \ref{thm:i_and_r}.  
\end{rmk}

\subsection{The computation of reactivity} \label{ss:computation_of_reactiviy}
In this subsection we use Proposition \ref{prop:imtp} to prove lower bounds on the
reactivity of every spin $7$-manifold $M$.  When $M$ is $2$-connected we also prove
that this lower bound is sharp and so compute the reactivity of $2$-connected $M$.
Recall from Section \ref{ss:inertia_and_reactivity2} that if $f$ is a self-almost diffeomorphism of $M$, 
then the mapping torus $T_f$ is
almost smooth, the spin characteristic class $p_{T_f} \in H^4(T_f)$ is well defined and so is the integer
\[ p^2(f) = \an{p_{T_f}^2, [T_f]} \in 8\Z.\]
The next lemma provides the bridge between the algebraic arguments of
Sections \ref{ss:aut_b_on_gauss_functions} and \ref{ss:aut_q_on_gauss_functions}
and the computation of reactivity of $M$, as defined in \eqref{eq:defR}.
Note that if $f$ is a self-almost diffeomorphism of $M$, then the induced map
$f^* : H^4(M) \to H^4(M)$ preserves $q_M^\circ$,
\ie we have $f^* \in \aut_{q_M^\circ}$
and so $\tp(f^*) \in 8\Z/2\td_\pi\Z$ is defined. 

\begin{prop} \label{prop:p2mod}
$\tp(f^*) = p^2(f) \mmod 2\td_\pi$ for any self-almost diffeomorphism of $M$.
\end{prop}


\begin{proof}
Let $f \colon M \acong M$ be a self-almost diffeomorphism and let $W$ be a
$3$-connected spin coboundary for $M$.
In \S\ref{ss:eek} we used $\gr_W$ to denote the Gauss refinement of $q_M^\circ$
induced by the form $(FH^4(W, M), \lambda_W, p_W)$ (and used that to define
$\EK_M$).
We can use $f$ to glue two copies of $W$ together along $M$
and form the almost smooth spin manifold
$X : = (-W) \cup_f W$. Lemma \ref{lem:T_and_X} gives $p^2(f) = p^2_X$.
Applying \eqref{eq:P_tilde} to $F = f^*$ and combining with the
comparison of Gauss refinements in \eqref{eq:compare} we obtain 
\[ \wt P(f^*) \equiv 8(\gr_W - (f^*)^\# g_W)
\equiv p^2_X - \sigma(X) \equiv p^2_X \equiv p^2(f) \mod 2 \tilde d_\pi ,\]
where $\sigma(X) = \sigma(W) - \sigma(W) = 0$ by Novikov additivity.
\end{proof}

\begin{cor}
\label{cor:p2image}
For any closed spin 7-manifold $M$ we have that:
\begin{enumerate}
\item \label{it:p2null-H} 
$R_H(M)$ is divisible by $2\td_\pi$;
\item \label{it:p2null}
$R(M)$ is divisible by $\lcm(8, 2^r \mdiv)$;
\item \label{it:p2null-DiffH}
$R^{\Diff}_H(M)$ is divisible by $\lcm(224, 2\td_\pi)$; \vspace{0.75 mm}
\item \label{it:p2null-Diff}
$R^{\Diff}(M)$ is divisible by $\lcm(224, 2^r \mdiv)$.
\end{enumerate}
If $M$ is 2-connected then equality holds in each case.
\end{cor}

\begin{proof}
For part \ref{it:p2null-H}, recall that for $[f]$ to belong to $\wt \pi_0\ADiff_H(M, m_0)$
by definition means that $f^* = \Id$ on $H^4(M)$.
Thus $\wt P(f^*) = 0 \in 8\Z/2\td_\pi$, and Proposition \ref{prop:p2mod}
implies that $2 \td_\pi \mid R_H(M)$. Meanwhile Proposition \ref{prop:p2null}
shows that $R_H(M) \mid  2 \td_\pi$ if $M$ is 2-connected.

For part \ref{it:p2null}, let $M$ have refinement $(G, q^\circ, p)$.
Proposition \ref{prop:imtp} computes the image $\im \wt P = \lcm(8, 2^r\mdiv) \Z/2\td_\pi$, so Proposition \ref{prop:p2mod} gives $\lcm(8, 2^r\mdiv) \mid R(M)$.
On the other hand, if $M$ is 2-connected then
Theorem \ref{thm:a_class} states that every automorphism of $(G, q^\circ, p)$
is realised by an almost diffeomorphism $f \colon M \acong M$,
and so part \ref{it:p2null-H} and Proposition \ref{prop:imtp} imply the
equality.

Parts \ref{it:p2null-DiffH} and \ref{it:p2null-Diff} follow from parts \ref{it:p2null-H} and \ref{it:p2null} 
and Lemma~\ref{lem:smooth_reactivity}.
\end{proof}

\begin{proof}[Proof of Theorem \ref{thm:i_and_r}]
The computation of $R(M) = \lcm(8, 2^r\mdiv)$ is given
in Corollary \ref{cor:p2image}\ref{it:p2null}.
Then $I(M) = \Num\big(\frac{2^r\mdiv}{8}\big)\Theta_7$
by Proposition \ref{prop:role_of_reactivity}\ref{it:r_of_r}.
By Remark \ref{rmk:values_of_r}, $r=1$ if $TH^4(M)$ is of odd order.  
\end{proof}

\section{Examples} \label{sec:examples}
Ever since Milnor's discovery of exotic 7-spheres \cite{milnor56}, $2$-connected
$7$-manifolds have provided interesting examples in topology and geometry.
In this section we discuss various examples of $2$-connected $7$-manifolds.
In Section \ref{ss:sphere_bundles} we consider the total spaces of $3$-sphere bundles over $S^4$ 
and their connected sums.
In Section \ref{ss:geometry} we mention some examples admitting interesting metrics.
In Section \ref{ss:the} we give examples which are
tangentially homotopy equivalent but not homeomorphic.
Finally in Section \ref{ss:monoid} we present a refinement of Wilkens' list \cite[Theorem 1]{wilkens72} 
of the indecomposable generators
for the monoid of almost diffeomorphism classes of $2$-connected $7$-manifolds.

\subsection{3-sphere bundles over \texorpdfstring{$S^4$}{S\^{}4} and their connected sums} \label{ss:sphere_bundles}
Following the notation of~\cite{crowley13}, let $(n, p)$ be integers with the same
parity and let $M_{n, p} := S(\xi_{n, p})$ denote the total space of the
$3$-sphere bundle over $S^4$ for which the corresponding vector bundle
$\xi_{n, p}$ has Euler class $e(\xi_{n, p}) = n \in H^4(S^4)$
and spin characteristic class $\frac{p_1}{2}(\xi_{n, p}) = p \in H^4(S^4)$.  
By definition, we have $M_{0, p} = M(\Z, p)$, where 
$M(\Z, p)$ is as defined in Definition \ref{def:MF}.  
Using \eqref{eq:alg_qdef} and \eqref{eq:grchat}
and recalling the notation of Example \ref{ex:cyclic_b_and_q}, we compute for 
$n \neq 0$ that there is a diffeomorphism
\[ M_{n, p} \cong M\!\left( \biganq{-1}{2n}_{-p}\, , \;
\left[ \frac{p^2 - |n|}{8n} \right] \right). \]
\begin{ex} \label{ex:Milnor_sphere}
The Milnor sphere, $\Sigma_{\rm Mi} : = M_{1, 3}$, is homeomorphic to $S^7$ but not diffeomorphic 
to $S^7$ since $\mu(\Sigma_{\rm Mi}) = 1 \neq 0$~mod $28$: see \cite{milnor56} and \cite{eells62}.
\end{ex}
\noindent
In \cite{crowley03} the total spaces of $3$-sphere bundles over $S^4$ were classified up to
homotopy homeomorphism and diffeomorphism.

We now give an example which illustrates the subtleties of the inertia group.
Building on Examples \ref{ex:r0} and \ref{ex:r2}, Theorem \ref{thm:i_and_r} 
gives to the following

\begin{ex} \label{ex:r012}
The connected sums
\[ M_0 : = M_{-8, 0} \sharp M_{0, 8}, \quad M_1 : = M_{-8, 2} \sharp M_{0, 8} 
\quad \text{and} \quad M_2 : = M_{-8, 4} \sharp M_{0, 8}, \]
have $r(M_i) = i$.  In each case $d_\pi(M_i) = 8$, whereas
$\mdiv(M_0) = 8$, $\mdiv(M_1) = 2$ and $\mdiv(M_2) = 4$.  From Theorem \ref{thm:i_and_r}
we have $I(M_0) \cong I(M_1) \cong \Theta_7$ and $I(M_2) \cong 2\Theta_7$.

Notice that when $r = 1$ the \cite[Conjecture p.\,548]{wilkens74} correctly predicts  $I(M_1) = \Theta_7$.
However when $r \neq 1$, \cite[Conjecture p.\,548]{wilkens74} incorrectly predicts that $I(M_0)$ is $2 \Theta_7$ and
that $I(M_2)$ is $\Theta_7$.
\end{ex}

\begin{ex} \label{ex:pre_exotic_G2}
While \cite[Theorem 1]{wilkens74} and Theorem \ref{thm:i_and_r} give 
$I(M(\Z^b, d)) = \NumB{\frac{d}{4}}\Theta_7$, 
the classical Eells-Kuiper invariant is not defined for $M(\Z^b, d)$ when $d_\pi = d \neq 0$.
Using \eqref{eq:grchat} we compute that
\[ \mu \bigl( M(\Z^b, d) \sharp \Sigma \bigr) = [\mu(\Sigma)] \in \Z/\dM \Z . \]
Hence we have $\mu(M(\Z^b, 8)) = 0$, whereas
$\mu(M(\Z^b, 8) \sharp \Sigma_{\rm Mi}) = 1 \in \Z/2\Z$
and we see that the generalised Eells-Kuiper 
invariant distinguishes the diffeomorphism types of $M(\Z^b, 8)$ and
$M(\Z^b, 8) \sharp \Sigma_{\rm Mi}$.  

We can also deduce from Theorem \ref{thm:classification_categorical}, for example, 
that $M(\Z^b, 8) \sharp \Sigma_{\rm Mi}$
admits an orientation reversing diffeomorphism, whereas 
$M(\Z^b, 16) \sharp \Sigma_{\rm Mi}$ does not.
\end{ex}

\begin{ex} \label{ex:circle_actions}
Let $N$ be a simply-connected oriented $6$-manifold with $\pi_2(N) \cong \Z$
and suppose that $S^1 \to M \to N$ is a principal $S^1$ bundle with primitive
first Chern class. Then $M$ is $2$-connected with a preferred orientation and
hence spin structure.  Conversely, by \mbox{\cite[Lemma~2.1]{jiang14}},
every free
$S^1$ action on $M$ is equivalent to such a principal bundle action.

In \cite[Theorem 1.3]{jiang14} Yi Jiang identifies the homeomorphism %
and diffeomorphism types of all $2$-connected $M$ which admit free circle
actions.
In particular, by \mbox{\cite[Theorem 1.3]{jiang14}} every such 
$M$ is almost diffeomorphic to a connected sum
$M_{bk, b(k+12m)} \sharp_{2r}M_{0, 0}$ for $b \in \{1, 2\}$,
$r \in \Z^{\geq 0}$ and $m, k \in \Z$.
\end{ex}
\subsection{Examples from geometry} \label{ss:geometry}
There are a many $2$-connected \mbox{$7$-manifolds} that admit metrics
with interesting geometric properties. 
Indeed, according to \cite[Theorem~B]{crowley14b},
every $2$-connected $7$-manifold admits a metric with positive Ricci curvature.

\subsubsection{The Gromoll-Meyer sphere}
Let $Sp(n)$ denote the $n$-dimensional
symplectic group of ortho\-gonal $n \times n$ quaternionic matrices.
The Gromoll-Meyer sphere is a certain quotient of $Sp(2) \times Sp(1)$ by
$Sp(1) \times Sp(1)$ and the smooth manifold underlying the Gromoll-Meyer
sphere, $\Sigma_{GM}$, is an exotic $7$-sphere admitting a metric of
non-negative sectional curvature~\cite{gromoll74}.
By \cite[Theorem 1]{gromoll74},
there are diffeomorphisms $\Sigma_{GM} \cong M_{-1, -5} \cong 3 \Sigma_{\rm Mi}$.

\subsubsection{Berger Space}
The smooth manifold underlying the Berger space $B$ is a homogeneous
space of the form $B = SO(5)/SO(3)$ (where $SO(3) \into SO(5)$ by the adjoint
representation) that admits a metric of positive sectional curvature.
The Berger space is 2-connected with $H^4(B) \cong \Z/10$
and Goette, Kitchloo and Shankar \cite[Corollary 2]{goette04} proved there is a 
diffeomorphism
\[ B \cong M_{10, 8}. \]

\subsubsection{The manifold \texorpdfstring{$P_2$}{P\_2}}
More recently Grove, Verdiani and Ziller \cite[Theorem A]{grove11} constructed
a metric of positive sectional curvature  on a 2-connected $7$-manifold $P_2$ with
an isomorphism $H^4(P_2) \cong \Z/2$. Applying \mbox{\cite[Theorem A]{crowley02}},
they deduced that there is an almost diffeo\-morphism $P_2 \acong S(TS^4)$, where $S(TS^4) = M_{2, 0}$ 
is the unit tangent sphere bundle of $S^4$.  
In~\cite[Theorem 0.3, Example 3.12]{goette11} Goette proved that there are diffeomorphisms
\[ P_2 \cong M_{2, 2} \sharp (-\Sigma_{\rm Mi})
\quad \text{and} \quad
P_2 \cong -M_{2,4}.
\]
Computation shows that $P_2$ is not
orientation preserving diffeomorphic to the total space
of any $S^3$-bundle over $S^4$.

\subsubsection{\texorpdfstring{$G_2$}{G\_2}-manifolds}
In \cite{g2m} Corti, Haskins, the second author and Pacini constructed 
a very large class of examples of simply connected manifolds with $G_2$ holonomy
metrics. Many of these examples are $2$-connected with $H^4(M)$ torsion-free.
For instance, \cite[Table 3]{g2m} gives 7 explicit ways to
construct holonomy $G_2$ metrics on
$M(\Z^{85}, 2)$.
By \cite[Theorem 1(ii)]{wilkens74},
see also Corollary \ref{cor:n_+}, the underlying topological manifold
admits a unique smooth structure.
In \cite{exotic} we find examples of manifolds with $G_2$ holonomy
where the smooth structure is not unique and calculating
the Generalised Eells-Kuiper invariant we find 
pairs of closed \gtmfd s that are homeomorphic but not diffeomorphic.
For example $(M(\Z^{89}, 8), M(\Z^{89}, 8) \sharp \Sigma_{\mathrm{Mi}})$
is a pair of homeomorphic but not diffeomorphic smooth manifolds 
both of which admit metrics with $G_2$ holonomy.

\subsection{Tangentially homotopy equivalent manifolds} \label{ss:the}
Let $N_0$ and $N_1$ be closed smooth manifolds 
with tangent bundles $TN_0$ and $TN_1$.  
A homotopy equivalence $f \colon N_0 \to N_1$ is called {\em tangential}
if there is a bundle isomorphism $f^*TN_1 \cong TN_0$.
It is natural to ask under what conditions tangentially homotopy equivalent manifolds
are necessarily homeomorphic, and this question was studied in detail 
by Madsen, Taylor and Williams in \cite{madsen80}.

In \cite[p.\,144]{crowley02} it was proven the 
$2$-connected manifolds give rise to examples of
non-homeomorphic tangentially homotopy equivalent manifolds.  We present a simplified
version of the proof here, which starts with the following

\begin{lem} \label{lem:the1}
Let $M_0$ and $M_1$ be $2$-connected and 
let $f \colon M_0 \simeq M_1$ be a homotopy equivalence such that 
$f^* p_{M_1} = p_{M_0}$.
Then $f$ is tangential.
\end{lem}

\begin{proof}
The proof is a relative version of Remark \ref{rem:pM_and_TM}.
The bundles $TM_0$ and $f^*TM_1$ are classified by maps 
$M_0 \to BSO(7)$.
Since $M_0$ is $2$-connected, the primary obstruction to a null-homotopy
between these maps may be identified with $p_{M_0} - f^*p_{M_1}$.
The computations of \cite{kervaire60} show that $\pi_i(BSO(7)) = 0$
for $i = 5, 6, 7$ and so there are no further obstructions to finding
a homotopy between the classifying maps of $TM_0$ and $f^*TM_1$.
Hence if $p_{M_0} = f^*p_{M_1}$ then $f^*TM_0 \cong TM_1$.
\end{proof}

\begin{prop}[\cf~{\cite[p.\,114]{crowley02}}] \label{prop:the}
The manifolds $M_{-8, 1}$ and $M_{-8, 5}$ are tangentially homotopy equivalent
but not homeomorphic.	
\end{prop}

\begin{proof}
We first show that $M_{-8, 2}$ and $M_{8, -10}$ are tangentially homotopy equivalent.
By Definition \ref{def:bbocf}, both manifolds have base
$\bigl( \Z/8, \anb{-1}{8}, \rho_8(2) \bigr)$
and applying \eqref{eq:alg_qdef},
we see that their quadratic refinements 
are respectively $\anq{-1}{16}_{\!-2}$ and $\anq{-1}{16}_{\!-10}$.
Now
$$\biganq{-1}{16}_{\!-10} = \left( \biganq{-1}{16}_{\!-2}\right)_{\rho_8(4)}$$
and $\rho_8(4) \in 12(\Z/8)$.
By Theorem \ref{thm:homotopy}, it follows that $M_{-8, 2}$ and $M_{-8, 5}$
are orientation preserving homotopy equivalent via a homotopy equivalence
$f \colon M_{-8, 2} \to M_{-8, 5}$, 
which is the identity with respect to the above bases. It follows that 
$f^*p_{M_{-8, 5}} = p_{M_{-8, 1}}$,
and so $f$ is tangential by Lemma \ref{lem:the1}.

Applying Proposition \ref{prop:arf}
we compute that 
$A(q_{M_0}) = -1/16$ mod $\Z$ but $A(q_{M_1}) = -9/16$ mod $\Z$ and
by Theorem~\ref{thm:a_class},
the quadratic refinement $q_M$ is a homeomorphism invariant and hence 
$M_{-8, 2}$ and $M_{-8, 10}$ are not homeomorphic.
\end{proof}

\begin{rmk} \label{rmk:the}
Proposition \ref{prop:the} contradicts \cite[Theorem C and Theorem 5.10]{madsen80} 
where it is stated, amongst other things, that all tangentially homotopy equivalent 
$2$-connected $7$-manifolds
are homeomorphic.  The source of the mistake in the arguments of \cite{madsen80} can be
found in \cite[Theorem 3.12]{madsen80} which is not correct.
It is claimed that a certain cohomology class 
\[ f^*\pi^*(l_n) \in H^{4n}(S^2\Omega^2(SG[3, \infty]); \Z_{(2)}) \]
vanishes.
Here $SG[3, \infty]$ is the $2$-connected cover of $SG$,
the space of orientation preserving stable self-homotopy equivalences of the sphere,
$S^2\Omega^2$ denotes the double suspension of the double loop space, 
the coefficient group $\Z_{(2)}$
is the integers localised at $2$ and we shall not define the maps $f$ or $\pi$ 
or the class $l_n$.
However the argument given for the proof
of \mbox{\cite[Theorem 3.12]{madsen80}} only shows that $f^*\pi^*(l_n) = 2 x$
for some $x \in H^{4n}(S^2\Omega^2(SG[3, \infty]); \Z_{(2)})$
and not that $f^*\pi^*(l_n) = 0$.
To the best of our knowledge, this is the only
flaw in the arguments of \cite{madsen80}.
\end{rmk}
\subsection{Generators for the monoid of 2-connected 7-manifolds} \label{ss:monoid}
The connected sum operation gives the set of spin diffeomorphism classes of 
$2$-connected $7$-manifolds the structure of a commutative monoid with unit $S^7$.
Owing to the existence of homotopy $7$-spheres, every $M$ has non-trivial connected sum splittings
\[ M \cong (M \sharp \Sigma) \sharp (-\Sigma) \]
for each $\Sigma \in \Theta_7$.  Hence we call $M$ {\em topologically decomposable} if
there is a diffeomorphism
\[ M \cong M_0 \sharp M_1 \]
where neither $M_0$ nor $M_1$ is a homotopy sphere and
{\em topologically indecomposable} otherwise.

By Theorem \ref{thm:classification_categorical}, every connected sum splitting of $M$
gives rise to an orthogonal splitting of the refinement of $M$,
and by Theorem \ref{thm:class} every orthogonal splitting of the
refinement of $M$ is realised by a connected sum splitting of $M$.
Hence we call a refinement $(G, q^\circ\!, p)$ or a base $(G, b, p)$ decomposable 
if it can be written as a non-trivial orthogonal sum and indecomposable otherwise.  
It is clear from the definitions that a refinement is indecomposable if and only if its base if indecomposable.  
Moreover, the indecomposable bases are of the form $(\Z, 0, p)$ and $(T, b, p)$ where $b$ is an
indecomposable torsion form; \ie~$b$ cannot be written as a non-trivial orthogonal sum. In this case 
we also call $q$ indecomposable.  A list of all isomorphism classes of indecomposable torsion forms 
was given by Wall \cite[Theorem 4]{wall63} and torsion forms were then classified by 
Kawauchi and Kojima \cite[Theorem 4.1]{kawauchi80}.  We do 
not go into details but note that if $(T, b, p)$ is indecomposable then $T \cong \Z/r^k$ for a prime $r$
or $T \cong (\Z/2^k)^2$.  Summarising the above discussion we have the following refinement of 
a theorem of Wilkens.

\begin{thm}[\cf~{\cite[Theorem 1]{wilkens72}}] \label{thm:indecomposable}
Every $2$-connected $M$ is diffeomorphic to a connected sum of topologically indecomposable 
manifolds $M_i:$
\[ M \cong \sharp_{i=1}^n M_i. \]
Moreover, $M$ is topologically indecomposable if and only if it is almost diffeomorphic to 
a manifold of one of the following forms
\[ S^7, \quad M(\Z, d), \quad M(q, s), \]
where in the final case, $q$ is a prime refinement and hence $H^4(M(q, s)) \cong \Z/r^k$ for
$r$ a prime or $H^4(M(q, s)) \cong (\Z/2^k)^2$. \qed
\end{thm}

Even up to almost diffeomorphism, the splitting $M \cong \sharp_{i=1}^n M_i$ 
of Theorem \ref{thm:indecomposable} is in general far from being unique.  For
example the manifolds $M_{4, 0}$ and $M_{4, 2}$ have non-isomorphic bases 
$(\Z/4, \an{\frac{-1}{4}}, 0)$ and $(\Z/4, \an{\frac{-1}{4}}, 2)$ respectively, but 
$M_{4, 0} \sharp M_{0, 2}$ and $M_{4, 2} \sharp M_{0, 2}$ are diffeomorphic.
Even when $H^4$ is torsion, there are many examples of torsion forms
where $b_0 \oplus b_2 \cong b_1 \oplus b_3$ but $b_0$ is isomorphic to 
neither $b_1$ nor $b_3$ and the same holds for $b_2$: see for example \cite[\S 3]{kawauchi80}.  
This leads to non-uniqueness
of connected sum splittings for manifolds with torsion linking form isomorphic to
$b_0 \oplus b_2$.
\section{Mapping class groups and inertia} \label{sec:mcg_and_inertia}
In this section we point out some
implications of our classification results for mapping class groups of $2$-connected $M$.  
Throughout this section $M$ will be $2$-connected. 

Recall the mapping class group $\wt \pi_0\Diff(M)$ 
of pseudo-isotopy classes of diffeomorphisms
from Section \ref{ss:almost_diffeomorphism} and the subgroup
$\wt \pi_0\Diff_H(M) \subseteq \wt \pi_0\Diff(M)$ of classes acting trivially on $H^*(M)$
from Section \ref{ss:inertia_and_reactivity2}.
For brevity, let $\aut_\EK(H^4(M))$ denote
the group of automorphisms of the mod 28 distillation
$(H^4(M), q_M^\circ, \EK_M, p_M)$ and let
$\aut_{q^\circ}(H^4(M))$ denote the group of automorphisms of the refinement
$(H^4(M), q_M^\circ, p_M)$.  As an immediate consequence of 
Theorem \ref{thm:class} we obtain
\begin{prop} \label{prop:mcg1}
For each $2$-connected $M$, there is a short exact sequence
\[
0 \to \wt \pi_0 \Diff_H(M) \to \wt \pi_0\Diff_{}(M) \to \aut_\EK(H^4(M)) \to 0. 
\qed
\] 
\end{prop}
\begin{rmk} \label{rmk:mcg}
The exact sequence of Proposition \ref{prop:mcg1} serves as a starting point
for studying the mapping class groups $\wt \pi_0\Diff_{}(M)$.  
The determination of 
$\wt \pi_0 \Diff_H(M)$ and the extension in the sequence of Proposition \ref{prop:mcg1}
lie outside the scope of this paper.
For example, we do not currently know if $\wt \pi_0 \Diff_H(M)$ is abelian in general.
\end{rmk}
  
Recall the mapping class groups $\wt \pi_0 \ADiff_H(M, m_0) \subseteq \wt \pi_0 \ADiff(M, m_0)$ 
defined in Sections \ref{ss:almost_diffeomorphism} and \ref{ss:inertia_and_reactivity2} and
the homomorphism $\wt P \colon \aut_{q^\circ}(H^4(M)) \to \lcm(8, 2^r\mdiv)\Z/2 \tilde d_\pi \Z$ from 
Section \ref{ss:aut_q_on_gauss_functions}; see \eqref{eq:P_tilde} and Proposition \ref{prop:imtp}.
Noting that $\dM = \gcd(\frac{\tilde d_\pi}{4}, 28)$, we define
\[ \wh P \colon \aut_{q^\circ}(H^4(M)) \to \NumB{2^{r-3}\mdiv}\Z/\dM \Z, \quad F \mapsto \frac{\wt P(F)}{8} \mod \dM, \]
to be the mod $\dM$ reduction of $\wt P$ divided by $8$.  
By Theorem \ref{thm:i_and_r} and
Remark \ref{rmk:p2null} we have $I(M)/I_H(M) \cong \NumB{2^{r-3}\mdiv}\Z/\dM \Z$ and so we can equally regard
$\wh P$ as a homomorphism $\wh P \colon \aut_{q^\circ}(H^4(M)) \to I(M)/I_H(M)$.

\begin{thm} \label{thm:mcg_all}
For each $2$-connected $M$ there is a commutative diagram of group homomorphisms with short exact sequences for rows
and with exact columns:
\[
\xymatrix@R=1.75em{  0 \ar[r] & \wt \pi_0\Diff_H(M) \ar[d] \ar[r] & \wt \pi_0 \Diff_{}(M) \ar[d] \ar[r] & 
\aut_\EK(H^4(M)) \ar[d] \ar[r] & 0 \\
0 \ar[r] & \wt \pi_0 \ADiff_H(M, m_0) \ar[d]^(0.45){\del_H} \ar[r] & 
\wt \pi_0 \ADiff_{}(M, m_0) \ar[d]^(0.45){\del} \ar[r] & 
\aut_{q^\circ}(H^4(M)) \ar[d]^(0.45){\wh P_{}} \ar[r] & 0 \\
0 \ar[r] & I_H(M) \ar[d] \ar[r] & I(M) \ar[d] \ar[r] & I(M)/I_H(M) \ar[d]  \ar[r] & 0 \\
& 0 & 0 & 0 } 
\]
In particular, an automorphism 
$F \in \aut_{q^\circ}(H^4(M))$ is realised by a diffeomorphism of $M$ if and only if $\wh P(F) = 0$.
\end{thm}

\begin{proof}
The top row is the exact sequence of Proposition \ref{prop:mcg1}.
The exactness of the second row follows from Theorem \ref{thm:a_class}.
The first two columns are exact by the discussion at the beginning of Section \ref{ss:inertia_and_reactivity2} and in particular \eqref{eq:I_and_del} 
and the third column is exact by the definition of $\wh P_{}$.
The only part of the commutativity of the diagram which needs comment is the bottom right hand square,
where the commutativity follows from from Lemma \ref{lem:mapping_torus}~\ref{lem:mapping_torus:mu}
and Proposition~\ref{prop:p2mod}.  The final statement follows from the exactness of final row
and the top column.
\end{proof}

We shall call an almost diffeomorphism {\em exotic} if it is not 
pseudo-isotopic to a diffeomorphism.
A feature of the diagram in Proposition \ref{thm:mcg_all} is that when $I(M)/I_H(M) \neq 0$, $M$ admits
exotic almost diffeomorphisms which are detected by their action on $H^4(M)$.  
Specifically, if $f \colon M \acong M$ is an almost diffeomorphism,
then $\wh P_{}(f^*)$ is the obstruction to $f^* \colon H^4(M) \cong H^4(M)$ being induced by {\em any diffeomorphism}
of $M$.  Since $\wh P$ is onto, it is enough to find cases where $I(M)/I_H(M)$ is non-zero to
show that $\wh P$ is non-zero.

\begin{prop} \label{prop:pairs_of_inertia_groups}
Any pair of subgroups $I_0 \subseteq I_1 \subseteq \Theta_7$ can arise as the pair of inertia
groups $(I_0, I_1) = (I_H(M), I(M))$ for some $2$-connected $M$.
\end{prop}

\begin{proof}
There are three pairs of subgroups $(T_0, T_1)$ in $\Z/7$ and six pairs of subgroups $(T_0, T_1)$ in $\Z/4$,
leaving 18 cases to realise.
By Theorem \ref{thm:i_and_r} and Remark \ref{rmk:p2null}, $I(M)$ and $I_H(M)$ depend only on the
base $(G, b, p)$.  We list manifolds, their bases $(G, b, p)$ 
and the pairs $(I_H, I)$ of inertia groups they realise
in the following table, where it is helpful to note that $112 = 7 \times 16$:
\vspace{0pt plus \baselineskip}
\[ \begin{array}[b]{c@{\hspace{10mm}}rcc@{\hspace{10mm}}cc@{\hspace{10mm}}c}
\toprule
M  & G\hspace{15pt} & b & p  & I_H & I & I/I_H \\ \midrule
S^7& \{0\} \hspace{10.5pt}  & 0 & 0 & 0 & 0 & 0 \\
M_{0, 56}& \Z\hspace{15pt}  & 0 & 56 & \Z/2 & \Z/2 & 0 \\
M_{0, 28}& \Z\hspace{15pt}  & 0 & 28 & \Z/4 & \Z/4 & 0 \\
M_{0, 16}& \Z\hspace{15pt}  & 0 & 16 & \Z/7 & \Z/7 & 0 \\
M_{0, 8}& \Z\hspace{15pt}  & 0 & 8 & \Z/14 & \Z/14 & 0 \\
M_{0, 4}& \Z\hspace{15pt}  & 0 & 4 & \Z/28 & \Z/28 & 0 \\
M_{-16, 0} \sharp M_{0, 112}  & \Z/16 \oplus \Z & \an{1/16} & (0, 112)  & 0 & \Z/2 & \Z/2 \\
M_{-8,0 } \sharp M_{0, 56} & \Z/8 \oplus \Z & \an{1/8} & (0, 56)  & \Z/2 & \Z/4 & \Z/2 \\
M_{-16, 2} \sharp M_{0, 112} & \Z/16 \oplus \Z & \an{1/16} & (2, 112)  & 0 & \Z/4 & \Z/4 \\
M_{-16, 0} \sharp M_{0, 16} & \Z/16 \oplus \Z & \an{1/16} & (0, 16)  & \Z/7 & \Z/14 & \Z/2 \\
M_{-8, 0} \sharp M_{0, 8} & \Z/8 \oplus \Z & \an{1/8} & (0, 8)  & \Z/14 & \Z/28 & \Z/2 \\
M_{-16, 2} \sharp M_{0, 16} & \Z/16 \oplus \Z & \an{1/16} & (2, 16)  & \Z/7 & \Z/28 & \Z/4 \\
M_{-7, 1} \sharp M_{0, 112} & \Z/7 \oplus \Z & \an{1/7} & (1, 112)  & 0 & \Z/7 & \Z/7 \\
M_{-7, 1} \sharp M_{0, 56} & \Z/7 \oplus \Z & \an{1/7} & (1, 56)  & \Z/2 & \Z/14 & \Z/7 \\
M_{-0, 14} \sharp M_{0, 14} & \Z/7 \oplus \Z & \an{1/7} & (1, 14)  & \Z/4 & \Z/28 & \Z/7 \\
M_{-7, 1} \sharp M_{112, 16} & \Z/112 \oplus \Z & \an{1/112} & (16, 112)  & 0 & \Z/14 & \Z/14 \\
M_{-56, 0} \sharp M_{0, 56} & \Z/56 \oplus \Z & \an{1/56} & (8, 56)  & \Z/2 & \Z/28 & \Z/14\\
M_{-112, 2} \sharp M_{0, 112} & \Z/112 \oplus \Z & \an{1/112} & (2, 112)  & 0 & \Z/28 & \Z/28\\
\bottomrule
\end{array} \qedhere \]
\vskip -2pt plus 4pt
%
\end{proof}

Theorem \ref{thm:hatP_short} follows immediately from Theorem \ref{thm:mcg_all} and
Proposition \ref{prop:pairs_of_inertia_groups}.  We conclude with an example drawn
from the bottom line of the table above.

\begin{ex} \label{ex:mapping_inertia}
Let $M = M_{-112, 2} \sharp M_{0, 112}$ so that $H^4(M) = \Z/112 \oplus \Z$ and consider
the automorphism of $(H^4(M), q^\circ_M, p_M)$ defined by
\[ F = \begin{pmatrix} 1 & [1] \\ 0 & 1 \end{pmatrix} \colon \Z/112 \oplus \Z \cong \Z/112 \oplus \Z. \]
\vskip -0.25cm
\noindent
In this case $\dM = 28$ and from the proof of Proposition \ref{prop:pairs_of_inertia_groups} we see that
$M$ admits an almost diffeomorphism $f \colon M \acong M$ with $f^* = F$, and $\wh P(f) = 1 \in \Z/28 \Z$.
By Theorem \ref{thm:mcg_all}, $F^n$ is realised by a diffeomorphism of $M$ if and only  if $n \equiv 0$~mod~28.
\end{ex}
%
%
%


\bibliographystyle{amsinitial}
\bibliography{g2geom}

\begin{multicols}{2}[]
\bigskip
\noindent
\emph{Diarmuid Crowley}\\
  \vskip -0.125in \noindent
  {\small
  \begin{tabular}{l}%
    School of Mathematics and Statistics\\
    University of Melbourne\\
    Parkville, VIC, 3010, Australia\\
    \textsf{dcrowley@unimelb.edu.au}
  \end{tabular}}

\noindent
\emph{Johannes Nordstr\"{o}m}\\  
\vskip -0.125in \noindent
{\small
  \begin{tabular}{l}%
   Department of Mathematical Sciences\\
   University of Bath\\
   Claverton Down, Bath BA2 7AY, UK\\
    \textsf{j.nordstrom@bath.ac.uk}
  \end{tabular}}
\end{multicols}

\end{document}